\documentclass[10pt]{amsart}

\usepackage{graphicx}        
\usepackage{multicol}        
\usepackage[bottom]{footmisc}
\usepackage{amscd,amsmath,amssymb,amsfonts}
\usepackage[cmtip, all]{xy}

\usepackage{hyperref}

\usepackage{bbm}
\usepackage{tikz-cd}
\usepackage{enumerate}
\usepackage{mathrsfs}
\usepackage{amsthm}
\usepackage{comment}

\usepackage{calligra}
\usepackage[T1]{fontenc}

\newtheorem{theorem}{Theorem}[section] 
\newtheorem{proposition}[theorem]{Proposition}
\newtheorem{lemma}[theorem]{Lemma}

\newtheorem{corollary}[theorem]{Corollary}

\newtheorem{IntroThm}{Theorem}
\newtheorem*{IntroCor}{Corollary}

\theoremstyle{definition}
\newtheorem{definition}[theorem]{Definition}
\theoremstyle{remark}
\newtheorem*{IntroRemark}{Remark}
\newtheorem{remark}[theorem]{Remark}
\newtheorem{remarks}[theorem]{Remarks}
\newtheorem{ex}[theorem]{Example}

\newtheorem{Not}[theorem]{Notation}
\newtheorem*{IntroNotConv}{Notation and Conventions}
\numberwithin{equation}{section}

\newcommand{\cl}{{\rm cl}}

\newcommand{\CH}{{\operatorname{CH}}}

\newcommand{\red}{{\rm red}}
\newcommand{\codim}{{\rm codim}}

\newcommand{\Pic}{{\rm Pic}}

\newcommand{\Hom}{{\rm Hom}}

\newcommand{\Spec}{{\rm Spec\,}}
\newcommand{\sing}{{\rm sing}}

\newcommand{\supp}{{\rm supp}\,}
\newcommand{\0}{\emptyset}

\newcommand{\sE}{{\mathcal E}}
\newcommand{\sF}{{\mathcal F}}

\newcommand{\sH}{{\mathcal H}}
\newcommand{\sI}{{\mathcal I}}

\newcommand{\sK}{{\mathcal K}}
\newcommand{\sL}{{\mathcal L}}

\newcommand{\sO}{{\mathcal O}}
\newcommand{\sP}{{\mathcal P}}

\newcommand{\sV}{{\mathcal V}}

\newcommand{\sX}{{\mathcal X}}
\newcommand{\sY}{{\mathcal Y}}
\newcommand{\sZ}{{\mathcal Z}}
\newcommand{\A}{{\mathbb A}}

\newcommand{\F}{{\mathbb F}}

\newcommand{\N}{{\mathbb N}}
\renewcommand{\P}{{\mathbb P}}
\newcommand{\Q}{{\mathbb Q}}

\newcommand{\V}{{\mathbb V}}

\newcommand{\Z}{{\mathbb Z}}

\renewcommand{\det}{\operatorname{det}}
\newcommand{\Nm}{{\operatorname{Nm}}}

\newcommand{\id}{{\operatorname{\rm Id}}}

\newcommand{\Sch}{{\operatorname{\mathbf{Sch}}}} 
\newcommand{\colim}{{\mathop{\rm colim}}}

\newcommand{\op}{{\text{\rm op}}}
\newcommand{\<}{\langle}
\renewcommand{\>}{\rangle}

\renewcommand{\dim}{{\operatorname{\rm dim}}} 
\newcommand{\Div}{{\operatorname{div}}} 

\newcommand{\Coh}{\operatorname{Coh}}

\newcommand{\CM}{{\operatorname{C-M}}}

\newcommand{\reg}{{\operatorname{reg}}}

\newcommand{\Sets}{{\mathbf{Sets}}}

\renewcommand{\max}{{\operatorname{\rm max}}}

\newcommand{\Ab}{{\mathbf{Ab}}}

\newcommand{\Sym}{{\operatorname{Sym}}}

\newcommand{\Fl}{{\sF{l}}}
\newcommand{\lci}{{\operatorname{lci}}}
\newcommand{\nlci}{{\operatorname{nlci}}}

\newcommand{\Tor}{{\operatorname{\rm Tor}}}

\newcommand{\Gr}{{\operatorname{\rm Gr}}}

\newcommand{\topo}{{\operatorname{top}}}
\newcommand{\coeq}{{\operatorname{coeq}}}

\newcommand{\GL}{\operatorname{GL}}
\newcommand{\SL}{\operatorname{SL}}
\newcommand{\PGL}{\operatorname{PGL}}

\newcommand{\et}{\text{\'et}}

\newcommand{\ind}[1]{}

\newcommand{\inp}[1]{}

\newcommand{\gr}{{\operatorname{gr}}}
\newcommand{\Tot}{{\operatorname{Tot}}}

\newcommand{\Ring}{{\operatorname{\bf Ring}}}

\newcommand{\Vect}{{\operatorname{\bf Vec}}}
\newcommand{\ses}{{\operatorname{\bf ses}}}

\newcommand{\Fun}{{\operatorname{Fun}}}

\newcommand{\gl}{{\operatorname{gl}}}

\newcommand{\qp}{{\operatorname{\text{\rm qp}}}}
\newcommand{\con}{{\operatorname{\text{\rm con}}}}

\newcommand{\plu}{\operatorname{Pl}}
\newcommand{\ver}{\operatorname{Ver}}

\newcommand{\gp}{{\operatorname{gp}}}
\newcommand{\mot}{{\operatorname{mot}}}

\newcommand{\sm}{{\operatorname{sm}}}

\makeatletter
 \@namedef{subjclassname@2020}{\textup{2020} Mathematics Subject Classification}
\makeatother

\date{ \today}

\author[M.~Levine]{Marc~Levine}
\address{Marc~Levine, Universit\"at Duisburg-Essen,
Fakult\"at Mathematik, Campus Essen, 45117 Essen, Germany}
\email{marc.levine@uni-due.de}

\title{Riemann-Roch for 0-cycles on a singular variety}


\begin{document}

\begin{abstract}  Let $X$ be a quasi-projective scheme of dimension $d$ over an infinite field $k$, such that $X$ is reduced after removing all components of dimension $<d$, and let $X^*\subset X$ be a closed subset of dimension $<d$ such that $X\setminus X^*$ is regular of dimension $d$. Using a modification $\CH^d(X, X^*)$ of the Chow group of 0-cycles on a singular variety $X$ defined in \cite{LW}, we construct a Chern class map $c_d:K_0(X)\to \CH^d(X, X^*)$, a cycle class map from $\CH^d(X, X^*)$ onto a subgroup $F^{\dim_kX}K_0(X)$ of $K_0(X)$,  and prove a Riemann-Roch theorem computing the two compositions of these two maps. 
\end{abstract}

\maketitle

\tableofcontents

\section*{Introduction} For a regular quasi-projective $k$-scheme $X$ we have the Grothendieck group of locally free coherent sheaves $K_0(X)$ and the Chow groups $\CH^n(X)$ of codimension $n$ algebraic cycles modulo rational equivalence. Since $X$ is regular, we have the topological filtration $F^*_\topo K_0(X)$ and the cycle class map
\[
\cl^n:\CH^n(X)\to \gr^n_\topo K_0(X),\ n=0,\ldots, \dim_kX.
\]
If $X$ is smooth over $k$, we have Grothendieck's construction of the Chern class maps
\[
c_n:K_0(X)\to  \CH^n(X),
\]
and the Grothendieck-Riemann-Roch theorem without denominators gives the fundamental identities relating $\CH^n(X)$ and $\gr^n_\topo K_0(X)$, namely: $c_n$ descends to a group homomorphism
\[
c_n:\gr^n_\topo K_0(X)\to \CH^n(X)
\]
with
\[
\cl^n\circ c_n=(-1)^{n-1}(n-1)!\id_{\gr^n_\topo K_0(X)},\ 
c_n\circ \cl^n=(-1)^{n-1}(n-1)!\id_{\CH^n(X)}
\]
This is classical, going  back to Grothendieck \cite{Grothendieck} and Jouanolou \cite[\S1, (F)]{Jouanolou}; for a more modern treatment we refer the reader to Fulton's book \cite{Fulton}.

Suppose now that $X$ is no longer smooth or even regular. One can still consider the Chow groups of dimension $n$ cycles modulo rational equivalence $\CH_n(X)$ and the Grothendieck group of coherent sheaves on $X$, $G_0(X)$ with its topological filtration $F^\topo_*G_0(X)$ (now by dimension),   one still has the surjective cycle class map
\[
\cl_n: \CH_n(X)\to \gr^\topo_n G_0(X),
\]
and the Baum-Fulton-MacPherson Riemann-Roch theorem for singular varieties \cite[Theorem, pg. 103]{BFM} says that $\cl_n$ is an isomorphism after $\otimes_\Z\Q$. 

However, for singular $X$, the natural map $K_0(X)\to G_0(X)$ is no longer an isomorphism and the cycle class map $\cl_n$ does not have an evident lifting to a map to a suitable graded object for $K_0(X)$, even for the grading for the $\gamma$-filtration. Indeed, for $Z\subset X$ an integral closed subscheme, the structure sheaf $\sO_Z$ will in general not admit a finite resolution by locally free coherent sheaves, so it is not clear if one can define a reasonable map $Z_n(X)\to K_0(X)$, much less ask such a map to descend to a map from $\CH_n(X)$ to some suitable graded object for $K_0(X)$.

Using the sheaf of $K$-groups $\sK_n$, and building on Bloch's formula (proven by Quillen \cite[\S 7, Theorem 5.19]{Quillen}) for a regular scheme $X$ of finite type over a field
\[
\CH^n(X)\cong H^n(X, \sK_n),
\]
one could use $H^n(X, \sK_n)$ as a definition of $\CH^n(X)$ for singular $X$. Alternatively, Kerz \cite{Kerz}, following partial results by Kato \cite{Kato} and  Elbaz-Vincent and M\"uller-Stach \cite{EBMS}, uses  the sheaf of Milnor $K$-groups $\sK^M_n$ instead of $\sK_n$ and proves a version of Bloch's formula,  for smooth $X$
\[
\CH^n(X)\cong H^n(X, \sK^M_n),
\]
so one could also define  $\CH^n(X):=H^n(X, \sK^M_n)$ for singular $X$. 

Motivic cohomology is also related to the Chow ring for smooth, quasi-projective $X$, via the special case of the isomorphism due to Voevodsky-Suslin-Friedlander \cite[Proposition 4.2.9, Theorem 4.3.7]{VSF}
\[
H^p(X, \Z(q))\cong \CH^q(X, 2q-p),
\]
where $\CH^q(X, 2q-p)$ is Bloch's higher Chow group \cite{Bloch}, giving for $(p,q)=(2n,n)$ the isomorphism
\[
H^{2n}(X, \Z(n))\cong \CH^n(X, 0)=\CH^n(X).
\]
 In this direction, Elmanto and Morrow \cite{EM} have constructed a non-$\A^1$-invariant version of motivic cohomology, $H^*_\mot(X, \Z(*))$, with a close relation to algebraic $K$-theory $K_*(X)$, and agreeing with the Voevodsky-Suslin-Friedlander motivic cohomology for smooth $X$ over a field. For $X$ reduced, and quasi-projective over a perfect field, they construct an isomorphism\footnote{Private communication}
\[
H^{2n}_\mot(X, \Z(n))\cong H^n(X, \sK^M_n).
\]

However, there is only a weak connection of  $H^{2n}_\mot(X, \Z(n))$ or $H^n(X, \sK^M_n)$ with algebraic cycles on a singular $X$. I had made an attempt to construct a Chow ring $\CH^*(X, X_\sing)$ for reduced, quasi-projective $X$ over a field in the unpublished preprint \cite{L85}, by mixing together cycles built out of integral closed subvarieties that are supported in the regular locus $X_\reg$ of $X$, together with a kind of ``pullback by a general translate'' of codimension $n$ subvarieties $Z\subset H$, with $H$ a homogeneous space for a linear algebraic group $G$, and $f:X\to H$ a morphism. This uses the fact that given $f$ and $Z$, there is an open subset $U(f, Z)\subset G$ such that for all $g\in U(f,Z)(k)$, $f^{-1}(g\cdot Z)$ has pure codimension $n$ on $X$,  intersects the singular locus $X_\sing$ in pure codimension $n$, and any relevant Tor-groups  vanish. One could then hope that the cycle $|f^{-1}(gZ)|$ associated to the closed subscheme $f^{-1}(gZ)$ would be a reasonable generator for a group $\CH^n(X, X_\sing)$, in addition to those codimension $n$ cycles supported in $X_\reg$. 


For cycles of  codimension $d=\dim_kX$, such $|f^{-1}(gZ)|$ would be a 0-cycle supported in the regular locus of $X$, so one only expects such 0-cycles to play a role. Indeed, the above  approach was inspired by the construction of a reasonable group of 0-cycles modulo rational equivalence relative to a closed subset $Y\subset X$ with $X\setminus Y$ regular and dense in $X$, $\CH_0(X, Y)$, constructed in a joint work with Weibel \cite{LW}.  Here, $\CH_0(X, Y)$ is the quotient of the group $Z_0(X\setminus Y)$ of 0-cycles supported in $X\setminus Y$, modulo relations $R_0(X,Y)$ given by divisors of suitable rational functions on ``Cartier curves on $X$ relative to $Y$'' (see Definition~\ref{def:CartierCurve}  below). This choice of generators and relations is motivated by the fact that each closed point $x\in X\setminus Y$ has residue field of finite homological dimension over $\sO_X$, giving a well-defined class $\cl_0(x)=[k(x)]\in K_0(X)$,  and the relations are chosen so that the map $\cl_0:Z_0(X\setminus Y)\to K_0(X)$ descends to a well-defined group homomorphism
\[
\cl_0:\CH_0(X, Y)\to K_0(X).
\]
See \cite[Proposition 2.1]{LW}.

In this paper, we consider a quasi-projective scheme $X$ over an infinite field $k$. We let $d$ denote the maximum of the dimensions over $k$ of the integral components of $X$ (we write $d=\dim_kX$), let $X(i)\subset X$ be the union of the integral, dimension $i$ components of $X$ and let $X(d)^0:=X\setminus\cup_{i=0}^{d-1}X(i)$, an open subscheme of $X$.  We also assume that $X(d)^0$ is reduced.  Let $X_\reg(d)$ denote  the maximal open subscheme of $X$ that is regular and of dimension $d$ over $k$; since $X(d)^0$ is reduced,   $X_\reg(d)\subset X(d)^0$ is open and dense in $X(d)$.

Let $X^*\subset X$ be a closed subset containing $\cup_{i<d}X(i)$ such that $X\setminus X^*$ is a dense open subscheme of $X_\reg(d)$. We define a group $\CH^d(X, X^*)$ of codimension $d$ algebraic cycles on $X$, with generators $Z^d(X\setminus X^*)$ and relations $R^d(X, X^*)$ given by divisors of suitable rational functions on Cartier curves relative to $X^*$.  

We let $F^dK_0(X)$ denote the subgroup of $K_0(X)$ generated by the classes of the $\sO_X$-modules $k(x)$ as $x$ runs over the closed points of $X_\reg(d)$. Similar to the construction of $\cl_0$, there is a well-defined  group homomorphism
\[
\cl^d:\CH^d(X, X^*)\to F^dK_0(X),
\]
sending $x\in X\setminus X^*$ to the class $[k(x)]\in F^dK_0(X)$.

Our main achievement here is to define a Chern class map (of pointed sets)
\[
c_d:K_0(X)\to \CH^d(X, X^*),
\]
and prove the following result.

\begin{IntroThm}\label{IntroThm} Let $k$ be an infinite field, and let $X$ be a quasi-projective $k$-scheme. Let $d=\dim_kX$ and suppose that  $X(d)^0$ is reduced.   Let $X^*$ be a  closed subset of $X$ such that $X\setminus X^*$ is a dense open subset of  $X_\reg(d)$. Then 
\begin{enumerate}
\item The restriction of $c_d: K_0(X)\to \CH^d(X, X^*)$ to $F^dK_0(X)$ is a group homomorphism
\[
c_d:F^dK_0(X)\to \CH^d(X, X^*).
\]
\item The cycle class map $\cl^d:\CH^d(X, X^*)\to F^dK_0(X)$ is surjective.
\item We have
\[
\cl^d\circ c_d=(-1)^{d-1}(d-1)!\id_{F^dK_0(X)}
\]
and
\[
c_d\circ \cl^d=(-1)^{d-1}(d-1)!\id_{ \CH^d(X, X^*)}.
\]
\end{enumerate}
\end{IntroThm}
As immediate corollary, we have

\begin{IntroCor}\label{IntroCor} Let $k$ be an infinite field, and let $X$ be a  quasi-projective $k$-scheme. Let $d=\dim_kX$ and suppose that  $X(d)^0$ is reduced.  Let $X^*$ be a closed subset with $X\setminus X^*$ a dense open subset of $X_\reg(d)$. Then 
\[
\cl^d:\CH^d(X, X^*)\to F^dK_0(X)
\]
is a surjective group homomorphism with kernel killed by $(d-1)!$
\end{IntroCor}

Similar results, without relying on the unpublished manuscript \cite{L85}, have been obtained by Biswas-Srinivas \cite{BS98} for $k$ algebraically closed and $X$ of dimension two.   Binda-Krishna-Saito \cite{BKS} handle the case of $X$ of pure dimension two over an arbitrary infinite field. Gupta-Krishna \cite{GK} show that the kernel of $\cl^d$ is killed by $(d-1)!$  for reduced, finite-type affine $X=\Spec A$, equi-dimensional of dimension $d$ over $k$, and $k$ an infinite field that is either algebraically closed or has $(d-1)!\in k^\times$.

\begin{IntroRemark} We  use codimension $d$ rather than dimension zero as the basis for our group of cycles since the target of the Chern class $c_d$ is naturally built out of codimension $d$ cycles, while if $X$ has additional components of dimension $<d$, some 0-cycles supported in the regular locus of $X$ will have codimension $<d$. Our ultimate goal of having a Riemann-Roch theorem for our cycle group forces us to use the group $\CH^d(X, X^*)$ rather than the construction $\CH_0(X, Y)$ of \cite{LW}.

In fact, the Chow group $\CH_0(X,Y)$ may have some poor properties if $X$ is not equi-dimensional over $k$, see Example~\ref{ex:CH0} below for one such case.
\end{IntroRemark}
 
Several papers rely on or cite the result stated in the  Corollary, and many of them cite the unpublished manuscript \cite{L85} for this, so one purpose of this paper is to remove this gap, at least under the assumption of an infinite base-field.  The main result in \cite{L85a} is that for $X$ a reduced quasi-projective surface over $k$,  there is an isomorphism
\[
\CH_0(X, X_\sing)\cong H^2(X, \sK^M_2).
\]
The proof in \cite{L85a} relies on a Riemann-Roch theorem for 0-cycles from the unpublished manuscript \cite{L85}, so our result here fills this gap, at least for $X$ over an infinite field. In fact, this was already filled for quasi-projective, reduced $X$ of pure dimension two over an infinite field by \cite[Theorem 7.7]{BKS}, using a different method.\footnote{In \cite{BKS}, they use  a somewhat different group of relations than what we use here, but we show in \S\ref{sec:Presentation}  that the two  groups of relations agree.}  Other papers that can use Theorem~\ref{IntroThm} and its Corollary  to replace \cite{L85} include \cite[\S 5]{BindaKrishna}, \cite{L87, L85b} and  \cite{Murthy}. 
 
 \begin{IntroNotConv} 1. Throughout this paper, $k$ will be a field. We let $\Sch_k$ denote the category of separated schemes of finite type over $k$ and let $\Sch^\qp_k$ denote the full subcategory of quasi-projective $k$-schemes. Unless noted otherwise, all morphisms are assumed to be in $\Sch_k$.\\[2pt]
 2. For $Y$ a scheme, we often consider only the underlying topological space of $Y$, speaking for example of open or closed subsets of $Y$.   For $Z$ a closed subscheme of $Y$, the {\em support} of $Z$, $\supp(Z)$, is $Z$, considered as a closed subset of $Y$, and we say that $Z$ is supported in a closed subset $Y_0$ of $Y$ if  $\supp(Z)\subset Y_0$.\\[2pt]
 3. Let $Y$ be a finite-type $k$ scheme. If $Y$ is integral, we define $\dim_kY$ to be the transcendence dimension of $k(Y)$ over $k$, which is the same as the Krull dimension of $Y$. For $y\in Y$, define $\dim_k(Y,y)$ to be the maximum of $\dim_kY'$ as $Y'$ runs over the integral components of $Y$ containing $y$, and define $\dim_kY:=\max_{y\in Y}\dim_k(Y,y)$, i.e., the maximum of $\dim_kY'$ as $Y'$ runs over the integral components of $Y$. If $\dim_kY=d$, we say $Y$ has dimension $d$, and we say $Y$ has pure dimension $d$ if $\dim_k(Y,y)=d$ for all $y\in Y$; we also express this by saying that $Y$ is equi-dimensional of dimension $d$ over $k$.
 
 We let $Y(i)$ denote the union of the integral components $Y'$ of $Y$ with $\dim_kY'=i$, $Y(\le i):=\cup_{j\le i}Y(j)$, $Y(\ge i)=\cup_{j\ge i}Y(j)$. We also have the open subschemes $Y(i)^0=Y\setminus\cup_{j\neq i}Y(j)$ of $Y$, $1=0,\ldots, \dim_kY$.

  Let $Z\subset Y$ be a closed subscheme. We say $Z$ has pure codimension $c$ in $Y$ if $\dim_k(Z,z)+c=\dim_k(Y,z)$ for all $z\in Z$, and we say that $Z$ has codimension $\ge c$ in $Y$ if for each $z\in Z$, we have $\dim_k(Z,z)+c\le \dim_k(Y,z)$. Note that $Z=\0$ if $Z$ has codimension $\ge \dim_kY+1$. \\[2pt]
 4. For $Y\in \Sch_k$, line bundles $L_1,\ldots, L_r$ on $Y$ with invertible sheaves of sections $\sL_1,\ldots, \sL_r$ and sections $s_i\in H^0(Y, \sL_i)$, $i=1,\ldots, r$, the zero-subscheme $V(s_1,\ldots, s_r)$ is defined as the pullback in the Cartesian square
 \[
 \xymatrix{
 V(s_1,\ldots, s_r)\ar[r]\ar[d]&Y\ar[d]^0\\
 Y\ar[r]^-{(s_1,\ldots, s_r)}&L_1\times_Y\ldots\times_YL_r
 }
 \]
If $\sI$ is an ideal sheaf in $\sO_Y$ and $s_i\in H^0(Y, \sI\otimes\sL_i)$, we define $V(s_1,\ldots, s_r)$ as above, considering $s_i$ as sections of $\sL_i$, so $V(s_1,\ldots, s_r)\supset \Spec_{\sO_Y}\sO_Y/\sI$.\\[2pt]
5. Let $W\subset Y$ be a closed subscheme of a $Y\in \Sch_k$. We call $W$ a local complete intersection on $Y$ if the inclusion $i_W:W\to Y$ is a regular embedding, that is, for each $w\in W$, the ideal $\sI_{W,w}\subset \sO_{Y,w}$ is generated by a regular sequence. If $Y=\Spec A$ is affine, we call $W$ a complete intersection in $Y$ if the ideal $I_W\subset A$ is of the form $I_W=(a_1,\ldots, a_d)$ with $a_1,\ldots, a_d$ a regular sequence in $\sO_{Y,w}$ for each $w\in W$. 
 \end{IntroNotConv}
 
 \noindent
 {\em Acknowledgement}. I would like to thank Federico Binda, Amalendu Krishna and V. Srinivas for their very helpful comments and suggestions concerning an earlier version of this paper.

\section{0-cycles on a singular variety} \label{sec:Defs} 
\begin{definition} Take $X\in\Sch^\qp_k$.  Let  $d=\dim_kX$ and suppose that the open subscheme $X(d)^0$ of $X$ is reduced. \\[5pt]
1. Let $X_\reg\subset X$ denote the regular locus of $X$ over $k$, that is $x\in X$ is in $X_\reg$ if and only if the local ring $\sO_{X,x}$ is a regular local ring. Since $X$ is  of finite type over a field, $X_\reg$ is an  open subscheme of $X$. We can write $X_\reg$ as a disjoint union
\[
X_\reg=\amalg_{i=0}^dX_\reg(i),
\]
with $X_\reg(i)$ open in  $X(i)^0$;  $X_\reg(i)$ is regular of pure dimension $i$ over $k$ or is empty. If $X(i)^0$ is reduced, then $X_\reg(i)$ is dense in  $X(i)^0$, in particular  $X_\reg(d)$ is dense in $X(d)^0$. For these facts, see Lemma~\ref{lem: ReducedRegOpen} below.

Let $X_\sm\subset X_\reg$ denote the smooth locus, 
\[
x\in X_\sm\Leftrightarrow  \Omega_{X/k, x}\text{ is a free $\sO_{X,x}$-module of rank }\dim_k(X,x). 
\]
$X_\sm$ is an open subscheme of $X_\reg$ and is a disjoint union
\[
X_\sm=\amalg_{i=0}^dX_\sm(i),
\]
with $X_\sm(i)$ smooth of pure  dimension $i$ over $k$.\\[5pt]
2. Let $X_\sing:=X\setminus X_\reg$, which we call the singular locus of $X$,  and let $X^{(d)}_\sing:=X\setminus X_\reg(d)$.\\[5pt]
3. Let $X_\CM\subset X(d)^0$ be the Cohen-Macaulay locus in $X(d)^0$.\\[5pt]
4. For $Y\in \Sch_k$, and an integer $n\ge0$, let $Z_n(Y)$ be the free abelian group on the integral closed subschemes of dimension $n$ over $k$, and $Z^n(Y)$ the free abelian group on the integral closed subschemes of pure codimension $n$ in $Y$.
\end{definition}

For the reader's convenience, we recall why a reduced finite type $k$-scheme has dense open regular locus.

\begin{lemma}\label{lem: ReducedRegOpen} For $X\in \Sch_k$, $X_\reg\subset X$ is open. If $X$ is reduced, then  $X_\reg$ is dense in $X$.
\end{lemma}

\begin{proof} Recall that a scheme $Y$ is said  be J-2 if for every  morphism $f:Z\to Y$, locally of finite type,   the regular locus $Z_\reg$ in $Z$ is open in $Z$. By \cite[\href{https://stacks.math.columbia.edu/tag/07R5}{tag 07R5}]{stacks-project} $k$ is J-2, so for each $Z\in \Sch_k$, $Z_\reg\subset Z$ is open.  

Let $x\in X$ be a generic point. If $X$ is reduced, the local ring $\sO_{X,x}$ is a field, hence a regular ring, so $x$ is in $X_\reg$. Thus $X_\reg$ is an open subscheme of $X$ containing each generic point, hence is dense in $X$.
\end{proof}

\begin{remark}\label{rem:Prelims}  Take $X\in\Sch^\qp_k$, let   $d=\dim_kX$ and suppose that $X(d)^0$ is reduced.\\[5pt]
1. $X^{(d)}_\sing=X_\sing\cup X(\le d-1)$.\\[2pt]
2. $X_\reg(d)=(X(d)^0)_\reg$, hence $X_\reg(d)$ is open in $X(d)$ and $X(d)_\sing\subset X_\sing\subset X^{(d)}_\sing$.\\[2pt]
3. By \cite[\href{https://stacks.math.columbia.edu/tag/00RF}{Tag 00RF}]{stacks-project},  $X_\CM$ is an open subset of $X(d)^0$, hence $X_\CM$ is also open in $X$. Since we are assuming that $X(d)^0$ is reduced, each local ring $\sO_{X,x}$ has depth $\ge1$, and thus the complement of $X_\CM$ in $X(d)^0$ has codimension at least two; as $X(d)\setminus X(d)^0$ has codimension $\ge2$ in $X(d)$,   $X(d)\setminus X_\CM$ in $X(d)$ has codimension at least two and $X(d)\setminus X_\CM\subset X_\sing$.  \\[2pt]
4. In general $X_\sm\subset X_\reg$, and the two are equal if $k$ is perfect. For $k\subset L$ a field extension, $(X\times_kL)_\sm=(X_\sm)\times_kL$. The same holds for $X_\reg$ if $L$ is separable over $k$, but not in general. 
\end{remark}

\begin{definition} A pair $(Y, Y^*)$ consisting of a  quasi-projective $k$-scheme $Y$  and a closed subset $Y^*$ of $Y$ is called an {\em admissible pair over $k$} if  $Y(\dim_kY)^0$ is reduced and $Y\setminus Y^*$ is a dense open subset of $Y_\reg(\dim_kY)$.
\end{definition}
Equivalently, $Y$ is in $\Sch^\qp_k$,   $Y(d)^0$ is  reduced, $Y\setminus Y^*$ is regular and of pure dimension $\dim_kY$ over $k$ and $\dim_kY^*<\dim_kY$.  For $d=\dim_kY$,  we often refer to elements of $Z^d(Y\setminus Y^*)$ as 0-cycles rather than codimension $d$ cycles.
As we are usually working in $\Sch_k$, we will usually drop the ``over $k$'' and refer simply to an admissible pair.

Take $Y\in\Sch_k$ and let $Y^*\subset Y$ be a closed subset such that
$Y\setminus Y^*$ is regular.
In \cite{LW}, with C. Weibel, we have defined the group of 0-cycles on $Y$ relative to $Y^*$, $\CH_0(Y,Y^*)$.\footnote{The definition in \cite{LW} also requires that $Y\setminus Y^*$ is dense in $Y$, but we omit this condition to allow for a comparison with $\CH^d(X,X^*)$ under less restrictive conditions.}   This is a quotient of the group of 0-cycles, $Z_0(Y\setminus Y^*)$,  modulo a subgroup $R_0(Y,Y^*)$, which is defined using the notion of a {\em Cartier curve} on $Y$ relative to $Y^*$ (see \cite[Definition 1.2]{LW}).

\begin{definition}\label{def:CartierCurve}  For $Y\in \Sch_k$, let $Y^*\subset Y$ be a closed subset such that $Y\setminus Y^*$ is regular. A {\em Cartier curve} on $Y$ relative to $Y^*$ is a purely 1-dimensional closed subscheme $C$ of $Y$ such that\\[5pt]
i. $C\cap Y^*$ is a finite set.\\[2pt]
ii.  At each point $x\in C\cap Y^*$, the ideal $\sI_{C,x}\subset \sO_{Y,x}$ is generated by a regular sequence.
\end{definition}

\begin{remark}\label{rem:CartierCurve}
1. Let $C\subset Y$ be a Cartier curve on $Y$ relative to $Y^*$ and take $y\in C\cap Y^*$. Then a regular sequence in $\sO_{Y,y}$ generating $\sI_{C,y}$ must have length $\dim_k(Y,y)-1$. Indeed, it follows from \cite[Lemma 5]{Matsumura} that if $\sI_{C,y}$ is generated by a regular sequence of length 
 $\ell$, then
\[
 1=\dim\sO_{Y,y}/\sI_{C,y}=\dim\sO_{Y,y}-\ell.
 \]   
2. Let $Y, Y^*$, $C$ and $y\in C\cap Y^*$ be as in (1) and let $Y'\subset Y$ be an integral component of $Y$ containing $y$. Suppose that $C$ has an integral component $C'$ with $C'\subset Y'$.  Then $\dim_kY'=\dim_k(Y,y)$. Indeed, since each generic point of $C$ is contained in $Y_\reg$, $Y'$ is the unique integral component of $Y$ containing $C'$. By (1), the $\sI_{C,y}$-depth of $\sO_{Y,y}$ is $\dim_k(Y,y)-1$ and by  \cite[Theorem 27, (15.G) Proposition]{Matsumura}, we have
\[
\dim_k(Y,y)-1\le \text{depth}\sO_{Y', C'}\le\dim \sO_{Y', C'}=\dim_kY'-1\le \dim_k(Y,y)-1,
\]
giving equality throughout. 

In consequence, if $C$ is connected and has one irreducible component supported in $Y(i)$ for some $i$, then $C$ is supported in $Y(i)\setminus Y(\ge i+1)$. 
\end{remark}

\begin{definition} For $C\in \Sch_k$, $\dim_kC=1$,   let $C_{(0)}$ denote the set of closed points of $C$. For $S\subset C_{(0)}$ a finite set of closed points, let $\sO_{C:S}$ denote the colimit
\[
\sO_{C:S}=\colim_{T\subset (C_{(0)}\setminus S)}\sO_C(C\setminus T).
\]
and let $\sO_{C,S}$ denote the usual semi-local ring
\[
\sO_{C,S}=\colim_{\substack{U\subset C\text{ open,}\\ S\subset U}}\sO_C(U).
\]
\end{definition}

\begin{remark}
If $C$ has generic points $\eta_1,\ldots, \eta_r$, then $\sO_{C:\0}$ is the total quotient ring $k(C)$, 
$k(C)=\prod_{i=1}^r\sO_{C,\eta_i}$.

For $S\subset C_{(0)}$ a finite set,  let $C_S$ be the union of the integral components $C'$ of $C$ with $S\cap C'\neq\0$ and let $C^c_S\subset C$ be the open subscheme $C\setminus C_S$. Then
\[
\sO_{C:S}=\sO_{C,S}\times k(C^c_S). 
\]
We let  $\alpha_S:\sO_{C:S}\to k(C)$ be the localization map.
\end{remark} 

\begin{definition} For $C\in \Sch_k$ with $\dim_kC=1$,  let $i: C_\red\hookrightarrow C$ be the associated reduced closed subscheme with irreducible components $C_{1,\red},\ldots, C_{r,\red}$, where $C_{i,\red}$ has generic point $\eta_i$, and let $i_j:C_{j,\red}\to C$ be the inclusion.  Then $k(C_\red)=\prod_{i=1}^r k(\eta_i)$. Let $m_i$ denote the length of the local ring  $\sO_{C,\eta_i}$ as $\sO_{C,\eta_i}$-module. For $f\in k(C)^\times$, let $f_j$ denote the restriction of $f$ to $k(\eta_i)^\times$, considered as a rational function on $C_{j,\red}$, and define
$\Div f\in Z_0(C)=Z_0(C_\red)$ as
\[
\sum_{j=1}^r i_{j*}m_j\Div f_j, 
\]
where $\Div f_j$ is the usual divisor of $f_j$ \cite[\S 1.3]{Fulton}; note that by definition, $\Div f_j=0$ if  $\dim_kC_{j, \red}=0$.  For $f\in \sO_{C:S}^\times$, define $\Div f=\Div\, \alpha_S(f)$.
\end{definition}

\begin{definition}\label{def:RatlEquiv} Let $(X, X^*)$ be an admissible pair. Define $R^d(X,X^*)\subset Z^d(X\setminus X^*)$ to be the subgroup generated by 0-cycles of the form $i_{C*}(\Div f)$, where $i_C:C\to X$ is the inclusion of a Cartier curve on $X$ relative to $X^*$, and $f$ is a unit in   $\sO_{C: C\cap X^*}$:
\begin{multline*}
R^d(X,X^*):=\<\{i_{C*}(\Div f)\mid i_C:C\to X\text{ is a Cartier curve relative to }X^*,\\\text{and } f\in \sO_{C:C\cap Y}^\times\}\>
\end{multline*}
\end{definition}

\begin{definition} For $(X, X^*)$ an admissible pair, define 
\[
\CH^d(X,X^*)=Z^d(X\setminus X^*)/R^d(X,X^*).
\] 
\end{definition}
Note that $X\setminus X^*\subset X_\reg(d)\subset X_\reg$ by the definition of admissible pairs.

\begin{remark} For $X$ reduced, the definition of $\CH^d(X, X^*)$ given here already appears in  \cite[\S 2]{BS98} in case $\dim_kX=2$, with $X^*=X_\sing\cup X(\le d-1)$ (this closed subset is denoted $X_\sing$ in {\it loc.\! cit.}), and is extended to  reduced $X$ of arbitrary dimension $d$ in \cite[\S 2]{BS99}, again with $X^*$ the minimal choice $X_\sing\cup X(\le d-1)$ (still denoted $X_\sing$ in \cite{BS99}). This definition is also used in \cite{ESV}.  In \cite{BS99}, they kindly cite \cite{LW} as source, even though this definition does not really appear in \cite{LW} as such, unless $X$ is equi-dimensional over $k$, as we will explain below.
\end{remark} 

We recall the definition of the Chow group of 0-cycles from \cite{LW}.
\begin{definition}[\hbox{\cite[Definition 1.2]{LW}}]\label{def:LW0Cyc} 1. Let $Y$ be a finite-type separated $k$-scheme and let $Y^*$ be a closed subset such that $Y\setminus Y^*$ is regular. Let $R_0(Y, Y^*)\subset Z_0(Y\setminus Y^*)$ be the subgroup generated by 0-cycles of the form $i_{C*}(\Div f)$, where $i_C:C\to Y$ is the inclusion of a Cartier curve on $Y$ relative to $Y^*$ and $f\in\sO_{C: C\cap Y^*}^\times$.\\[2pt]
2. With $Y, Y^*$ as in (1), define $\CH_0(Y, Y^*)=Z_0(Y\setminus Y^*)/R_0(Y, Y^*)$.
\end{definition}

\begin{remark}\label{rem:CartierCurve2} 1. Let $i_C:C\to X$ be a Cartier curve on $X$ relative to $X^*$ and take $f\in\sO_{C:C\cap X^*}^\times$. Then the support of $\Div f$ is contained in $C\setminus X^*$, so  $R^d(X,X^*)$ is indeed a subgroup of $Z^d(X\setminus X^*)$; similarly, $R_0(Y, Y^*)$ is indeed a subgroup of $Z_0(Y\setminus Y^*)$. \\[2pt]
2.  In addition to the conditions (i) and (ii) of Definition~\ref{def:CartierCurve}, the definition of Cartier curve on $Y$ relative to $Y^*$  given in, e.g., \cite{BS99} 
requires  
\\[5pt]
iii. $C$ has no embedded components in $C\cap (Y\setminus Y^*)$ and if 
$C'$ is an irreducible component of $C$ with $C'\cap Y^*=\0$, then $C'$ is reduced at its generic point.\\[5pt]
It is not hard to show that adding the requirement (iii) does not affect the resulting subgroup $R_0(Y,Y)$ of $Z_0(Y\setminus Y^*)$,  so the group $\CH_0(Y, Y^*)$ defined above agrees with that of {\it loc. cit.}  \\[2pt]
3. Since we are using codimension rather than dimension, one might think that a Cartier ``curve'' should rather be a closed subscheme $C$ of pure {\em codimension} $d-1$ on $X$,   satisfying the conditions (i) and (ii) of Definition~\ref{def:CartierCurve}. However,  such a $C$ decomposes a disjoint union
\[
C=C(d)\amalg C(d-1)
\]
where $C(d)$ is a purely one-dimensional closed subscheme of $X$, supported in $X(d)$, and $C(d-1)$ is a purely 0-dimensional closed subscheme of $X$, supported in $X(d-1)\setminus X(d)$. Then $C(d)$ is a Cartier curve relative to $X^*$ following  Definition~\ref{def:CartierCurve}. Moreover, letting $f(d)$ denote the restriction of $f$ to $C(d)$, we have
\[
i_{C*}(\Div f)=i_{C(d)*}(\Div f(d)),
\]
since $\dim_kX(d-1)=0$. Thus,  we end up with the same subgroup $R^d(X, X^*)$ whether we use purely dimension one closed subschemes or purely codimension $d-1$ closed subschemes to define $R^d(X, X^*)$; we can even omit the condition (ii) at points of $C(d-1)\subset X(d-1)\setminus X(d)$.
\\[2pt]
4. Suppose $X$ has connected components $X_1,\ldots, X_s$ and let $X^*_j=X^*\cap X_j$. Suppose that $\dim_kX_j=\dim_kX$  for $j=1,\ldots, r$ and $\dim_kX_j<\dim_kX$  for $j=r+1,\ldots, s$. Then
\[
\CH^d(X, X^*)=\oplus_{j=1}^r\CH^d(X_j, X_j^*)
\]
For this reason we often assume that $X$ is connected.
\end{remark} 

We compare $\CH^d(X, X^*)$  and $\CH_0(X, Y^*)$. Suppose we have an admissible pair $(X, X^*)$ and a closed subset $Y^*\subset X^*$ such that $X\setminus Y^*$ is regular. Let $d=\dim_kX$. Then both $\CH^d(X, X^*)$ and $\CH_0(X, Y^*)$ are defined, and the inclusion $j^{X^*, Y^*}:X\setminus X^*\to X\setminus Y^*$ induces the map
\[
j^{X^*, Y^*}_*:Z^d(X\setminus X^*)=Z_0(X\setminus X^*)\to Z_0(X\setminus Y^*).
\]
Since a Cartier curve relative to $X^*$ is  also a Cartier curve relative to $Y^*$,  we have $j^{X^*, Y^*}_*(R^d(X, X^*))\subset R_0(X, Y^*)$, giving the map
\[
j^{X^*,Y^*}_*:\CH^d(X, X^*)\to \CH_0(X, Y^*).
\]

\begin{proposition}\label{prop:Comparison} Let $(X, X^*)$ be an admissible pair with $X$ of dimension $d$ over $k$ and let $Y^*\subset X^*$ be a closed subset such that $X\setminus Y^*$ is regular. \\[5pt]
1.  If $Y^*\cap X(d)=X^*\cap X(d)$, then $j^{X^*,Y^*}_*$ is injective\\[2pt]
2. If   $Y^*=X^*$, then $j^{X^*,Y^*}_*$ is an isomorphism.
\end{proposition}

\begin{proof} It follows immediately from the definitions that, after identifying $Z^d(X\setminus X^*)$ with $Z_0(X\setminus X^*)$, we have $R^d(X, X^*)=R_0(X, X^*)$, proving (2).

For (1), if a zero-cycle $z\in Z^d(X\setminus X^*)=Z_0(X\setminus X^*)$ goes to zero in $\CH_0(X, Y^*)$, there are Cartier curves $C_1,\ldots, C_s$ on $X$ relative to $Y^*$ and functions $f_j\in \sO_{C_j: C_j\cap Y}^\times$, $j=1,\ldots, s$, such that
\[
z=\sum_{j=1}^si_{C_j*}\Div(f_j).
\]
We may assume that each $C_j$ is connected. 

Suppose that $C_1,\ldots, C_r$ each have some integral component supported in $X(d)$, and $C_{r+1},\ldots, C_s$ are all supported in $X(\le d-1)$. By Remark~\ref{rem:CartierCurve}(2), $C_i$ is supported in $X(d)$ for $i=1,\ldots, r$ and $C_i$ is supported in $X(\le d-1)\setminus X(d)$ for $i=r+1,\ldots, s$. Thus 
\[
z=\sum_{j=1}^ri_{C_j*}\Div(f_j).
\]
Since $C_j$ is a Cartier curve relative to $Y^*$ and $Y^*\cap X(d)=X^*\cap X(d)$ by assumption, we see that  $C_j$ is a Cartier curve relative to $X^*$ for $j=1,\ldots, r$, and thus $[z]=0$ in $\CH^d(X, X^*)$, showing that $j^{X^*, Y^*}_*$ is injective.
\end{proof}

\begin{remark}\label{rem:Comparison2} Following the argument used in the proof of Proposition~\ref{prop:Comparison}, we see that the subgroup of relations $R_0(X, Y^*)$ is a direct sum
\[
R_0(X, Y^*)=R^d(X, Y^*\cup X(\le d-1))\oplus R_0'(X(\le d-1), Y^*\cap X(\le d-1))
\]
where 
\[
R_0'(X(\le d-1), Y^*\cap X(\le d-1))\subset R_0(X(\le d-1), Y^*\cap X(\le d-1))
\]
is the subgroup generated by 0-cycles of the form $i_{C*}(\Div f)$, with  $i_C:C\to X(\le d-1)$ a Cartier curve relative to $Y^*\cap X(\le d-1)$ such that $C\cap X(d)=\0$, with $f\in \sO_{C: C\cap Y^*}^\times$. 

This refines Proposition~\ref{prop:Comparison} to give the description of $\CH_0(X, Y^*)$ as
\begin{multline*}
\CH_0(X, Y^*)=\CH^d(X, Y^*\cup X(\le d-1))\\\oplus Z_0(X(\le d-1)\setminus Y^*)/R_0'(X(\le d-1), Y^*\cap X(\le d-1)).
\end{multline*}
\end{remark}

\begin{ex}\label{ex:CH0} Let $X=\Spec k[x,y,z]/(z(x,y))$, so $X$ is the union of the $x$-$y$-plane with the $z$-axis. We take $Y^*$ to be the origin. Since a curve supported in $X(\le 1)$ must contain $Y^*=X(\le1)\cap X(2)$, it follows that the group of relations $R_0'(z\text{-axis}, \{(0,0,0)\})$ is zero and  as $\CH^2(X, z\text{-axis})=0$, we have
\[
\CH_0(X, \{(0,0,0)\})=Z_0(z\text{-axis}\setminus\{(0,0,0\}).
\]
\end{ex}

\begin{remark}\label{rem:CH1=Pic} Let $(X, X^*)$ be an admissible pair with $X$ of pure dimension one. Then sending a closed point $x\in X\setminus X^*$ to the invertible sheaf $\sO_X(x)$ induces an isomorphism
\[
\cl^1:\CH^1(X, X^*)\xrightarrow{\sim} \Pic(X)
\]
After identifying  $\CH^1(X, X^*)$ with $\CH_0(X, X^*)$, this is proven in \cite[Proposition 1.4]{LW}.
\end{remark}

\begin{remark} Let $Y\in \Sch_k$ be a Cohen-Macaulay scheme, that is, $\sO_{Y,y}$ is a Cohen-Macaulay local ring for each $y\in Y$, see \cite[(16.A)]{Matsumura}. Let $W\subset Y$ be a closed subscheme of pure codimension $c>0$. Then $W$ is a local complete intersection on $Y$ if and only if the ideal $\sI_{W,w}\subset \sO_{Y,w}$ is generated by $c$ elements for each $w\in W$. Indeed, since $\sO_{Y,w}$ is Cohen-Macaulay, elements $s_1,\ldots, s_c$ in the maximal ideal $\mathfrak{m}_w\subset \sO_{Y,w}$ form a regular sequence if and only if the ideal $(s_1,\ldots, s_c)$ has height $c$ \cite[Theorem 31]{Matsumura}.

We will be using this fact without further comment throughout the paper.
\end{remark}

Given admissible pairs $(X, X^*)$, $(X, X')$ with $X'\subset X^*$, the inclusion $j^{X^*, X'}:X\setminus X^*\to X\setminus X'$ induces the map $j^{X^*, X'}_*:Z^d(X\setminus X^*)\to Z^d(X\setminus X')$. As a Cartier curve on $X$ relative to $X^*$ is automatically a Cartier curve relative to $X'$, we see that
$j^{X^*, X'}_*(R^d(X, X^*))\subset R^d(X, X')$, giving the induced map ($d=\dim_kX$)
\[
j^{X^*, X'}_*:\CH^d(X, X^*)\to \CH^d(X, X').
\]

\begin{lemma}\label{lem:CHSurj} Let $(X, X^*)$, $(X, X')$ be admissible pairs with $X'\subset X^*$. Let $d=\dim_kX$. Then the map
\[
j^{X^*, X'}_*:\CH^d(X, X^*)\to \CH^d(X, X')
\]
 is surjective.
\end{lemma}

\begin{proof} Take a closed point $x\in X^*\setminus X'$ and note that $x$ is in $X_\reg(d)$. It suffices to show that $[x]\in \CH^d(X, X')$ is in the image of $j^{X^*, X'}_*:\CH^d(X, X^*)\to \CH^d(X, X')$. For this, choose a locally closed immersion $X\subset \P^N$ and let $C\subset X(d)$ be a general  intersection in $X(d)$ of $d-1$ hypersurfaces of large degree $e$, containing $y$. Taking $e\gg0$, we may assume that 
\begin{enumerate}
\item[(a)] $C$ is of pure codimension $d-1$ on $X(d)$.
\item[(b)] $x$ is a regular point of $C$.
\item[(c)]   $C\cap X^*$ is a finite set 
\item[(d)]  $C$ is contained in $X_\CM$.
\end{enumerate}
Note that $X^*$ has dimension $\le d-1$ and  $X(d)\setminus X_\CM$  has dimension $\le d-2$ (see Remark~\ref{rem:Prelims}  above), which is why we may assume that (c) and (d) hold.

Since $X_\CM\subset X(d)^0$, it follows from (a) and (d) that $C$ is locally defined by a regular sequence of length $d-1$ as a closed subscheme of $X$, and hence by (c), $C$ is a Cartier curve relative to $X^*$. By (b) and the usual arguments using the Chinese remainder theorem, we may find an $f\in \sO_{C: C\cap X'}^\times$ such that $i_{C*}(\Div f)=1\cdot x-y$ with $y$ supported in $X\setminus X^*$. This gives us the class $[y]\in \CH^d(X, X^*)$ with $1\cdot x-y$ in $R^d(X, X')$, hence $[x]=j^{X^*, X'}_*([y])$ in $\CH^d(X, X')$.
\end{proof}

\begin{definition} 1. Let $(X, X^*)$ be an admissible pair. Let $S\subset X$ be a closed subset and let $C$ be a Cartier curve on $X$ relative to $X^*$. We say that $C$ is {\em $S$-strict} if $C\cap S=\0$.\\[2pt]
2. Let $R^d(X, X^*)_S\subset R^d(X, X^*)$ be the subgroup generated by $i_{C*}(\Div f)$ with $C$ an $S$-strict Cartier curve on $X$ relative to $X^*$, and $f\in \sO_{C: C\cap X^*}^\times$.\\[2pt]
3. Define
\[
 R^d(X, X^*)^\CM:= R^d(X, X^*)_{X(d)\setminus X_\CM}
\] 
\end{definition}

A version of the following ``moving lemma''  is proven in  \cite[Lemma 2.1]{BS99}; as our context is a bit different we go through the argument, which is essentially the same as that given in {\it loc.\! cit}. 

\begin{lemma}[A moving lemma]\label{lem:MovingLemmaR0} Take $X\in \Sch_k^\qp$, let $d=\dim_kX$, and suppose $X(d)^0$ is reduced.  Let $Y^*\subset X$ be a closed subset such that $X\setminus Y^*$ is regular and $X(d)\setminus Y^*$ is dense in $X(d)$. Let $S\subset Y^*\cap X(d)$ be a closed subset such that $\codim_{X(d)}S\ge 2$.  Let $i_C:C\to X$ be a Cartier curve on $X$ relative to $Y^*$ such that $C$ is supported in $X(d)$, and take $f\in \sO_{C: C\cap Y^*}^\times$. Then there are Cartier curves on $X$ relative to $Y^*$, $i_{C_j}:C_j\to X$, and $f_j\in \sO_{C_j: C_j\cap Y^*}^\times$, $j=1,\ldots, s$, such that each $C_j$ is supported in $X(d)$, $C_j\cap S=\0$,  and we have
\[
i_{C*}(\Div f)=\sum_{j=1}^si_{C_j*}(\Div f_j).
\]
\end{lemma}

\begin{proof} If we write $C=C'\cup C''$ where   $C''$ is disjoint from $Y^*$, then clearly  $C''$ is an $S$-strict Cartier curve relative to $Y^*$, $f_{|C''}$ is in $\sO_{C'': C''\cap Y^*}^\times=k(C'')^\times$, $f_{|C'}$ is in $\sO_{C': C'\cap Y^*}^\times$ and
\[
i_{C*}(\Div f)=i_{C'*}(\Div f_{|C'})+i_{C''*}(\Div f_{|C'}).
\]
Thus we may assume from the start that each irreducible component of $C$ intersects $Y^*$. 

By assumption $C$ is supported in $X(d)$ and there is an affine open neighborhood $U$ of the finite set $C\cap Y^*$ in $X$ such that $\sI_C$ is defined by a regular sequence $g_1,\ldots, g_{d-1}$ on $U$, equivalently, the images of $g_1,\ldots, g_{d-1}$ in $\sI_C/\sI_C^2$ generate this sheaf on $U$.  Letting $\sO_X(1)$ be a very ample invertible sheaf on $X$, we may assume that the $g_i$ extend to global sections $g_i\in H^0(X, \sI_C(m))$ for some $m\gg0$. Again taking $m\gg0$, we may assume that $H^0(X, \sI_C(m))$ generates $\sI_C(m)$ globally. Thus we may assume that the linear system defined by $H^0(X, \sI_C(m))$ has only $C$ as base-locus. Since $k$ is infinite, we may assume that 
\[
\codim_{S\setminus C} V(g_1,\ldots, g_i)\cap (S\setminus C)\ge i
\]
for each $i=1,\ldots, d-1$. But since $\dim S\le d-2$, it follows that we may choose $g_1,\ldots, g_{d-1}$ such that  $V(g_1,\ldots, g_{d-1})\cap (S\setminus C)=\0$. 

Similarly, we may assume that 
\[
\codim_{Y^*\setminus C} V(g_1,\ldots, g_i)\cap (Y^*\setminus C)\ge i
\]
for each $i=1,\ldots, d-1$. In particular,  $V(g_1,\ldots, g_{d-1})\setminus C$ is disjoint from $X(i)$ for  $i<d-1$. Also, since $X(d-1)\cap X(j)$ has codimension $\ge2$ in $X$ for all $j\neq d-1$, we may assume that $V(g_1,\ldots, g_{d-1})\cap X(d-1)\setminus C$ is a finite set contained in $X(d-1)^0$.
Letting 
\[
\tilde{X}_\CM=X_\CM\amalg X(d-1)^0.
\]
we may therefore assume that $V(g_1,\ldots, g_{d-1})\setminus C\subset \tilde{X}_\CM$, and that $V(g_1,\ldots, g_{d-1})\setminus C$ has pure codimension $d-1$ in $\tilde{X}_\CM\setminus C$.  Since $C$ is supported in $X(d)$, $V(g_1,\ldots, g_{d-1})\setminus C$ is the disjoint union of closed subschemes,
\[
V(g_1,\ldots, g_{d-1})\setminus C=(V(g_1,\ldots, g_{d-1})\setminus C)_\CM\amalg V(g_1,\ldots, g_{d-1})_{d-1},
\]
with $(V(g_1,\ldots, g_{d-1})\setminus C)_\CM\subset X_\CM$ and 
$V(g_1,\ldots, g_{d-1})_{d-1}\subset  X(d-1)^0$. Similarly,  
\[
V(g_1,\ldots, g_{d-1})=V(g_1,\ldots, g_{d-1})_d\amalg V(g_1,\ldots, g_{d-1})_{d-1}
\]
with $V(g_1,\ldots, g_{d-1})_d$ supported in $X(d)$.  Since $X_\CM$ is Cohen-Macaulay, it follows that $g_1\ldots, g_{d-1}$ is a regular sequence at each point of $(V(g_1,\ldots, g_{d-1})\setminus C)_\CM$.

Since $C\setminus Y^*$ has pure codimension $d-1$ on $X\setminus Y^*$ and is contained in $X_\reg(d)$, we may also assume that  $V(g_1,\ldots, g_{d-1})\setminus (Y^*\cap C)$ has pure codimension $d-1$ in $X\setminus (Y^*\cap C)$, so   $g_1\ldots, g_{d-1}$ is a regular sequence at each point of $V(g_1,\ldots, g_{d-1})_d\setminus (Y^*\cap C)$. But since $g_1\ldots, g_{d-1}$ is also a regular sequence on $U$, it follows that $g_1\ldots, g_{d-1}$ is a regular sequence at each point of $V(g_1,\ldots, g_{d-1})_d$.  

Letting $E$ be the closure of 
$V(g_1,\ldots, g_{d-1})_d \setminus C$ in $X$, the assumption that $g_1,\ldots, g_{d-1}$ generates $\sI_C$ on $U$ shows that $E\cap C\subset X(d)\setminus U$. As $V(g_1,\ldots, g_{d-1})_{d-1}$ is a finite set disjoint from $X(d)$, we may assume that $U\cap V(g_1,\ldots, g_{d-1})_{d-1}=\0$, hence 
\[
V(g_1,\ldots, g_{d-1})\cap U=(C\cap U)\amalg (E\cap U),
\]
 in particular, 
\[
E\cap C\cap Y^*=\0,\ E\cap S=\0,
\]
$E$ is supported in $X(d)$, and is a local complete intersection on $U\cup (X\setminus C)$.  Since $Y^*\subset U\cup (X\setminus C)$,   we see that $E$ is an $S$-strict Cartier curve on $X$, supported in $X(d)$. 

Replacing $g_{d-1}$ with a sufficiently general element $g_{d-1}'\in H^0(X, \sO_X(m))$, we may assume as above that 
\[
V(g_1,\ldots, g_{d-1}')=V(g_1,\ldots, g_{d-1}')\cap X_\CM\amalg V(g_1,\ldots, g_{d-1}')\cap X(d-1)^0,
\]
writing  $V(g_1,\ldots, g_{d-1}')$ as
\[
V(g_1,\ldots, g_{d-1}')=V(g_1,\ldots, g_{d-1}')_\CM\amalg V(g_1,\ldots, g_{d-1}')_{d-1}.
\]
Let $E'=V(g_1,\ldots, g_{d-2},g_{d-1}')_\CM$. Arguing as above, $E'$ is a closed subscheme of $X$, $g_1,\ldots, g_{d-2}, g_{d-1}'$  is a regular sequence  at each point of $E'$, and $E'$ is an $S$-strict Cartier curve on $X$ relative to $Y^*$. We may also chose $g_{d-1}'$ so that $\Div g_{d-1}$ and $\Div g_{d-1}'$ have no common components in $X$, giving the well-defined rational function $g_{d-1}/g_{d-1}'$ on $X$.

We may write the rational function $f\in \sO_{C: C\cap Y^*}^\times$ as the restriction to $C$ of a fraction $f_0/f_\infty$, with  $f_0, f_\infty\in H^0(X, \sO_X(n))$ for some $n\gg0$ and with $f_0, f_\infty$ restricting to $\sO_{X, C\cap Y^*}^\times$ (after trivializing $\sO_X(n)$ in a neighborhood of $C\cap Y^*$.  As we may modify $f_0, f_\infty$ by adding any element $g\in H^0(X, \sI_C(n))$, we may argue as above to show that we may choose $f_0, f_\infty$ so that $\Div(f_0)$ and $\Div(f_\infty)$ have no common components on $X$, so we have the well-defined rational function $f_0/f_\infty$ on $X$, and in addition:
\\[5pt]
a. $V(g_1,\ldots, g_{d-2}, f_0)$ and $V(g_1,\ldots, g_{d-2}, f_\infty)$ have pure codimension $d-1$ on $X$, and are supported in $X_\CM\amalg X(d-1)^0$, writing them as above as
\[
V(g_1,\ldots, g_{d-2}, f_i)=V(g_1,\ldots, g_{d-2}, f_i)_\CM\amalg V(g_1,\ldots, g_{d-2}, f_i)_{d-1}
\]
for $i\in \{0, \infty\}$.\\[2pt]
b. $g_1,\ldots, g_{d-2}, f_0$,  is a regular sequence at each point of $V(g_1,\ldots, g_{d-2}, f_0)_\CM$, and  $g_1,\ldots, g_{d-2}, f_\infty$ is a regular sequence at each point of  $V(g_1,\ldots, g_{d-2}, f_\infty)_\CM$.\\[2pt]
c. $V(g_1,\ldots, g_{d-2}, f_0)\cap C\cap Y^*=\0$ and $V(g_1,\ldots, g_{d-2}, f_\infty)\cap C\cap Y^*=\0$.\\[2pt]
d. Let 
\[
\hat{C}:=V(g_1,\ldots, g_{d-2}, f_0)_\CM, \ \hat{C}':=V(g_1,\ldots, g_{d-2}, f_\infty)_\CM. 
\]
Then $\hat{C}$ and $\hat{C}'$  are supported in $X(d)$ and are $S$-strict Cartier curves on $X$, relative to $Y^*$.\\[2pt]
e. The rational function $f_0/f_\infty$ restricts to well-defined rational functions on $E$ and $E'$
\[
(f_0/f_\infty)_E\in \sO_{E: E\cap Y^*}^\times,\ (f_0/f_\infty)_{E'}\in \sO_{E': E'\cap Y^*}^\times
\]
f. The rational function $g_{d-1}/g_{d-1}'$ restricts to well-defined rational functions on $\hat{C}$ and $\hat{C}'$
\[
(g_{d-1}/g_{d-1}')_{\hat{C}}\in \sO_{\hat{C}: \hat{C}\cap Y^*}^\times,\ (g_{d-1}/g_{d-1}')_{\hat{C}'}\in \sO_{\hat{C}': \hat{C}'\cap Y^*}^\times
\]

An elementary computation shows that
\begin{multline*}
i_{C*}(\Div f)=i_{\hat{C}*}(\Div (g_{d-1}/g_{d-1}')_{\hat{C}})+i_{E*}(\Div(f_\infty/f_0)_E)\\+i_{E'*}(\Div(f_0/f_\infty)_{E'})+i_{\hat{C}'*}(\Div (g_{d-1}'/g_{d-1})_{\hat{C}'}).
\end{multline*}
Since $E, \hat{C}, E'$ and $\hat{C}'$ are all $S$-strict Cartier curves supported in $X(d)$, this completes the proof.
\end{proof}

\begin{lemma}[A moving lemma for \hbox{$\CH^d(X, X^*)$}]\label{lem:Mov1} Suppose that $k$ is an infinite field. Let $(X, X^*)$ be an admissible pair, let $d=\dim_kX$, and let $S\subset X^*\cap X(d)$ be a closed subset such that $\codim_{X(d)}S\ge2$. Then $R^d(X, X^*)_S= R^d(X, X^*)$.
\end{lemma}

\begin{proof} This follows directly from  Lemma~\ref{lem:MovingLemmaR0} with $Y^*=X^*$.
\end{proof}

\begin{proposition}\label{prop:CMMoving} Suppose that $k$ is an infinite field. Let $(X, X^*)$ be an admissible pair and let $d=\dim_kX$. Then   $R^d(X, X^*)=R^d(X, X^*)^\CM$. 
\end{proposition}

\begin{proof} This follows from Lemma~\ref{lem:Mov1}, taking $S=X(d)\setminus X_\CM\subset  X^*\cap X(d)$.
\end{proof}

As a complement to Proposition~\ref{prop:Comparison} we have one last comparison result.  Suppose $X$ is reduced. Let $Y^*\subset X$ be a closed subset such that  $X\setminus Y^*$ is regular and dense in $X$,  and let $X^*(\le i)=(Y^*\cap X(i))\cup X(\le i-1)$. If $i_C:C\to X$ is a connected Cartier curve relative to $Y^*$, then $C$ is supported in $X(i)\setminus X(\ge i+1)$ for some $i$ (see Remark~\ref{rem:CartierCurve}(2)). Via the identity
\[
Z_0(X\setminus Y^*)=\oplus_{i=0}^{\dim_kX}Z_0(X(\le i)\setminus X^*(\le i))
\]
this shows that the composition
\[
Z_0(X\setminus Y^*)=\oplus_{i=0}^{\dim_kX}Z^i(X(\le i)\setminus X^*(\le i))\to
\oplus_{i=0}^{\dim_kX}\CH^i(X(\le i), X^*(\le i))
\]
descends to a well-defined homomorphism
\begin{equation}\label{eqn:CH0CH*Comp}
\gamma_{X, Y^*}:\CH_0(X, Y^*)\to \oplus_{i=0}^{\dim_kX}\CH^i(X(\le i), X^*(\le i))
\end{equation}

\begin{proposition}\label{prop:CH0CH*} Suppose $k$ is an infinite field and $X$ is reduced. Let $Y^*\subset X$ be a closed subset such that $X\setminus Y^*$ is regular  and dense in $X$, and let $X^*(\le i)=(Y^*\cap X(i))\cup X(\le i-1)$. Suppose that  
\begin{equation}\label{eqn:AssumptionCH0CH*}
\codim_{X(i)}X(i)\cap X(j)\ge 2
\end{equation}
for all $0\le i<j\le \dim_kX$. Then the map $\gamma_{X, Y^*}$ in  \eqref{eqn:CH0CH*Comp} is an isomorphism.
\end{proposition}

\begin{proof} Let $d=\dim_kX$. By  Remark~\ref{rem:Comparison2}, we have
\begin{multline*}
\CH_0(X, Y^*)=\CH^d(X, X^*(\le d))\\\oplus 
Z_0(X(\le d-1)\setminus Y^*)/R_0'(X(\le d-1), Y^*\cap X(\le d-1)),
\end{multline*}
where we recall that 
\[
R_0'(X(\le d-1), Y^*\cap X(\le d-1))\subset R_0(X(\le d-1), Y^*\cap X(\le d-1))
\]
 is the subgroup generated by 0-cycles of the form $i_{C*}(\Div f)$, where $i_C:C\to X(\le d-1)$ is a Cartier curve relative to $Y^*\cap X(\le d-1)$, $f\in \sO_{C: C\cap Y^*}^\times$ and $C\cap X(d)=\0$.

Assuming \eqref{eqn:AssumptionCH0CH*}, we claim that 
\begin{equation}\label{eqn:ClaimR'=R}
R_0'(X(\le d-1), Y^*\cap X(\le d-1))=R_0(X(\le d-1), Y^*\cap X(\le d-1)).
\end{equation}
For this, consider a connected Cartier curve $i_C:C\to X(\le d-1)$ on $X(\le d-1)$, relative to $Y^*\cap X(\le d-1)$ and function $f\in \sO_{C: C\cap Y^*}^\times$. By Remark~\ref{rem:CartierCurve}(2), $C$ is supported in $X(i)$ for some $i\le d-1$ and $C\cap \cup_{j=i+1}^{d-1}X(j)=\0$; in particular $C$ is a Cartier curve on $X(\le i)$ and $f\in \sO_{C: C\cap Y^*\cap X(\le i)}^\times$. Moreover $Y^*\cap X(i)\supset X(i)\cap X(\ge i+1)$, and $X(i)\cap X(\ge i+1)$ has codimension $\ge 2$ on  $X(i)$. We apply Lemma~\ref{lem:MovingLemmaR0} (with $X=X(\le i)$, $S=X(i)\cap X(\ge i+1)$ and $Y^*\cap X(\le i)$ replacing $Y^*$). This gives us Cartier curves $\iota_j:C_j\to X(\le i)$ relative to $Y^*\cap X(\le i)$, supported on $X(i)$,  and functions $f_j\in \sO_{C_j: C_j\cap Y^*}^\times$, $j=1,\ldots, s$, with $C_j\cap X(\ge i+1)=\0$ and
\[
i_{C*}(\Div f)=\sum_{j=1}^s i_{C_j*}(\Div f_j).
\]
But then $C_j$ is a Cartier curve on $X(\le d-1)$ relative to $Y^*\cap X(d-1)$, with $C_j\cap X(d)=\0$, hence $i_{C*}(\Div f)$ is in $R_0'(X(\le d-1), Y^*\cap X(\le d-1))$.

As such 0-cycles $i_{C*}(\Div f)$ generate $R_0(X(\le d-1), Y^*\cap X(\le d-1))$, this verifies \eqref{eqn:ClaimR'=R}, and thus 
\[
Z_0(X(\le d-1)\setminus Y^*)/R_0'(X(\le d-1), Y^*\cap X(\le d-1))=\CH_0(X(\le d-1), Y^*\cap X(\le d-1))).
\]
By induction on $d$, 
\[
\CH_0(X(\le d-1), Y^*\cap X(\le d-1)))=\oplus_{i=0}^{d-1}\CH^i(X(\le i), X^*(\le i))
\]
and thus
\[
\CH_0(X, Y^*)=\oplus_{i=0}^d\CH^i(X(\le i), X^*(\le i)).
\]
\end{proof}

\section{Another presentation of $\CH^d(X, X^*)$}\label{sec:Presentation}

In this section, $k$ is an infinite  field.  

Let  $(X, X^*)$ be an admissible pair and let  $d:=\dim_kX$. Given a Cartier curve $C$ on $X$  relative to $X^*$ with $C\subset X_\CM$,  and an $h\in \sO_{C: C\cap X^*}^\times$, taking the closure of the graph of $h$ gives the pure dimension one closed subscheme $W:=\gamma_{C,h}\subset X\times \P^1$, with the following properties: 
\begin{equation}\label{eqn:Properties}
\vbox{
\noindent
i. $W\subset X_\CM\times\P^1$, $p_1:W\to X$ is finite, $W\cap X^*\times\{0,\infty\}=\0$ and $W\cap X^*\times\P^1$ is supported in a finite set.   \\[2pt]
ii. In an open neighborhood of $W\cap X^*\times\P^1$, $W$ is defined by a regular sequence of length $d$.\\[2pt]
iii. $W\cap X\times\{0,\infty\}$ is a closed subscheme of $(X\setminus X^*)\times \{0,\infty\}$, finite over $k$. \\[2pt]
iv. Letting $|W|$ denote the cycle associated to $W$, we have
\[
p_{1*}(|W|\cdot (X  \times \{0\} - X \times \{\infty\})) =i_{C*}(\Div h)
\]
in $Z^d(X\setminus X^*)$.
}
\end{equation}

\begin{definition}\label{def:RelationsII} For an admissible pair $(X, X^*)$, let $\tilde{R}^d(X,X^*)$ be the subgroup of $Z^d(X\setminus X^*)$ generated by cycles of the form
\[
p_{1*}(|W|\cdot (X \times \{0\} - X \times \{\infty\})) \in Z^d(X\setminus X^*),
\]
where $W\subset X\times\P^1$ is a closed subscheme of pure dimension one, satisfying properties \eqref{eqn:Properties}(i.-iii.).
\end{definition}

From \eqref{eqn:Properties},  taking the graph of $f\in \sO_{C: C\cap X^*}^\times$ for a Cartier curve $C$ on $X$ relative to $X^*$ gives an inclusion  $R^d(X,X^*)^\CM\subset \tilde{R}^d(X,X^*)$. By Proposition~\ref{prop:CMMoving}, we have
\[
R^d(X,X^*)\subset \tilde{R}^d(X,X^*).
\]

\begin{remark} \label{rem:CodimNotDim} A closed subscheme $W$ of $X\times\P^1$ of pure codimension $d$  has pure dimension one support in $X(d)\times\P^1$, but could also have dimension zero support in $X(d-1)\times\P^1$. For some arguments later on, it may be useful to index via codimension rather than dimension, and allow closed subschemes $W'\subset X\times\P^1$ of pure codimension $d$, rather than pure dimension one. In this case, we modify \eqref{eqn:Properties}(i,ii) to: \\[5pt]
(i$'$) $W'\subset X\times\P^1$ has pure codimension $d$. We have $W'\subset 
(X_\CM\amalg X(d-1)^0)\times\P^1$,  $W'\cap X^*\times\{0,\infty\}=\0$ and $W'\cap X^*\times\P^1$ is supported in a finite set.\\[2pt]
(ii$'$) In an open neighborhood of $W'\cap (X^*\cap X_\CM)\times \P^1$, $W'$ is defined by a regular sequence of length $d$.
\\[5pt]
and still require  \eqref{eqn:Properties}(iii.).

Given such a $W'$, note that $W'$ is a pure codimension $d$ closed subscheme of the disjoint union $X_\CM\times\P^1\amalg X(d-1)^0\times\P^1$,  the component $W:=W'\cap X(d)\times\P^1=W'\cap X_\CM\times\P^1$ has pure dimension one, satisfies 
\eqref{eqn:Properties}(i.-iii.) and
\[
p_{1*}(|W'|\cdot (X \times \{0\} - X \times \{\infty\}))= p_{1*}(|W|\cdot (X \times \{0\} - X \times \{\infty\}))\in Z^d(X\setminus X^*). 
\]
Thus, this change of conditions does not alter the subgroup $\tilde{R}^d(X,X^*)\subset Z^d(X\setminus X^*)$; compare with Remark~\ref{rem:CartierCurve2}.
\end{remark}

\begin{theorem}\label{thm:Relations} Let $(X, X^*)$ be an admissible pair and let  $d=\dim_kX$. Then we have $R^d(X,X^*)=\tilde{R}^d(X,X^*)$.
\end{theorem}

In case $X$ of pure dimension two and $k$ algebraically closed, this is  \cite[Proposition 1]{BS98}. The proof of {\it loc.\! cit.\!} works just as well  assuming only that $k$ is infinite. The proof given here follows the lines of the argument of {\it loc.\! cit.}; a version for general $d$, roughly along the same lines,  appears in the unpublished work \cite{L85}. The proof proceeds in a series of Lemmas, following \cite[\S3]{BS98}.

\begin{remark}\label{rem:CaseDim1} For an admissible pair $(X, X^*)$ with $X$ of dimension one, Theorem~\ref{thm:Relations}  is classical. If we write $X=X_0\amalg X_1$, it is easy to see that 
Theorem~\ref{thm:Relations} for $(X_1, X^*\cap X_1)$ implies this result for $(X, X^*)$, so we may assume that $X$  has pure dimension one. Then $X_\sing$ is the finite set of singular points of $X$ and $X^*$ is thus any finite set of closed points of $X$  such that $X\setminus X^*$ is regular.  Following Remark~\ref{rem:CH1=Pic}  (going back to \cite[Proposition 1.4]{LW}) we have the isomorphism $\CH^1(X,X^*)\cong \Pic(X)$ sending the class of a 0-cycle  $z=\sum_in_ix_i \in Z^1(X\setminus X^*)$ to the class of the invertible sheaf $\sO_X(z)$; in particular, for each invertible sheaf $\sL$ on $X$, there is a $z\in Z^1(X\setminus X^*)$ with $\sL\cong  \sO_X(z)$.  Now, each  $W\subset X\times\P^1$ satisfying properties \eqref{eqn:Properties}(i.-iii.) is a Cartier divisor, giving the invertible sheaf $\sO_{X\times\P^1}(W)$, and 
we have isomorphisms
\[
\sO_X(W\cdot X\times \{0\})\cong i_0^*\sO_{X\times\P^1}(W),\  \sO_X(W\cdot X\times \{\infty\})\cong i_\infty^*\sO_{X\times\P^1}(W),
\]
where $i_0, i_\infty$ are the closed immersions $X\times\{0\}\hookrightarrow X\times\P^1$, $X\times\{\infty\}\hookrightarrow X\times\P^1$. By the $\P^1$-bundle formula, the projections $p_1:X\times\P^1\to X$, $p_2:X\times\P^1\to \P^1$ induce an isomorphism
\[
\Pic(X)\times \Pic(\P^1)\xrightarrow{p_1^*\otimes p_2^*}\Pic(X\times\P^1)
\]
so $\sO_X(W)\cong p_1^*\sO_X(z)\otimes p_2^*\sO_{\P^1}(m)$ for some $z\in Z^1(X\setminus X^*)$, $m\in \Z$. Thus 
\[
[W\cdot X\times \{0\}]=[z]=[W\cdot X\times \{\infty\}]\in \CH^1(X, X^*)
\]
so $W\cdot (X\times \{0\}-X\times \{\infty\})$ is in $R^1(X,X^*)$ and thus $\tilde{R}^1(X, X^*)\subset R^1(X, X^*)$.

In case $X$ has pure dimension zero, we have   $R^0(X, X^*)=\tilde{R}^0(X,X^*)=\{0\}$, so Theorem~\ref{thm:Relations} holds for all $X$ of dimension $\le 1$.
\end{remark}

Since we are possibly working in positive characteristic, we recall the version of the Bertini theorem that we will need.

\begin{lemma}[Bertini Theorem \hbox{\cite[Cor. 3.4.14]{FOV}}]\label{lem:Bertini} Let $Y\subset \P^N$ be a reduced quasi-projective scheme over an infinite field $k$, and let $\check{\P}^N$ denote the dual projective space of hyperplanes. Then there is a Zariski dense open subset $U\subset \check{\P}^N$ such that for each $H\in U(k)$, the scheme-theoretic intersection $H\cap Y$ is reduced.
\end{lemma}

\begin{remark} Given a quasi-projective $k$-scheme $Y$ and coherent sheaves $\sF_1,\ldots, \sF_m$, when we say that a property $P$ holds for all sufficiently general $(s_1,\ldots, s_m)\in \prod_{i=1}^mH^0(Y, \sF_i)$, we mean there is projective closure $j:Y\hookrightarrow \bar{Y}$ of $Y$, coherent sheaves $\bar{\sF}_1,\ldots, \bar{\sF}_m$ on $\bar{Y}$ with $j^*\bar{\sF}_i\cong \sF_i$ and a Zariski open subset $U\subset \prod_{i=1}^mH^0(\bar{Y}, \bar{\sF}_i)$ (viewing this product as a product of the associated affine spaces) such that $P$ holds for all $(s_1,\ldots, s_m)\in \prod_{i=1}^mH^0(Y, \sF_i)$ with $s_i=j^*\bar{s}_i$ and with $(\bar{s}_1,\ldots, \bar{s}_m)\in U(k)$. 
\end{remark}

\begin{lemma}\label{lem:ClosedImmersion} Let $f:V\to U$ be a finite \'etale morphism in $\Sch_k$ and let $i:Z\to V$ be a closed subscheme of $V$. Let $u\in U$ be a point and let  $\bar{u}:\Spec F\to U$ be geometric point  factoring through $u$. Suppose that the projection $Z\times_U\Spec F\to \Spec F$ is an isomorphism. Then there is a Zariski open neighborhood $U_u$ of $u$ in $U$ such that $Z\times_UU_u\to U_u$ is a closed immersion.
\end{lemma}

\begin{proof} We may assume that $F$ is an algebraic closure $\overline{k(u)}$ of $k(u)$. Let $\sO^{sh}_{U,u}$ be the strict henselization of the local ring $\sO_{U,u}$, giving us the canonical extension of $\bar{u}$ to a morphism  $u^{sh}:\Spec\sO^{sh}_{U,u}\to U$ and the  diagram with all squares cartesian
\[
\xymatrix{
Z\times_U\Spec \overline{k(u)}\ar[d]^{\bar{i}}\ar[r]&Z\times_U\Spec\sO^{sh}_{U,u}\ar[d]^{i^{sh}}\ar[r]&Z\ar[d]^i\\
V\times_U\Spec \overline{k(u)}\ar[d]^{\bar{f}}\ar[r]&V\times_U\Spec\sO^{sh}_{U,u}\ar[d]^{f^{sh}}\ar[r]&V\ar[d]^f\\
\Spec \overline{k(u)}\ar[r]\ar@/_20pt/[rr]_{\bar{u}}&\Spec\sO^{sh}_{U,u}\ar[r]^-{u^{sh}}&U
}
\]
Since $f$ is finite (say of degree $d$)  and \'etale, $V\times_U\Spec\sO^{sh}_{U,u}$ is isomorphic to a finite disjoint union of $d$ copies of $\Spec\sO^{sh}_{U,u}$, with $f^{sh}$ the product of $d$ identity maps
\[
V\times_U\Spec\sO^{sh}_{U,u}=\amalg_{i=1}^d\Spec\sO^{sh}_{U,u}\xrightarrow{\prod_{i=1}^d\id} \Spec\sO^{sh}_{U,u}.
\]
The map $\bar{f}$ breaks up as a disjoint union of identity maps $\Spec \overline{k(u)}\to \Spec \overline{k(u)}$ and the maps $\bar{i}$ and $i^{sh}$ similarly break up into a disjoint union of closed immersions. By assumption, the map $\bar{f}\bar{i}:Z\times_U\Spec \overline{k(u)}\to \Spec \overline{k(u)}$ is an isomorphism, so the map $\bar{i}$ maps $Z\times_U\Spec \overline{k(u)}=\Spec \overline{k(u)}$ to a single component in $V\times_U\Spec \overline{k(u)}$; as each component in $V\times_U\Spec\sO^{sh}_{U,u}$ is local with closed point $V\times_U\Spec \overline{k(u)}\cong \Spec \overline{k(u)}$, this implies that $i^{sh}:Z\times_U\Spec\sO^{sh}_{U,u}\to V\times_U\Spec\sO^{sh}_{U,u}$ is a closed immersion of $Z\times_U\Spec\sO^{sh}_{U,u}$ into a single component of $V\times_U\Spec\sO^{sh}_{U,u}$. Thus $f^{sh}\circ i^{sh}:Z\times_U\Spec\sO^{sh}_{U,u}\to \Spec\sO^{sh}_{U,u}$ is a closed immersion.

On the other hand, the map $f\circ i:Z\to U$ is a finite morphism. Let $\sK$ be the cokernel of the canonical map $\sO_U\to  (f\circ i)_*\sO_Z$. Then $f\circ i$ is a closed immersion over some open neighborhood $U_u$ of $u$ if and only if  $\sK|_{U_u}=0$.  Since $f^{sh}\circ i^{sh}$ is a closed immersion, we have $\sK\otimes_{\sO_U}\sO^{sh}_{U,u}=0$ and since $u^{sh}$ is flat, this implies that $\sK\otimes_{\sO_U}\sO_{U,u}=0$. Since $\sK$ is a coherent sheaf, there is thus an open neighborhood $U_u$ of $u$ such that $\sK|_{U_u}=0$, and thus $f\circ i$ is a closed immersion over $U_u$.
\end{proof}

For $V$ a finite dimensional $k$-vector space of dimension $n$, we have the affine scheme $\V(\Sym^*V^\vee)\cong \A^n$; we will sometimes consider $V$ as this affine space, speaking for example of a Zariski open subset of $V$ when we mean a Zariski open subset of $\V(\Sym^*V^\vee)$.

Given an admissible pair  $(X, X^*)$, let $X^*(d-1)$ be the union of the irreducible components of $X^*$ that are of dimension $d-1$ and contained in $X(d)$.

\begin{lemma}[cf. \hbox{\cite[Lemma 3]{BS98}}] \label{lem:GenPos} Let $(X, X^*)$ be an admissible pair, suppose that $X$ is connected and $d:=\dim_kX\ge 2$. Let $W\subset X\times \P^1$ be a pure dimension one closed subscheme  satisfying properties \eqref{eqn:Properties}(i.-iii.); we assume in addition that each irreducible component of $W$ has non-empty intersection with $X^*\times\P^1$.   Let $\hat{W}=p_1(W)$.  Fix a (locally closed) projective embedding  $i:X\times\P^1\to \P^N$, with corresponding very ample invertible sheaf $\sO_{X\times\P^1}(1)$ on $X\times\P^1$.
Then for  $0\ll m_1\ll m_2\ll\ldots\ll m_d$ sufficiently large, and $g_1,\ldots, g_d\in \prod_iH^0(X\times\P^1, \sI_W(m_i))$ sufficiently general, we have 
\begin{equation}\label{eqn:GenPos}
V(g_1,\ldots, g_d)\subset X_\CM\times \P^1\amalg X(d-1)^0\times\P^1.
\end{equation}
Let $H:=V(g_1)\cap X(d)^0\times \P^1$, let  $V(g_1,\ldots, g_d)_H:=V(g_1,\ldots, g_d)\cap H$ and let $E$ be the closure of $V(g_1,\ldots, g_d)_H\setminus W$ in $H$. We  have in addition:\\[5pt]
i.  $H$ and $E$ are both reduced, $H$ is a Cartier divisor on $X(d)^0\times\P^1$. Both  $V(g_1,\ldots, g_d)_H$ and  $E$ are  closed subschemes of  pure codimension $d$ in $X\times \P^1$.\\[2pt]
ii. For each $i=1,\dots, d$, $V(g_1,\ldots, g_i)$ has pure codimension $i$ on $X\times\P^1$, $g_1,\ldots, g_i$ form a regular sequence at each point of $V(g_1,\ldots, g_i)\cap X(d)^0\times \P^1$ and $V(g_1,\ldots, g_d)=W$ in a neighborhood of $W\cap X^*\times\P^1$.\\[2pt]
iii. $E\cap (X^*\times\{0,\infty\})=\0$, $E\cap W\cap  (X^*\times\P^1)=\0$ and $E$ and $W$ have no common components.\\[2pt]
iv. Let $p_1:X\times\P^1\to X$  be the first projection and let $\hat{W}=p_1(W_\red)$. Then we have
\begin{enumerate}
\item[(a)] $p_1:V(g_1)\to X$ is finite. There is an open subscheme $U_1$ of $X(d)^0\setminus \hat{W}$ that contains   each generic point of $X^*\cap X(d)^0\setminus \hat{W}$ and such that $p_1:H\to X(d)^0$ is \'etale over $U_1$.
\item[(b)]   There is an open subscheme $U_2$ of $U_1$,   containing each generic point of $X^*(d-1)$, such that,  for $i=2,\ldots, d$,  $V(g_1,\ldots, g_i)\cap U_2\times\P^1$ is dense in $V(g_1,\ldots, g_i)\cap  X(d)^0\times\P^1\setminus W)$ and  $p_1: V(g_1,\ldots, g_i)\cap (U_2\times\P^1)\to U_2$ is a closed immersion.
\end{enumerate}
v.  $E\cap (X^*\times\P^1)\subset U_2\times\P^1$.\\[2pt]
vi. $E$ is a local complete intersection in $H$ in a neighborhood of $E\cap (X^*\times\P^1)$.\\[2pt]
vii. Let $C:=V(g_1,\ldots, g_d)_H$. Then $C$ is a local complete intersection in $H$, and we have the identity of 0-cycles on $X\times\{0,\infty\}$.
\[
[|C|\cdot (X\times\{0,\infty\})]=[|W|\cdot(X\times\{0,\infty\})]+[|E|\cdot (X\times\{0,\infty\})]\\
\]
\end{lemma}

\begin{proof} This closely follows the proof of \cite[Lemma 3]{BS98}. 

It follows from property \eqref{eqn:Properties}(i) that $W\subset X_\CM\times\P^1\subset X(d)^0\times\P^1$; since each irreducible component of $W$ intersects $X^*\times\P^1$ in a finite set, $W$ contains no fibers $\{x\}\times\P^1$, $x\in X$. Also, $\hat{W}\subset X_\CM$ and $\hat{W}\cap X^*$ is finite.

 As in the proof of Lemma~\ref{lem:MovingLemmaR0}, we can assume that  $V(g_1,\ldots, g_d)\cap U=W\cap U$ for some neighborhood $U$ of $W\cap X^*\times\P^1$ in $X_\CM\times\P^1$ (depending on $g_1,\ldots, g_d$), and $V(g_1,\ldots, g_d)\cap T=\0$ for $T$ a  given closed subset of $X\times\P^1\setminus U$ of dimension $\le d-1$.  As 
 \[
 \dim_k(X\setminus (X_\CM\times\P^1\amalg X(d-1)^0\times\P^1))\le d-1
 \]
 this  proves \eqref{eqn:GenPos}.

Let $\{x_1,\ldots, x_t\}=W_\red\cap X^*\times\P^1$. At each $x_i$,   $\sI_{W,x_i}/\sI_{W,x_i}^2$ is a $k(x_i)$-vector space of dimension $d$. 

For (i), consider the blow-up $\mu:(X\times\P^1)_W\to X\times\P^1$ with exceptional (Cartier) divisor $D$. The invertible sheaf $\sO_{(X\times\P^1)_W}(-D)$ is relatively very ample, so  $\mu^*\sO_{X\times\P^1}(m_1)(-D)$ is very ample on $(X\times\P^1)_W$ for $m_1\gg0$. We take a corresponding locally closed embedding $(X\times\P^1)_W\subset \P^M$ with $\sO_{(X\times\P^1)_W}(1)=\mu^*\sO_{X\times\P^1}(m_1)(-D)$.

Moreover, since $W$ is a complete intersection in a neighborhood of each $x_i$, the map $D\to W$ is a $\P^{d-1}$-bundle over a neighborhood of each $x_i$, with fibers embedded as linear subspaces of $\P^M$. Letting $D_{x_i}\cong \P^{d-1}$ denote the fiber over $x_i$, we have
\[
\mu^*\sO_{X\times\P^1}(m_1)(-D)\otimes \sO_{D_{x_i}}\cong \sO_{\P^{d-1}}(1)
\]
and the restriction map 
\[
H^0(\P^M, \sO_{\P^M}(1))\to H^0(D_{x_i}, \sO_{D_{x_i}}(1))
\]
is surjective (we trivialize $\sO_{X\times\P^1}(m_1)$ in a neighborhood of $x_i$ to avoid writing the twist). Identifying $H^0(D_{x_i}, \sO_{D_{x_i}}(1))$ with $\sI_{W,x_i}/\sI_{W,x_i}^2\otimes_{\sO_W}k(x_i)$,  and recalling that 
\[
\mu_*(\mu^*\sO_{X\times\P^1}(m_1)(-D))=\sI_W(m_1)
\]
we see that there is a dense open subset $\tilde{V}$ of $H^0(\P^M, \sO_{\P^M}(1))$ such that for $y\in\tilde{V}(k)$, the  corresponding global section $g_y$ of  $\sI_W(m_1)$ restricts to a non-zero element of $\sI_{W,x_i}/\sI_{W,x_i}^2\otimes_{\sO_W}k(x_i)$ for each $i=1,\ldots, r$.

In what follows, for $s_1,\ldots, s_j\in H^0(\P^N, \sO_{\P^N}(m))$, we write $V(s_1,\ldots, s_j)$ for the closed subscheme of $X\times\P^1$ defined by the restrictions of the $s_i$; similarly, for sections  $s_1,\ldots, s_j\in H^0(\P^M, \sO_{\P^M}(r))$, we have $V(s_1,\ldots, s_j)\subset (X\times\P^1)_W$.

 Note that $(X(d)^0\times\P^1)_W\subset (X\times\P^1)_W$ is reduced (see \cite[\href{https://stacks.math.columbia.edu/tag/0808}{Tag 0808}]{stacks-project}).    By the Bertini theorem (Lemma~\ref{lem:Bertini}), there is  a dense open subset of $H^0(\P^M, \sO_{\P^M}(1))$ whose $k$-points $y$ give sections $\tilde{g}_y$ of $(\mu^*\sO_{X\times\P^1}(m_1))(-D)$ such that $V(\tilde{g}_y)\cap (X(d)^0\times\P^1)_W$ is reduced. By the above, we may take $y$ so that the corresponding global section $g_1:=g_y$ of  $\sI_W(m_1)$ also restricts to a non-zero element of $\sI_{W,x_i}/\sI_{W,x_i}^2\otimes_{\sO_W}k(x_i)$ for each $x_i$. By our assumption on $W$, we see that $X(d)^0\times\P^1$ is Cohen-Macaulay in a neighborhood of each $x_i$, and hence by unmixedness, $V(g_1)\cap X(d)^0\times\P^1$ is reduced outside of a closed subset $W_0$ of $W$ contained in $W\setminus X^*\times\P^1$. Since $W$ has dimension one,  and each irreducible component of $W$ intersects $X^*\times\P^1$, $W_0$ is a finite set of points of $(X(d)^0\setminus X^*)\times\P^1\subset X_\reg(d)\times\P^1$. Thus, for each $x\in W_0$,   $\sO_{X\times\P^1,x}$ is a regular local ring of dimension $d+1$, hence Cohen-Macaulay of dimension $>2$, so $V(g_1)$ is reduced at $x\in W_0$ as well. Thus taking $H=V(g_1)\cap X(d)^0\times\P^1$, $H$ is thus a reduced Cartier divisor on $X(d)^0\times\P^1$.
 
We have $(W\cap X^*\times\P^1)\cup (X\setminus X^*)\times\P^1\subset X_\CM\times\P^1$ As mentioned in Remarks~\ref{rem:Prelims}, $X_\CM$ is an open subset of $X(d)^0$ and of $X$, and  the complement of $X_\CM$ in $X(d)$ has codimension at least two. 

Let $H_\CM$ be the Cohen-Macaulay locus in $H$. Since $g_1$ is a non-zero divisor on $X(d)^0\times\P^1$, we have
\[
H_\CM\supset  H\cap X_\CM\times\P^1\supset (W\cap X^*\times\P^1)\cup (H\cap (X\setminus X^*)\times\P^1).
\]

Let  $\overline{X\times\P^1}$ and $\bar{W}\subset \overline{X\times\P^1}$  be the closures of $X\times\P^1$ and $W$ in $\P^N$, and let $\sI_{\bar{W}}\subset \sO_{\overline{X\times\P^1}}$ be the ideal sheaf.  It follows by an easy induction that for $m_1\ll m_2\ll\ldots\ll m_d$, there is a dense open subset 
\[
U\subset \prod_{j=1}^dH^0(\overline{X\times\P^1}, \sI_{\bar{W}}(m_1))\times\ldots\times H^0(\overline{X\times\P^1}, \sI_{\bar{W}}(m_d))
\]
such that, for $(g_1,\ldots, g_d)\in U(k)$, we have:\\[5pt]
a. \eqref{eqn:GenPos} holds.\\[2pt]
b. For each $i=1,\ldots, d$,  $g_1,\ldots, g_i$ restricts to a regular sequence on $X(d)^0\times\P^1$, $V(g_1,\ldots, g_i)\cap X(d)^0\times\P^1$ is a reduced closed subscheme of $X(d)^0\times\P^1$, of pure codimension $i$, and $V(g_1,\ldots, g_i)\cap(X_\CM \times\P^1)$ is Cohen-Macaulay.  \\[2pt]
c.  $g_1,\ldots, g_d$ generate $\sI_{W, x}$ at each point $x\in W\cap X^*\times\P^1$.\\[2pt]
\\[2pt]
By \eqref{eqn:GenPos} the closed subscheme $V(g_1,\ldots, g_d)$ of $X\times\P^1$ is a disjoint union of two closed subschemes $V(g_1,\ldots, g_d)_H\subset X_\CM\times\P^1$ and $V(g_1,\ldots, g_d)\cap (X(d-1)^0\times\P^1)$, hence   $V(g_1,\ldots, g_d)_H$ is also closed in $X\times\P^1$. Since $E$ is closed in $V(g_1,\ldots, g_d)_H$, this implies that $E$ is also closed in $X\times\P^1$. Together with (a)-(c), this  implies (i) and (ii). 

In addition, by (c), there is an open neighborhood $V$ of $W\cap X^*\times\P^1$ in $X\times\P^1$ such that $V(g_1,\ldots, g_d)\cap V=W\cap V$. This shows that $E\cap W\cap X^*\times\P^1=\0$. If we now consider $X^*\times\{0,\infty\}$, since $W\cap X^*\times\{0,\infty\}=\0$, we may assume that 
\[
\codim_{X^*\times\{0,\infty\}}H\cap X^*\times\{0,\infty\}\ge 1
\]
and then by induction on $i$, that 
\begin{equation}\label{eqn:CodimXsing}
\vbox{
$\codim_{X^*\times\{0,\infty\}}V(g_1,\ldots, g_i)\cap X^*\times\{0,\infty\}\ge i$ for all $i=1,\ldots, d$.
}
\end{equation}
 Taking $i=d$, and noting that $\dim X^*\le d-1$, we see that $V(g_1,\ldots, g_d)\cap X^*\times\{0,\infty\}=\0$. Since $E\cap W\cap X^*=\0$ and each component of $W$ intersects $X^*\times\P^1$, we see that $E$ and $W$ have no common components,  and we have proven (iii).

For (iv)(a), consider the curves $\{x\}\times\P^1\subset \overline{X\times\P^1}$, $x\in X$. Since $X$ is connected, there is an integer $r\ge1$ such that, for all $x\in X$,  $\sO_{\overline{X\times\P^1}}(1)$ restricted to $\{x\}\times\P^1=\P^1_x$ is isomorphic to $\sO_{\P^1_x}(r)$, that is, the curves $\{x\}\times\P^1$ are all curves of degree $r$ in $\P^N$. 

There is an integer $m_0>0$ such that  for all $x\in X$, and all $m\ge m_0$ the restriction map
\[
H^0(\overline{X\times\P^1}, \sI_{\bar{W}}(m))\to H^0(\{x\}\times\P^1, \sI_W\otimes_{\sO_{X\times\P^1}}\sO_{\{x\}\times\P^1}(rm))
\]
is  surjective, and $H^1(\{x\}\times\P^1, \sI_W\otimes_{\sO_{X\times\P^1}}\sO_{\{x\}\times\P^1}(rm))=0$. Choosing a section $s\in H^0(\{x\}\times\P^1, \sO_{\{x\}\times\P^1}(r))$ with $V(s)\cap W=\0$, we have the exact sequence
\[
0\to \sI_W\otimes_{\sO_{X\times\P^1}}\sO_{\{x\}\times\P^1}(rm)\xrightarrow{\times s}
\sI_W\otimes_{\sO_{X\times\P^1}}\sO_{\{x\}\times\P^1}(r(m+1))\to 
\sO_{V(s)}\to 0, 
\]
giving 
\begin{multline*}
\dim_kH^0(\{x\}\times\P^1, \sI_W\otimes_{X\times\P^1}\sO_{\{x\}\times\P^1}(rm))\\=
H^0(\{x\}\times\P^1, \sI_W\otimes_{\sO_{X\times\P^1}}\sO_{\{x\}\times\P^1}(rm_0))+(m-m_0)r
\end{multline*}
for all $m\ge m_0$. In particular, for all integers $m_1> m_0+\dim_kX/r$, we have
\[
\dim_kH^0(\{x\}\times\P^1, \sI_W\otimes_{X\times\P^1}\sO_{\{x\}\times\P^1}(rm_1))>\dim_kX
\]
We have the constructible subset
\[
F_0^0:=\{g\in H^0(\overline{X\times\P^1}, \sI_{\bar{W}}(m_1))\setminus\{0\}\mid \exists x\in X\text{ with } V(g)\supset \{x\}\times\P^1\}
\]
of $H^0(\overline{X\times\P^1}, \sI_{\bar{W}}(m_1))$. By the theorem on the dimension of fibers \cite[\S 6.3, Theorem 1.25]{Shafarevich}, the closure $F_0$ of $F^0_0$
has codimension at least one in $H^0(\overline{X\times\P^1}, \sI_{\bar{W}}(m_1))$. This gives us the dense open subset
\[
V_0= H^0(\overline{X\times\P^1}, \sI_{\bar{W}}(m_1))\setminus F_0\subset H^0(\overline{X\times\P^1}, \sI_{\bar{W}}(m_1)), 
\]
such that,  for all $g\in V_0(k)$, the projection $p_1:V(g)\to X$ is quasi-finite and proper, hence finite. This proves the first assertion in (iv)(a).

For the rest of (iv)(a),  if we take $x\in X\setminus \hat{W}$, then $\sI_W\otimes\sO_{\{x\}\times\P^1}(m_1)=\sO_{\{x\}\times\P^1}(m_1)$,  and the restriction map
\[
H^0(\overline{X\times\P^1}, \sI_{\bar{W}}(m_1))\to H^0(\{x\}\times \P^1, \sO_{\{x\}\times\P^1}(rm_1))
\] 
is surjective. For each irreducible component $X^*_i$ of $X^*\cap X(d)^0\setminus \hat{W}$, choose a closed point $x_i\in X^*_i\setminus W$, not lying in $X^*_j$ for each $j\neq i$. Increasing $m_1$ if necessary, the restriction map
\[
H^0(\overline{X\times\P^1}, \sI_{\bar{W}}(m_1))\to \oplus_i H^0(\{x_i\}\times \P^1, \sO_{\P^1_x}(rm_1))
\] 
is still surjective.

Let $X_0, X_1$ be  the standard coordinates on $\P^1$, identifying $H^0(\P^1_x, \sO_{\P^1_x}(rm_1))$ with homogeneous polynomials of degree $m_1r$ in $X_0, X_1$.  Then the subset of $H^0(\P^1_x, \sO_{\P^1_x}(rm_1))$ consisting of polynomials $f(X_0, X_1)$ with $m_1r$ distinct roots (in a suitable algebraic closure) is open and dense. Choose for each irreducible component $X^*_j$ of $X^*\cap X(d)^0\setminus \hat{W}$ a closed point $y_j\in X_j^*\cap X(d)^0\setminus \hat{W}$.  Then there is a dense open subset $V$ of $H^0(\overline{X\times\P^1}, \sI_{\bar{W}}(m_1))$ such that, for each $g_1\in V(k)$, $V(g_1)\cap X(d)^0\times\P^1$ is reduced  Cartier divisor on $X(d)^0\times\P^1$, $p_1:V(g_1)\to X$ is finite and the restriction of $g_1$ to $\{y_j\}\times\P^1$ has $m_1r$ distinct roots for each $y_j$. This implies that $p_1:V(g_1)\to X$ is \'etale over an open subscheme $U_1$ of $X(d)^0$ containing each $y_j$, hence containing each generic point of $X^*\cap X(d)^0\setminus \hat{W}$.  Since $H=V(g_1)\cap X(d)^0\times\P^1$, this completes the proof of (iv)(a).

For (iv)(b), we may remove from $X$ any closed subset $T$ such that $T\cap W\cap X^*=\0$ and $T$ contains no component of $X^*(d-1)$. In particular, we   may assume that $X=X(d)=X(d)^0$ and $X^*=X^*(d-1)$. Since $d\ge2$, $X^*$ has no zero-dimensional components, and no generic point of $X^*$ is in $W$. We may further assume that $X=X_\CM$ and $X\setminus W=U_1$. Taking $(g_1,\ldots, g_d)$ as above, we have $H=V(g_1)$,   $H\subset X\times\P^1$ is finite over $X$, and  $H$ is \'etale over $U_1=X\setminus W$.

It suffices to prove the case $i=2$. Indeed, assuming the case $i=2$,  and given $i$, $2<i\le d$, $V(g_1,\ldots, g_i)\cap U_2\times \P^1$ is a closed subscheme of $V(g_1, g_2)\cap U_2\times\P^1$, so $p_1:V(g_1,\ldots, g_i)\cap U_2\times \P^1\to U_2$ is still a closed immersion. As the base-locus of $H^0(\overline{X\times\P^1}, \sI_{\bar{W}}(m_j))$ is $\bar{W}$ (for all $m_j$ sufficiently large),  we need only take $g_3,\ldots, g_i$ sufficiently general so that   $V(g_1,\ldots, g_i)\cap X\times \P^1\setminus W$ intersects each component of $(X\setminus U_2)\times\P^1\setminus W$ properly, and (iv)(b) is satisfied. 

For the case $i=2$, take $m_2\gg m_1$, chosen so that the restriction map  
\begin{equation}\label{eqn:ResSurjectivity}
H^0(\overline{X\times\P^1}, \sI_W(m_2))\to H^0(\{u\}\times\P^1, \sO_{\{u\}\times\P^1}(m_2r))
\end{equation}
is surjective for all $u\in U_1$. 

Let $M=\dim_kH^0(\overline{X\times\P^1}, \sI_{\bar{W}}(m_2))-1$, giving us the corresponding linear system $\P^M$ of divisors  on $\overline{X\times\P^1}$. Inside $\P^M$ we have the dense open subset $\P^M_0$ consisting of points $[D]$ such that the corresponding divisor $D$  does not contain any fiber $\{x\}\times\P^1$, $x\in X$. Let $I\subset \overline{X\times\P^1}\times \P^M$ be the incidence correspondence, and let $I'=I\cap (U_1\times \P^1\cap H)\times\P^M_0\subset U_1\times\P^1\times\P^M_0$. Since $H\to X(d)^0$ is finite, the map $p_{13}:I'\to U_1\times\P^M_0$ is finite, giving us the reduced closed subscheme $P:=p_{13}(I')\subset U_1\times\P^M_0$, with geometric points
\[
\{(u, [D])\in U_1\times \P^M_0\mid D\cap H\cap \{u\}\times\P^1\neq\0\}.
\]
Let $P'\subset P$ be the subset  with geometric points
\[
\{(u, [D])\mid \deg_{k(u)}(D\cap H\cap \{u\}\times\P^1)\ge 2\};
\]
the upper-semi-continuity of dimension for the proper pushforward of coherent sheaves \cite[Theorem 12.8]{Hartshorne} implies that $P'$ is a closed subset of $P$. We consider $P'$ as a reduced closed subscheme of $P$. For  $T\subset U_1$ a locally closed subscheme, let $P(T)=(P\cap T\times\P^M_0)_\red$, $P'(T)=(P'\cap T\times\P^M_0)_\red$. Since $H\to X(d)^0$ is finite and \'etale of degree $m_1r$ over $U_1$, the surjectivity of \eqref{eqn:ResSurjectivity} tell us that, for each $u\in U_1$, $P(u)$ is the union of $m_1r$ distinct codimension one linear subspaces of $\P^M_{\overline{k(u)}}$, intersected with $\P^M_{0, k(u)}$, and $P'(u)\subset P(u)$ is a similar union of codimension two linear subspaces of $\P^M_{\overline{k(u)}}$, intersected with $\P^M_{0, k(u)}$. Thus $\dim_{k(u)}P(u)=M-1$ or $P(u)=\0$, and $\dim_{k(u)}P'(u)=M-2$ or $P'(u)=\0$. 

We apply the theorem on the dimension of fibers  \cite[\S 6.3, Theorem 1.25]{Shafarevich} to the projection to $U_1$, which tells us that, for $T\subset U_1$ an closed subset of pure dimension $e$,  $P(T)$ has dimension $e+M-1$ (or is empty), and  $P'(T)$ has dimension $e+M-2$ (or is empty). Taking the projection to $\P^M_0$, this implies there is a dense open subscheme $V_T\subset \P^M_0$ such that the projection $P(T)\to \P^M_0$ has fiber dimension  $e-1$ over $V_T$ (or has all fibers  over $V_T$ empty), and $P'(T)\to \P^M_0$ has fiber dimension  $e-2$ over $V_T$ (or has all fibers  over $V_T$ empty). 

We apply this to $T$ an irreducible component of $X^*\cap U_1$.  We choose $g_2\in H^0(\overline{X\times\P^1}, \sI_{\bar{W}}(m_2))\setminus\{0\}$ so that  $[V(g_2)]\in \P^M$ is in $V_T\subset \P^M_0$. Let $Z=V(g_1, g_2)\subset X\times\P^1$, and let $\bar{Z}$ be the image of $Z$ in $X$. Suppose $\bar{Z}\cap T$ is not empty. Since $T$ has dimension $d-1$,  it follows that $Z \cap T\times\P^1$ and $\bar{Z}\cap T$ both have pure dimension $d-2$, and there is a  closed subset $T'$ of $\bar{Z}\cap T$ of dimension $\le d-3$ such that $\bar{u}\times\P^1\cap Z$ is a single point  for each geometric point $\bar{u}\to \bar{Z}\setminus T'$. Since $V(g_1)\to X$ is \'etale and finite over $U_1$, we map apply Lemma~\ref{lem:ClosedImmersion} to find an open subset $U_T$ of $U_1$ with $U_T\cap T=T\setminus T'$, such that $V(g_1, g_2)\to X$ is a closed immersion over $U_T$. Taking $U_2$ to be the union of the $U_T$ over the irreducible components of $X^*\cap U_1$ proves (iv)(b).
 
 For (v), let $F=(X^*\setminus \hat{W})\setminus U_1$, so by (iv)(a), $F$ contains no generic point of $X^*$, in particular each component of $F$ has dimension $\le d-2$. Letting $F_1=F\cup (X^*\cap \hat{W})$,    $F_1\times\P^1$ contains $W\cap X^*\times\P^1$ and each component of $F_1$ has dimension $\le d-2$. Taking the $m_i$ sufficiently large, we may assume that the base-locus of the linear system associated to $H^0(\overline{X\times\P^1}, \sI_{\bar{W}}(m_i))$ is exactly $\bar{W}$, so by taking $g_1,\ldots, g_d$ sufficiently general, we have that
 \[
 \codim_{F_1\times \P^1\setminus W}(V(g_1,\ldots, g_i)\cap (F_1\times \P^1\setminus W))\ge i
 \]
 for $i=1,\ldots, d$, hence
 \[
 (V(g_1,\ldots, g_d)\setminus W)\cap F_1\times \P^1=\0.
 \]
 As $V(g_1,\ldots, g_d)\setminus W=E\setminus E\cap W$ and  $E\cap W\cap X^*\times\P^1=\0$ by (iii), we see that
 \[
 E\cap X^*\times\P^1\subset U_1\times\P^1,
 \]
hence $p_1:H\to X$ is \'etale at each point of $E\cap X^*\times\P^1$. 

Similarly, let $F_2=X^*\cap X(d)^0\setminus U_2$. Then since $U_2 \cap X^*(d-1)$ is dense in $X^*(d-1)$, we see that each component of $F_2$ has dimension $\le d-2$, and thus $E\cap X^*\times\P^1\subset U_2\times\P^1$ for all $g_1,\ldots, g_d$ sufficiently general.

The assertion (vi) follows from (ii), (iii) and the fact that $V(g_1,\ldots, g_d)_H=W\amalg E$ in a neighborhood of $X^*\times \P^1$. To show (vii), we have $V(g_1,\ldots, g_d)_H=W\cup E$ (up to possible 0-dimensional embedded components) and $W$ and $E$ have no common irreducible components, which
 implies the identity of associated cycles
\[
|C|=|W|+|E|.
\]
Since the intersection 0-cycle of $C$, resp. $W$, resp. $E$ with the Cartier divisor $X\times0-X\times\infty$ on $X\times\P^1$ only depends on $|C|$, resp. $|W|$, resp. $|E|$, (vii) follows.
\end{proof}

\begin{proof}[Proof of \hbox{Theorem~\ref{thm:Relations}}] We may assume that $X$ is connected, and by Remark~\ref{rem:CaseDim1}, we may assume that $X$ has dimension $d\ge2$.

Let $W\subset X\times\P^1$ be one of the generating curves for $\tilde{R}^d(X,X^*)$. Let $W''$ be the union of the irreducible components of $W$ that have empty intersection with $X_\sing\times\P^1$, and $W'$ those with non-empty intersection. Let $\hat{W}''=p_1(W'')$, and let
\[
t:=\Nm_{W''/\hat{W}''}(X_1/X_0)
\]
where $X_0, X_1$ are the standard homogeneous coordinates on $\P^1$, with $\Div(X_1/X_0)=0-\infty$. Standard arguments (see, e.g., \cite[Proposition 1.4]{Fulton}) show that
\[
\Div(t)=p_{1*}(|W''|\cdot (X\times\{0\}-X\times\{\infty\}))
\]
and since $\hat{W}''\cap X^*=\0$, we see that 
$p_{1*}(|W''|\cdot (X\times\{0\}-X\times\{\infty\}))$ is in $R^d(X,X^*)$. Thus we may assume from the start that each irreducible component of $W$ intersects $X^*\times\P^1$, putting us in the setting covered by Lemma~\ref{lem:GenPos}; We retain the notations used in the proof of that lemma.

Take $0\ll m_1\ll\ldots\ll m_d$ and $g_i\in H^0(X\times\P^1, \sI_W(m_i))$, to satisfy the conclusions of Lemma~\ref{lem:GenPos}; in particular, we have the reduced closed curve $E\subset X\times\P^1$, supported in $X_\CM\times\P^1$, with $V(g_1,\ldots, g_d)_H=E\cup W$ (up to 0-dimensional embedded components). Since the linear systems defined by $H^0(X\times\P^1, \sI_W(m_i))$
are base-point free on $X\times\P^1\setminus W$, and $W\subset X_\CM\times\P^1$,  we may assume in addition that\\[5pt]
a. $V(g_1,\ldots, g_{d-1})\cap X\times\{0,\infty\}$ is a disjoint union of  closed subschemes
\[
V(g_1,\ldots, g_{d-1})\cap X\times\{0,\infty\}=E_0\amalg E_\infty\amalg E^*
\]
with $E_0$ a local complete intersection supported in $X_\CM\times\{0\}$, $E_\infty$ a local complete intersection supported in $X_\CM\times\{\infty\}$ and $E^*$ supported in a finite set contained in $X(d-1)^0\times \{0,\infty\}$.   \\[2pt]
b. $V(g_1,\ldots, g_{d-1})\cap X^*\times\{0,\infty\}$ is a finite set of points\\[5pt]
Taking $g_d'$ to be a sufficiently general element of $H^0(X\times\P^1, \sO_{X\times\P^1}(m_d))$, by the same arguments we may assume that\\[5pt]
c. $V(g_1,\ldots, g_{d-1}, g_d')$ is a disjoint union of closed subschemes 
\[
V(g_1,\ldots, g_{d-1}, g_d')=V(g_1,\ldots, g_{d-1}, g_d')_\CM\amalg V(g_1,\ldots, g_{d-1}, g_d')_{d-1}
\]
with $V(g_1,\ldots, g_{d-1}, g_d')_\CM$ a local complete intersection supported in $X_\CM\times\P^1$ and $V(g_1,\ldots, g_{d-1}, g_d')_{d-1}$ supported in a finite subset of $X(d-1)^0\times\P^1$.\\[2pt]
d. $V(g_1,\ldots, g_{d-1}, g_d')_\CM\cap X^*\times\P^1$ is a finite set of points.\\[2pt]
e. $V(g_1,\ldots, g_{d-1}, g_d')_\CM\cap X\times\{0,\infty\}$ is supported in $(X\setminus X^*)\times\{0,\infty\}$ and is finite over $k$.\\[5pt]
In addition:\\[5pt]
f. If $d=2$,  then $p_1:V(g_1,\ g_2')\to X(2)^0$ is a closed immersion over some open neighborhood of $p_1(V(g_1, g_2'))\cap X^*\cap X(2)^0$. \\[5pt]
The proof of (f) follows from the same argument used to show that $p_1:V(g_1, g_2)\to X$ is a closed immersion over the open subset $U_2$ given in Lemma~\ref{lem:GenPos}, where the only difference is that $g_2$ was a sufficiently general section of $H^0(X\times\P^1, \sI_W(m_2))$ rather than one of $H^0(X\times\P^1, \sO_{X\times\P^1}(m_2))$.  

As is standard, for a finite morphism $t:C\to \P^1$, we set $\Div_0(t)=t^*(0)$, $\Div_\infty(t)=t^*(\infty)$ and $\Div(t)=\Div_0(t)-\Div_\infty(t)$. 

Let $E':=V(g_1,\ldots, g_{d-1}, g_d')_\CM$. Let $\hat{E}'=p_1(E')$ and let $f'=\Nm_{E'/\hat{E}'}(X_1/X_0)$. 

By Lemma~\ref{lem:GenPos},  $p_1:E'\to \hat{E}'$ is an isomorphism over a neighborhood of $\hat{E}'\cap X^*$, so it follows from (a)-(f)  that $\hat{E}'\cap X^*$ is a finite set of points, all contained in $X_\CM$, $\hat{E}'$ is a complete intersection  on $X$ in a neighborhood of $\hat{E}'\cap X^*$, and $f'$ is in $\sO_{\hat{E}': \hat{E}'\cap X^*}^\times$. Thus  $\hat{E}'$ is a Cartier curve on $X$ relative to $X^*$, and we have the element $i_{\hat{E}'*}(\Div f')\in R^d(X, X^*)$.

The curve $E_0$ is supported in $X(d)=X(d)\times\{0\}$, satisfies the conditions of Definition~\ref{def:CartierCurve}, and has empty intersection with $X(\le d-1)$, hence $E_0$ is a Cartier curve on $X$ relative to $X^*$. Let $f_0$ be the restriction of $g_d/g_d'$ to $E_0$. Then $f_0$ is in $\sO_{E_0: E_0\cap X^*}^\times$, giving the element $i_{E_0*}(\Div f_0)\in R^d(X, X^*)$.
Similarly,  $E_\infty$ is a Cartier on $X$ relative to $X^*$, and  letting $f_\infty$ be the restriction of $g_d/g_d'$ to $E_\infty$, $f_\infty$ is in $\sO_{E_\infty: E_\infty\cap X^*}^\times$ and we have the element $i_{E_\infty*}(\Div f_\infty)\in R^d(X, X^*)$.

Let $\hat{E}=p_1(E)\subset X_\CM$ and let $f=\Nm_{E/\hat{E}}(X_1/X_0)$. By Lemma~\ref{lem:GenPos} again,  $\hat{E}$ is a Cartier curve on $X$ relative to $X^*$, and $f$ is in $\sO_{\hat{E}: \hat{E}\cap X^*}^\times$, so we have the element $i_{\hat{E}*}(\Div f)\in R^d(X, X^*)$.

Finally, we have
\begin{align*}
&(|W|+|E|)\cdot X\times\{0\}=i_{E_0*}(\Div_0(f_0))\\
&(|W|+|E|)\cdot X\times\{\infty\}=i_{E_\infty*}(\Div_0(f_\infty))\\
&|E'|\cdot X\times\{0\}=i_{E_0*}(\Div_\infty(f_0))\\
&|E'|\cdot X\times\{\infty\}=i_{E_\infty*}(\Div_\infty(f_\infty))\\
&i_{\hat{E}'*}(\Div f')=p_{1*}(|E'|\cdot (X\times\{0\}-X\times\{\infty\}))\\
&i_{\hat{E}*}(\Div f)=p_{1*}(|E|\cdot (X\times\{0\}-X\times\{\infty\})).
\end{align*}
Thus
\[
p_{1*}(W\cdot (X\times\{0\}-X\times\{\infty\}))= i_{E_0*}(\Div f_0)+i_{\hat{E}'*}(\Div f')-i_{E_\infty*}(\Div f_\infty)-i_{\hat{E}*}(\Div f),
\]
showing that $p_{1*}(|W|\cdot (X\times\{0\}-X\times\{\infty\}))$ is in $R^d(X, X^*)$, as desired.
\end{proof}

\section{Limited functoriality}\label{sec:Functoriality} 

\begin{lemma}\label{lem:EtalePullback} Let $(X, X^*)$ and $(Y, Y^*)$ be admissible pairs with $\dim_kX=\dim_kY=d$, and let $f:Y\to X$ be a morphism.  Suppose\\[5pt]
i.  $f^{-1}(X^*)\supset Y^*$ \\[2pt]
ii. $f$ is flat over $X\setminus X^*$\\[2pt]
iii. There is a closed subset $S\subset X^*\cap X(d)$ with $\codim_{X(d)}S\ge2$ and an open neighborhood $U$ of $f(Y^*)\cap X(d)^0\setminus S$ in $X(d)^0$  such that $f$ is flat over $U$.\\[5pt]
 Then
\begin{enumerate}
\item $f^{-1}(X\setminus X^*)\subset Y\setminus Y^*$.
\item For $z\in Z^d(X\setminus X^*)$, the pullback $f^*(z)$ is a well-defined 0-cycle on $Y$, supported in $Y\setminus Y^*$, giving the group homomorphism
\[
f^*:Z^d(X\setminus X^*)\to Z^d(Y\setminus Y^*)
\]
\item The homomorphism $f^*$ of (2) descends to a homomorphism
\[
f^*:\CH^d(X, X^*)\to \CH^d(Y, Y^*)
\]
\item Given a second admissible pair $(Z, Z^*)$, a morphism $g:Z\to Y$  and a closed subset $S'\subset Y(d)\cap Y^*$, satisfying (i)-(iii), there is a closed subset $S''$ of $X(d)\cap X^*$ such that  $fg:Z\to X$ and $S''$ satisfy (i)-(iii), hence $(fg)^*: \CH^d(X, X^*)\to \CH^d(Z, Z^*)$  is defined. Moreover,  we have
\[
(fg)^*=g^*f^*:\CH^d(X, X^*)\to \CH^d(Z, Z^*)
\]
\end{enumerate}
\end{lemma}

\begin{proof} (1) is our assumption (i) that $f^{-1}(X^*)\supset Y^*$. The assumption (ii) implies   (2) since we have a well-defined pullback map $f^*:Z^d(X\setminus X^*)\to Z^d(f^{-1}(X\setminus X^*))$, for the flat map $f:f^{-1}(X\setminus X^*)\to X\setminus X^*$, and we then apply (1), showing that $f^*(z)$ is supported in $Y\setminus Y^*$ if $z$ is in $Z^d(X\setminus X^*)$.

For (3), let  $T=X(d)\cap X(\le d-1)$. Then $T$ is a closed subset of $X^*\cap X(d)$, with $\codim_{X(d)}T\ge 2$;  note that $X(d)\setminus (S\cup T)=X(d)^0\setminus S$ is an open subset of $X(d)^0$; $f$ is flat over $U\setminus S$ by assumption.

By Lemma~\ref{lem:Mov1} we have
\[
R^d(X, X^*)=R^d(X, X^*)_{T\cup S},
\]
with $R^d(X, X^*)_{T\cup S}$ the subgroup of $Z^d(X\setminus X^*)$ generated by 0-cycles $i_{C*}(\Div h)$, where $i_C:C\to X$ is a Cartier curve on $X$ relative to $X^*$, $h\in \sO_{C: C\cap X^*}^\times$, and $C$ is supported in $X(d)\setminus (S\cup T)=X(d)^0\setminus S$. Given such a pair $(i_C:C\to X, h\in \sO_{C: C\cap X^*}^\times)$, let  $(C\cap X^*)_U=(C\cap X^*)\cap U$, write the finite set $C\cap X^*$ as a disjoint union
\[
C\cap X^*=(C\cap X^*)_U\amalg B,
\]
and let $V= (X\setminus X^*)\cup U$. Then $B\subset X(d)\setminus (S\cup T)=X(d)^0\setminus S$ and $V$ is an open neighborhood of $C^0:=C\setminus B$ in $X(d)$. Moreover, $f$ is flat over $V$.  

Let $C_Y^0=f^{-1}(C^0)\subset f^{-1}(V)$ and let $C_Y\subset Y$ be the closure of $C_Y^0$ in $Y$. Since $f_C$ is flat of relative dimension zero  over $C\cap V$, $C_Y$ is a pure dimension one closed subscheme of $f^{-1}(C)$, and $f$ restricts to a quasi-finite morphism $f_C:C_Y\to C$. We also have the flat quasi-finite morphism $f_{C^0}:C_Y^0\to C^0$. Let $h_{C_Y}$ be the rational function $f_C^*(h)$ on $C_Y$ and 
$h_{C^0_Y}$ the rational function $f_{C^0}^*(h)$ on $C_Y^0$. Let $i_{C_Y}:C_Y\to Y$, $i_{C^0_Y}:C^0_Y\to f^{-1}(V)$ be the respective closed immersions.

 Note that $(C_Y\setminus C_Y^0)\cap f^{-1}(U)=\0$, so $C_Y\cap Y^*\subset f^{-1}(C\cap U)$. As $f$ is  flat over $U$ and $Y^*\subset f^{-1}(X^*)$, this implies that $C_Y\cap Y^*$ is a finite set and that $C_Y$ is defined by a regular sequence in a neighborhood of $C_Y\cap Y^*$ in $Y$. Thus $C_Y$ is a Cartier curve on $Y$ relative to $Y^*$.   Moreover,    $h_{C_Y}$ is in $\sO_{C_Y: C_Y\cap Y^*}^\times$ since $h\in \sO_{C: C\cap X^*}^\times$. Since $h$ is a unit in $\sO_{C, B}$, it follows that $h_{C_Y}$ is also a unit in $\sO_{C_Y, C_Y\cap f^{-1}(B)}$ so
\[
\Div_{C_Y}(h_C)=\Div_{C^0_Y}(h_C^0)\in Z_0(C_Y).
\]

Let $f_V:f^{-1}(V)\to V$ be the restriction of $f$, and let $i_{C\cap V}:C\cap V\to V$, $i_{C_Y^0}:C_Y^0\to f^{-1}(V)$ be the inclusions. As $f_V$ is flat,  we have
\[
f_V^*(i_{C\cap V*})(\Div h|_{C\cap V}))=i_{C^0_Y*}(\Div h_{C_Y^0})
\]
by standard results of classical intersection theory (see, e.g., \cite[\href{https://stacks.math.columbia.edu/tag/02S1}{Tag 02S1}]{stacks-project}). Thus
\[
f^*(i_{C*}(\Div h))=
f_V^*(i_{C\cap V*}(\Div h|_{C\cap V}))=i_{C^0_Y*}(\Div h_{C_Y^0})=
i_{C_Y*}(\Div h_{C_Y}). 
\]
Thus $f^*:Z^d(X\setminus X^*)\to Z^d(Y\setminus Y^*)$ satisfies
\[
f^*(R^d(X, X^*))\subset R^d(Y, Y^*),
\]
and therefore $f^*$ descends uniquely to a map
\[
f^*:\CH^d(X, X^*)\to \CH^d(Y, Y^*)
\]
on the  quotients.

Given a second  admissible pair $(Z, Z^*)$, a subset $S'\subset Y^*$ and morphism $g:Z\to Y$ satisfying (i)-(iii), we have the pullback map $g^*:\CH^d(Y, Y^*)\to \CH^d(Z, Z^*)$. We first check that $(fg)^*:\CH^d(X, X^*)\to \CH^d(Z, Z^*)$ is also defined.

The condition (i) for $fg$ follows directly from (i) for $f$ and for $g$. Next, by (ii),  $f:Y\to X$ is flat over $X\setminus X^*$ and $g:Z\to Y$ is flat over $Y\setminus Y^*$. Since $Y\setminus Y^*\supset f^{-1}(X\setminus X^*)$, it follows that $fg$ is flat over $X\setminus X^*$.

Let  $S\subset X^*\cap X(d)$ be a closed subset  and  $U\subset X$ an open neighborhood of  $f(Y^*)\cap X(d)^0\setminus S$ in $X(d)^0$  satisfying (iii) for $f$, and let  $S'\subset Y^*\cap Y(d)$, $U'\subset Y$ be given similarly for $g$. Let $\overline{f(S')}\subset X^*$ be the closure of $f(S')$ and let $S''=S\cup \overline{f(S')}$. Since $g(Z^*)\subset Y^*$, we have $fg(Z^*)\subset f(Y^*)$ and it is then easy to see that $S''$ and $U'':=U$ satisfy (iii) for $fg$ and hence $(fg)^*:\CH^d(X, X^*)\to \CH^d(Z, Z^*)$ is defined.

The functoriality for flat pullback
\[
(fg)^*=g^*f^*:Z^d(X\setminus X^*)\to Z^d(Z\setminus Z^*)
\]
on the level of cycles then implies
\[
(fg)^*=g^*f^*:\CH^d(X, X^*)\to \CH^d(Z, Z^*). 
\]
\end{proof}

\begin{remark}\label{rem:OpenIm} 1. We may apply Lemma~\ref{lem:EtalePullback} to the case of an open immersion $j:U\to X$, with $\dim_kU=d$, taking $S=\0$, to give a well-defined restriction map
\[
j^*:\CH^d(X,X^*)\to \CH^d(U,U\cap X^*),
\]
functorial for open immersions.\\[2pt]
2. We may also apply Lemma~\ref{lem:EtalePullback} to the case of a morphism $f:Y\to X$, with $Y$  regular and of pure dimension $d=\dim_kX$ over $k$ and $f:Y\to X$ flat over $X\setminus X^*$, where we take $Y^*=S=U=\0$. The gives the well-defined pullback map
\[
f^*:\CH^d(X, X^*)\to \CH^d(Y),
\]
satisfying $(fg)^*=g^*f^*$ for flat morphisms $g:Z\to Y$ in $\Sch^\qp_k$ with $Z$ regular and of pure dimension $d$. 
\end{remark}

\begin{lemma}[Base-change]\label{lem:base-change} Let $(X, X^*)$ be an admissible pair over $k$, let $d=\dim_kX$, and let $k\hookrightarrow F$ be an extension of fields. Let $X_F=X\times_kF$ with projection $p_X:X_F\to X$ and let $X_F^*=p_X^{-1}(X^*)$. Suppose that $X_F(d)^0$ is reduced and $X_F\setminus X_F^*$ is regular. Then\\[5pt]
1. $(X_F, X_F^*)$ is an admissible pair over $F$ and $\dim_FX_F=\dim_kX$.\\[2pt]
2. Flat pullback for the restriction $p_X:X_F\setminus X_F^*\to X_F\setminus X_F^*$ gives a well-defined map $p_X^*:Z^d(X\setminus X^*)\to Z^d(X_F\setminus X_F^*)$, which descends to $p_X^*:\CH^d(X, X^*)\to \CH^d(X_F, X_F^*)$.
\end{lemma}

\begin{proof} For $Z\in \Sch_k$ let $Z_F:=Z\times_kF$. We have $\dim_kZ=\dim_FZ_F$, so in particular $X_F(i)=X(i)_{F, \red}$ for all $i$ and  $\dim_FX_F=\dim_kX=d$. Thus $X_F^*\supset X_F(i)$ for all $i<d$. Since $X_F(d)^0$ is reduced and $X_F\setminus X_F^*$ is regular by assumption, $(X_F, X_F^*)$ is thus an admissible pair, which proves (1).

For (2), $p_X:X_F\to X$ is flat and thus $p_X:X_F\setminus X_F^*\to X_F\setminus X_F^*$ is flat and we have the well-defined pullback $p_X^*:Z^d(X\setminus X^*)\to Z^d(X_F\setminus X_F^*)$. If $i_C:C\to X$ is a Cartier curve on $X$ relative to $X^*$, then $C_\red$ is a  closed subscheme of $X(d)$ of pure codimension $d-1$, so the base-change $C_F$ has support  a closed subset of $X_F(d)$ of pure codimension $d-1$. Thus $i_{C_F}:C_F\to X_F$ has pure dimension one and is supported in $X_F(d)$. Similarly, $C_F\cap X_F^*$ is the finite set of points $p_X^{-1}(C\cap X^*)$, and by flatness, the fact that $C$ is a complete intersection on $X$ in a neighborhood of $C\cap X^*$ implies that $C_F$ is a complete intersection on $X_F$ in a neighborhood of $C_F\cap X_F^*$, that is, $i_{C_F}:C_F\to X_F$ is a Cartier curve on $X_F$ relative to $X_F^*$. Finally, if we have $f\in \sO_{C: C\cap X^*}^\times$, then letting $p_C:C_F\to C$ be the restriction of $p_X$, we have $p_C^*(f)\in \sO_{C_F: C_F\cap Y^*}^\times$ and
\[
i_{C_F*}(\Div p_C^*(f))=p_X^*(i_{C*}(\Div f)),
\]
so $p^*(R^d(X, X^*))\subset R^d(X_F, X_F^*)$.  Thus  $p_X^*:Z^d(X, X^*)\to Z^d(X_F, X_F^*)$ descends to $p_X^*:\CH^d(X, X^*)\to \CH^d(X_F, X_F^*)$, proving (2).
\end{proof}

\begin{remark} In case $k$ is perfect, or more generally, if $F$ is separable over $k$, the assumptions on $X_F(d)^0$ and $X_F\setminus X_F^*$ in Lemma~\ref{lem:base-change} are automatically satisfied.
\end{remark}

\begin{lemma}[Localization sequence]\label{lem:Localization} Let $(X, X^*)$ be an admissible pair, let $d=\dim_kX$ and let $i:W\to X$ be a closed immersion, with $i(W)\subset X\setminus X^*$. Then the pushforward $i_*:Z_0(W)\to Z^d(X\setminus X^*)$ descends to a group homomorphism 
\[
i_*:\CH_0(W)\to \CH^d(X,X^*)
\]
Moreover, letting $U:=X\setminus W$ with open immersion $j:U\to X$, $(U, X^*)$ is an admissible pair, and the sequence
\begin{equation}\label{eqn:Localization}
\CH_0(W)\xrightarrow{i_*}\CH^d(X,X^*)\xrightarrow{j^*}\CH^d(U,X^*)\to 0
\end{equation}
is exact.
\end{lemma}

\begin{proof} It is clear that for $C\subset W$ an integral curve, $i(C)\subset X$ is a Cartier curve relative to $X^*$, and for $h\in k(C)^\times$, $i_*(i_{C*}(\Div h))$ is in $R^d(X,X^*)$. This shows that $i_*:Z_0(W)\to Z^d(X\setminus X^*)$ descends to $i_*:\CH_0(W)\to \CH^d(X,X^*)$.

The proof of the exactness of the sequence \eqref{eqn:Localization} is exactly the same as for the usual Chow groups:  The sequence
\[
0\to Z_0(W)\xrightarrow{i_*}Z^d(X\setminus X^*)\xrightarrow{j^*}Z^d(U\setminus X^*)\to 0
\]
is exact, and if $C_U$ is a Cartier curve on $U$ relative to $X^*$, with $h\in \sO_{C_U:C_U\cap X^*}^\times$, then since $U\cap X^*=X^*$, we see that the closure $C$ of $C_U$ in $X$  is a Cartier curve on $X$ relative to $X^*$, $h$ determines a unique rational function $h_C\in \sO_{C: C\cap X^*}^\times$ on $C$, and if $z\in Z^d(X\setminus X^*)$ has $j^*(z)=i_{C_U*}(\Div  h)$, then $z-i_{C*}(\Div h_C)=i_*(w)$ for a unique $w\in Z_0(W)$.
\end{proof}

We will need a technical, partial version of functoriality, as follows.

\begin{lemma}\label{lem:PartialFunct} Let $(X, X^*)$ be an admissible pair. Take $z, z'\in Z^d(X\setminus X^*)$, let $i_C:C\to X$ be a Cartier curve on $X$ relative to $X^*$ and take $h\in \sO_{C:C\cap X^*}^\times$ with $i_{C*}(\Div h)=z-z'$. Let $f:Y\to X$ be a morphism of quasi-projective varieties over $k$, and suppose that $Y$ is reduced, of pure dimension $d:=\dim_kX$ over $k$, and that there is an open subscheme $U$ of $X\setminus X^*$ such that, letting $V=f^{-1}(U)$ and $f_U:V\to U$ the restriction of $f$, we have\\[5pt]
i. $f_U:V\to U$ is flat\\[2pt]
ii. $U$ contains $\supp(z)\cup \supp(z')$, all closed points $p$ of $C$ such that $f$ is not a unit in $\sO_{C,p}$ and all generic points of $C$.
\\[5pt]
Let $f_{U\cap C}:f_U^{-1}(C\cap U)\to C\cap U$ be the restriction of $f_U$ and let  $D$ be the closure in $Y$ of $f_U^{-1}(C\cap U)$. Then $f_{U\cap C}$ is quasi-finite, so we have the well-defined rational function $f_{U\cap C}^*(h)$ on $f_U^{-1}(C\cap U)$ 

Let $f^*(h)\in k(D)^*$ be the rational function with $f^*(h)_{|f_U^{-1}(C\cap U)}=f_{U\cap C}^*(h)$. Then $f^*(z)$ and $f^*(z')$ are defined as cycles, and 
\[
f^*(z)-f^*(z')=i_{D*}(\Div(f^*(h)))
\]
In particular, $f^*(z-z')=0$ in $\CH^d(Y)$.
\end{lemma}

\begin{proof} $D$ is clearly a closed subscheme of the closed subscheme $f^{-1}(C)$ of $Y$, so the restriction of $f$ to $D$ defines a quasi-finite morphism $f_D:D\to C$. Thus  $f^*_D(h)$ is a well-defined  rational function on $D$, with $f^*_D(h)$ equal to $f_{U\cap C}^*(h_{| U\cap C})$ on $f_U^{-1}(C\cap U)$, so $f_D^*(h)=f^*(h)$. In addition, at each point $p$ of $D\setminus f^{-1}(U\cap C)$, $h$ is in $\sO_{C, f(p)}^\times$, so $f^*(h)$ is a unit in $\sO_{D, p}$, and thus $\Div(f^*(h))$ is supported on $f^{-1}(U\cap C)$. Since 
\[
f^*(z)-f^*(z')=i_{f^{-1}(U\cap C)*}(\Div f_{U\cap C}^*(h))
\]
by \cite[\href{https://stacks.math.columbia.edu/tag/02S1}{Tag 02S1}]{stacks-project}, we see that
\[
f^*(z)-f^*(z')=i_{D*}(\Div(f^*(h))).
\]
\end{proof}

\section{Functoriality for morphisms to homogeneous spaces}\label{sec: Homogeneous}

Let $(X, X^*)$ be an admissible pair and let $Y$ be smooth and quasi-projective  over $k$. Given a morphism $f:X\to Y$, we might expect to have a well-defined pullback map
\[
f^*:\CH^d(Y)\to \CH^d(X, X^*),\ , =\dim_kX,
\]
with  $f^*g^*=(gf)^*$ for an additional morphism $g:Y\to Z$ of smooth, quasi-projective $k$-schemes. Our main technical tool is the construction of such pullback maps for $Y$ a quasi-projective homogeneous space for a suitable type of linear algebraic group $G$. We closely follow the statements and arguments from \cite[\S 4]{BS98} and \cite[\S 6]{BKS}.

For $G$ a group-scheme over $k$, and $A, B$ subschemes of $G$, we let $A*B$ denote the image of $A\times B$ under the product map $G\times G\to G$. Similarly, we write $A(k)*B(k)$ for the image of $A(k)\times B(k)\subset G(k)\times G(k)$ under the product map $G(k)\times G(k)\to G(k)$. As a matter of notation,  given $k$-schemes $A, B$ and a point $q\in A\times_k B$, we write $q =(a,b)\in A\times_kB$ to mean $a=p_1(q)$, $b=p_2(q)$.

\begin{definition} 1. Let $G$ be a linear algebraic group over $k$, i.e., $G$ is a group-scheme over $k$ that admits a description as a closed subgroup scheme of $\GL_n/k$ for some $n$. Recall that $G$ is called {\em $k$-rational} if $G$ is smooth over $k$, geometrically connected, and there is a dense open subscheme $U$ of $G$ that is isomorphic to an open dense subscheme in a projective space $\P^{\dim_kG}_k$ over $k$. A $k$-rational linear algebraic group $G$ over $k$ is called {\em strongly $k$-rational} if given dense open subschemes $U, V\subset G$, we have $U(k)*V(k)=(U*V)(k)$.\\[5pt]
2. Let $G$ be a linear algebraic group over $k$. For us a {\em homogeneous space}  for $G$ is a smooth, finite-type $k$-scheme $H$ with $H(k)\neq\0$, endowed with a $G$-action $\rho:G\times_kH\to H$ such that, for $x\in H(k)$, the isotropy subgroup $G_x$ is a reduced subgroup-scheme of $G$ and the action map $\rho_x:G=G\times_kx\to H$, $\rho_x(g)=g\cdot x$, identifies $H$ with the quotient $G/G_x$.\end{definition}

\begin{remarks}\label{rem:Star} 1. Our definition of the term ``homogeneous space'' is non-standard in positive characteristic, in that we require smoothness over $k$ and the isotropy group $G_x$ to be reduced for each $k$-point $x$, but we use it to avoid encumbering the notation with additional adjectives. In our applications, $G$ will always be a product of general linear groups, and $H$ with always be a product of Grassmannians or flag varieties, which does satisfy our additional requirements.\\[5pt]
2. Let $G$ be a $k$-rational linear algebraic group over $k$ that satisfies Hilbert's Theorem 90 over $k$, that is, $H^1_\et(k, G)=\{*\}$. Then $G$ is strongly $k$-rational. Indeed, letting $\mu:G\times G\to G$ be the multiplication map, then for $x\in G(k)$, the fiber $\mu^{-1}(x)$ is a torsor for $G$ over $\Spec k$, $G$ acting by $g\cdot (y,z):=(yg^{-1}, gz)$. Thus $\mu^{-1}(x)\cong G$ as $k$-scheme, and in particular, every dense open subscheme $U$ of $\mu^{-1}(x)$ has $U(k)$ Zariski dense in $U$. We apply this to $A, B$ dense open subsets of $G$, $x\in (A*B)(k)$ and $U=A\times B\cap \mu^{-1}(x)$, which gives us a point $(a,b)\in A(k)\times B(k)$ with $\mu(a,b)=x$, showing $G$ is strongly $k$-rational.
\end{remarks}

For example, products of linear algebraic groups over $k$,  $G=\prod_{i=1}^r\PGL_{n_i}$, $\prod_{i=1}^r\GL_{n_i}$, or $\prod_{i=1}^r\SL_{n_i}$  are all  $k$-rational and $\prod_{i=1}^r\GL_{n_i}$, $\prod_{i=1}^r\SL_{n_i}$ are strongly $k$-rational.  The Grassmann varieties $\Gr(m, N)$  are homogeneous spaces for $\GL_N$, $\SL_N$ and $\PGL_N$, as are the partial flag varieties $\Fl(m_1, m_2, N)$.

\begin{definition} Let $Z$ be a quasi-projective $k$-scheme, and let $A, B\subset Z$ be closed subsets. We say that $A$ and $B$ intersect properly in $Z$ if for each point $z\in A\cap B$ we have
\[
\dim_k(A\cap B, z)=\dim_k(A,z)+\dim_k(B,z)-\dim_k(Z,z).
\]
Equivalently,
\[
\codim_Z(A\cap B, z)=\codim_Z(A,z)+\codim_Z(B,z).
\]
Note that $A$ and $B$ intersect properly if $A\cap B=\0$.
\end{definition}

\begin{remark}\label{rem:lci}
For $i:W\to Z$ a closed immersion in $\Sch_k$, the subset 
\[
W_\lci:=\{w\in W\mid \sI_{W,w}\subset \sO_{Z,w}\text{ is generated by a regular sequence}\}
\]
is open in $W$ and is equal to the maximal open subscheme $U$ of $W$ such that $U\to Z$ is a regular immersion (see, e.g., \cite[Proposition 16.9.10]{EGAIV4} or \cite[Lemma 31.21.8(1,5)\href{https://stacks.math.columbia.edu/tag/063I}{Tag 063I}]{stacks-project}).
Let $W_\nlci:=W\setminus W_\lci$.
\end{remark}

\begin{lemma}[\hbox{\cite[Lemma 4]{BS98}}]\label{lem:TranslateCycle} Let $G$ be a $k$-rational linear algebraic group, $H$ a homogeneous space for $G$. Take $Z\in \Sch_k$, let   $f :Z\to H$ be a morphism, and let $Y$ be a reduced closed subscheme of $H$ of
pure codimension $i$. Let $Z^*\subset Z$ be a closed subset with $\dim_kZ^*<\dim_kZ$ and let $Z_0\subset Z$ be a closed subset of dimension $<i$. Then there is a dense open subset $O(f , Y, Z^*, Z_0)\subset G$ such that for each
$g \in O(f, Y, Z^*, Z_0)(k)$, the following properties hold.
\begin{enumerate}
\item  $f^{-1}(gY)$ has pure codimension $i$ in $Z$.
\item  $f^{-1}(gY)$ intersects $Z^*$  properly in $Z$ and $f^{-1}(gY)\cap Z_0=\0$.
\item $f^{-1} (gY)$ is a local complete intersection in $Z$ at each generic point of
$f^{-1}(gY)\cap Z^*$.
\item  The sheaves $\Tor^{\sO_H}_i(\sO_{gY}, \sO_Z) $ vanish for all $i>0$, that is, $\sO_{gY}$ and $\sO_Z$ are Tor-independent over $\sO_H$.
\end{enumerate}
Moreover, $O(f,Y, Z^*, Z_0)(k)$ is Zariski dense in $O(f,Y, Z^*, Z_0)$.
\end{lemma}

\begin{proof} This is a version of Kleiman's transversality theorem \cite[Theorem 2]{Kleiman}. The case of   $k$ to be algebraically closed, with minor changes in the statement,  is \cite[Lemma 4]{BS98} and the result in general follows easily from the case of algebraically closed $k$, as we now explain. 

 Let $k\to \bar{k}$ be an algebraic closure. (1) and (2) follow directly from the case of algebraically closed $k$, since given a dense open subset $V\subset G$ such that (1) and (2) are satisfied for all $g\in V(\bar{k})$, the fact that $G$ is $k$-rational implies that $V(k)$ is dense in $V$, and we can consider each $k$-point of $V$ as a $\bar{k}$-point.

 (4) follows similarly from the algebraically closed case, using the fact that $k\to \bar{k}$ is faithfully flat, so we can deduce the vanishing of the Tor-sheaves from their vanishing after base-extension by $k\to \bar{k}$. 
 
 For (3),   since $H$ is smooth over $k$,  the local ring $\sO_{H,x}$  a regular local ring for all $x\in H$ \cite[\href{https://stacks.math.columbia.edu/tag/056S}{tag 056S}]{stacks-project}. In particular, for each generic point $y$ of $Y$, 
 the local ring $\sO_{H, y}$ is a regular local ring, of dimension $i$, so the maximal ideal is generated by a regular sequence of length $i$ and as $Y$ is reduced, it follows that $y$ is in $Y_\lci$. Following Remark~\ref{rem:lci},  $Y_\lci$ is open and dense in $Y$. Thus,  there is an affine open subset $U=\Spec A\subset H$ with $U\cap Y$ dense in $Y$, such that the ideal $I\subset A$ of the closed subscheme $Y\cap U$ of $U$ is generated by a regular sequence $s_1,\ldots, s_i$. By \cite[\S (18.D) Corollary]{Matsumura} the Koszul complex $\text{Kos}_A(s_1,\ldots, s_i)$ defines a finite free resolution of $A/I$ as $A$-module. 
 
 Let $Y_0=Y\setminus Y\cap U$. By (4), there is a Zariski dense open subscheme $U_1$ of $G$ such that $\sO_{gY}$ and $\sO_Z$ are Tor-independent over $\sO_H$ for all $g\in U_1(k)$; by (2) applied to $Y$ and $Y_0$, we may assume that $f^{-1}(g(Y\setminus Y_0))\cap Z^*$ contains each generic point of $f^{-1}(gY)\cap Z^*$. Thus, the Koszul complex $g\cdot \text{Kos}_A(s_1,\ldots, s_i)$ pulls back via $f$ to give a free resolution of $\sO_{f^{-1}(gY)}$ on some affine open neighborhood $V$ of the generic points of $f^{-1}(gY)\cap Z^*$ in $Z$, hence $f^{-1}(gY)$ is defined by the regular sequence $f^*((g\cdot s_1)),\ldots, f^*(g\cdot s_i)$ on $V$ (here $g\in G(k)$ acts on the Koszul complex and the $s_i$ by $(g^{-1})^*(-)$).
 
 Taking $O(f, Y, Z^*, Z_0)$ to be the intersection of the open subsets satisfying (1)-(4) completes the proof.
\end{proof}
We sometimes omit $Z^*$ or $Z_0$ from the notation $O(f, Y, Z^*, Z_0)$ if these are empty.

\begin{lemma}[\hbox{\cite[Lemma 6]{BS98}}]\label{lem:TranslateRelation} Let $G$ be a $k$-rational algebraic group, and $H$ a $G$-homogeneous space. Let  $(X, X^*)$ be an admissible pair and    $f : X \to H$ a morphism. Let $G$ act on $H\times\P^1$ by the trivial
action on $\P^1$ (that is, $g . (y, u) = (gy, u)$). Let $W \subset H \times\P^1$ be a reduced closed subscheme of pure codimension $i$, flat over a neighborhood of $\{0,\infty\}$ in $\P^1$, and let $S\subset X$ be a closed subset of dimension $\le i-2$. For $g\in G(k)$, let
$W(g) := (f \times\id_{\P^1})^{-1}(gW ) \subset X\times\P^1$.
Then there exists a dense open subset $U(f, W, X^*, S) \subset G$ such that for all $g \in U(f,W, X^*,S)(k)$, the following
properties hold.
\begin{enumerate}
\item $W(g)$ is of pure  codimension $i$ in $X\times\P^1$.
\item $W(g)$ intersects $X^* \times\P^1$ and $X^* \times \{0,\infty\}$ properly, and
$W(g)\cap S\times\P^1=\0$.
\item $W(g)$ is a complete intersection in $X\times\P^1$ in a neighborhood of each generic point of
$W(g) \cap X^* \times\P^1$.
\item $W(g)$ is flat over a neighborhood of $\{0,\infty\}\subset\P^1$.
\end{enumerate}
Moreover, $U(f, W, X^*, S)(k)$ is Zariski dense in $U(f, W, X^*, S)$.
\end{lemma}

\begin{proof} The items (1)-(3) follow from  Lemma~\ref{lem:TranslateCycle}, taking
\[
U(f, W, X^*, S)= O(f\times\id_{\P^1},W, X^*\times\P^1, S\times\P^1)\cap 
O(f\times\id_{\{0,\infty\}},W, X^*\times\{0,\infty\}, \0). 
\]
For (4), by generic flatness, $W(g)$ fails to be flat over $\P^1$ on a proper closed subset of $\P^1$, so we need only check flatness at $\{0,\infty\}$. This is a consequence of the  flatness of $W\subset H\times\P^1$ over a neighborhood of $\{0,\infty\}$ and Lemma~\ref{lem:TranslateCycle}(4) for the map $f\times\id_{\{0,\infty\}}$.

One could also reduce to   \cite[Lemma 6]{BS98},  by the same base-change argument we used to reduce the proof of Lemma~\ref{lem:TranslateCycle} to \cite[Lemma 4]{BS98}.
\end{proof}

\begin{lemma}\label{lem:GenPosRelation} Let $(X, X^*)$ be an admissible pair and let $d=\dim_kX$. 
Let $Y\in \Sch^\qp_k$ be smooth over $k$, and let $W\subset Y\times\P^1$ be a reduced closed subscheme of pure codimension $d$. Let $f:X\to Y$ be a morphism. Let $W_f:=(f\times\id)^{-1}(W)$. suppose that
\begin{enumerate}
\item[(a)] $W_f$ is of pure codimension $d$ on $X\times\P^1$.
\item[(b)]  $W_f\subset (X_\CM\amalg X(d-1)^0)\times\P^1$
\item[(c)] $W_f\cap X^* \times\P^1$ is a finite set 
\item[(d)] $(f\times\id)^{-1}(W_\nlci)\cap X^* \times\P^1=\0$
\item[(e)]  $W_f\cap X \times \{0,\infty\}$  is a closed subscheme of $(X\setminus X^*)\times \{0,\infty\}$, finite over $k$.
\end{enumerate}
Then $W_f\subset X\times\P^1$ satisfies the conditions \eqref{eqn:Properties}(iii.) and Remark~\ref{rem:CodimNotDim}(i$'$, ii$'$). Moreover, the 0-cycle
\[
p_{1*}(|W_f|\cdot (X\times\{0\}-X\times\{\infty\})
\]
is in $R^d(X, X^*)$.
\end{lemma}

\begin{proof} By Remark~\ref{rem:lci}, $W_\lci$ is dense in $W$.

Remark~\ref{rem:CodimNotDim}(i$'$) follows from (a), (b) and (c). \eqref{eqn:Properties}(iii.) is (e). 

For Remark~\ref{rem:CodimNotDim}(ii$'$) let $W_{f\CM}=W_f\cap X_\CM\times\P^1$.  It follows from (b) that $W_{f\CM}$ is closed in $X\times\P^1$ and is supported in the Cohen-Macaulay locus of $X\times\P^1$. Take $x\in W_{f\CM}\cap X^* \times\P^1$, and let $\sI_{W_f,x}\subset \sO_{X\times\P^1,x}$ be the ideal of $W_f$ in the local ring $\sO_{X\times\P^1,x}$. By (d),  $y:=(f\times\id)(x)$ is in $W_\lci$, so $W$ is defined by a regular sequence of length $d$, $s_1,\ldots, s_d$,  in a Zariski open neighborhood  of $y$ in $Y\times\P^1$, giving the generators $f^*(s_1),\ldots, f^*(s_d)$ of $\sI_{W_f,x}$. Since $W_f$ has pure codimension $d$ and $\sO_{X\times\P^1,x}$ is Cohen-Macaulay of dimension $\ge d$, this implies that $f^*(s_1),\ldots, f^*(s_d)$ is a regular sequence in $\sO_{X\times\P^1,x}$; this verifies Remark~\ref{rem:CodimNotDim}(ii$'$).

The fact that $p_{1*}(|W_f|\cdot (X\times\{0\}-X\times\{\infty\})$ is in $R^d(X, X^*)$ follows from Remark~\ref{rem:CodimNotDim} and Theorem~\ref{thm:Relations}.
\end{proof}

A version of the following result is proven in \cite[Lemma 7]{BS98}, and an even closer version is done in \cite[\S 5]{BKS}, for the case $G=\prod_{i=1}^r\GL_{n_i}$ (which is the only case we need anyway) with essentially the same proof as below. 

\begin{lemma}\label{lem:Parametrize} Let $(X, X^*)$ be an admissible pair and  let $d=\dim_kX$. Let $G$ be a   $k$-rational  linear algebraic group over $k$, $H$ a quasi-projective homogeneous space for $G$, $Y\subset H$ a reduced closed subscheme of pure codimension $d$ and $f:X\to H$ a morphism. Then there is a dense open subscheme $B$ of $G\times G$, a dominant  morphism   $v:B\times G\to G\times G$, and a reduced closed subscheme $\Sigma\subset B\times G\times H\times \P^1$ such that \\[5pt]
i. The projections $p_{12}:\Sigma\to B\times G$   and $p_4:\Sigma\to \P^1$ are both flat, \\[2pt]
ii. For $(b, g)\in B\times G$, the scheme-theoretic fiber $\Sigma_{b,g}$ is reduced and of pure codimension $d$ in $(H\times\P^1)_{k(b,g)}$.\\[2pt]
iii. For $(b,g)\in (B\times G)$, write $ v(b,g))=(g_1, g_2)\in G\times_kG$. Then we have the identity of fibers 
\[
\Sigma_{b,g}\otimes_{\sO_{\P^1}}k(0)=g_1\cdot Y,\ \Sigma_{b,g}\otimes_{\sO_{\P^1}}k(\infty)=g_2\cdot Y.
\]
iv. Let $G$ act on $B\times G\times H\times \P^1$ by $g'\cdot (b,g, h,t)=(b,g'\cdot g, g'\cdot h, t)$. Then $\Sigma$ is stable under this $G$-action and $(g'\Sigma)_{b,g}=\Sigma_{b, g'g}$. 
\end{lemma}

\begin{proof} In case $k$ is algebraically closed, this is proven in the beginning of the proof of \cite[Lemma 7]{BS98}, except that $B$ is not shown to be an open subscheme of $G\times G$; this is because it is not assumed in {\it loc. cit.} that $G$ is $k$-rational. We give here a construction of $B$, following the ideas in the proof of \cite[Lemma 7]{BS98} and the construction in \cite[\S 5.2]{BKS}. 

Let $d=\dim_k G$. Let $\bar{G}$ be a projective closure of $G$. By assumption, there is a dense open subscheme $U_1\subset G$ with $U_1$ isomorphic to a dense open subscheme of $\A^d\subset \P^d$, via an open immersion $j:U_1\to \A^d$.  We have the morphism
\[
\lambda:(U_1\times U_1\setminus \Delta_{U_1})\times\P^1\to \P^d
\]
sending $(g_1, g_2)\times\P^1$ to the line $\ell_{j(g_1), j(g_2)}\subset \P^d$ spanned by $j(g_1)\neq j(g_2)\in \P^d$; the restriction of $\lambda$ to $(g_1, g_2)\times\P^1$ is uniquely determined by  the identities
\[
\lambda((g_1, g_2), 0)=j(g_1),\ \lambda((g_1, g_2), \infty)=j(g_2),\ \lambda((g_1, g_2), 1)=\ell_{j(g_1), j(g_2)}\cap(\P^d\setminus\A^d).
\]

Over $V:=\lambda^{-1}(j(U_1))$, we can lift $\lambda$ to a morphism
\[
\lambda_V:V\to U_1, j\circ\lambda_V=\lambda|_V.
\]
For $\eta$ the generic point of $U_1\times U_1$,   $\eta\times \P^1\cap V$ is a dense open subscheme of $\eta\times\P^1$, hence the restriction of  $\lambda_V$ to $\eta\times \P^1\cap V$ extends uniquely to a morphism $\lambda_\eta:\eta\times\P^1\to \bar{G}$. Thus, there is a dense open subscheme $B$ of $U_1\times U_1\setminus\Delta_{U_1}$ such that the restriction of $\lambda_V$ to $B\times\P^1\cap V$ extends to a morphism
\[
p:B\times\P^1\to \bar{G}.
\]

By construction, $p(B\times\{0,\infty\})\subset U_1\subset G$, with
\[
p((g_1, g_2)\times\{0\})=g_1,\ p((g_1, g_2)\times\{\infty\})=g_2
\]
for $(g_1, g_2)\in B\subset U_1\times U_1$. Moreover the projections $p_1, p_2:B\to U_1$ both have dense image and since $B\subset U_1\times U_1$ and $U_1$ is an open subset of an affine space, this implies that $p_1, p_2:B(k)\to U_1(k)$ both have dense image as well.

Let $\Sigma\subset B\times G\times H\times\P^1$ be the closure of 
\[
\Sigma^0:=\{(b, g, h, t)\in B\times G\times H\times\P^1\mid p(b,t)\in G, h\in g\cdot p(b,t)\cdot Y\}
\]
with projection $\pi:\Sigma\to G\times B\times \P^1$ mapping $(b, g, h, t)$ to $(g, b,t)$. We have the open subscheme $p^{-1}(G)\subset B\times\P^1$, giving the open subscheme 
$G\times p^{-1}(G)$ of $G\times B\times\P^1$, with
\[
\pi(\Sigma^0)=G\times p^{-1}(G)
\]
and with $\Sigma^0$ closed in $\pi^{-1}(G\times p^{-1}(G))$. Moreover, for $(b,t)\in p^{-1}(G)$, we have
\[
\pi^{-1}(g,b,t)=\{(b,g)\}\times g\cdot p(b,t)\cdot Y\times \{t\}\cong g\cdot p(b,t)\cdot Y,
\]
the isomorphism as schemes over $k(g,b,t)$. 

Let $\tau$ be the automorphism of $G\times G\times G$ with $\tau(g_1, g_2, g)=(g, g_1, g_2)$, and define the  morphism $v:B\times G\to G\times G$ as the composition
\[
B\times G\hookrightarrow (G\times G)\times G\xrightarrow{\tau}G\times (G\times G)\xrightarrow{(\mu\circ p_{12}, \mu\circ p_{13})}G\times G
\]
that is 
\[
v(b, g)=v((g_1, g_2), g)=(gg_1, gg_2).
\]

In case $G$ is a product of general linear groups, the open immersion $\GL_n\subset M_{n\times n}=\A^{n^2}$ writes $G$ as an open subscheme of an affine space $\A^N$, so one can take $U_1=G$, $\bar{G}=\P^N$; this is the setting used in \cite{BKS}. The arguments used to prove  \cite[\S 3.3, Lemmas 5.6, 5.7, 5.9]{BKS} carry over with minor changes to give a proof of Lemma~\ref{lem:Parametrize}; note that we may need to shrink $B$ so that (i)-(iv) all hold.
\end{proof}

\begin{proposition}[\hbox{\cite[Lemma 7]{BS98}}]\label{prop:TranslationPullback}
Let $(X, X^*)$ be an admissible pair and  let $d=\dim_kX$.  Let $H$ be a quasi-projective homogeneous space for a strongly $k$-rational linear algebraic group $G$, and let
$Y$ be a reduced closed subscheme  of $H$ of pure codimension $d$. Let  $f:X \to H$ be a morphism. There is a dense open subset $V(f, Y, X^*)$ of $G$  such that, for all $g_1, g_2\in V(f,Y, X^*)(k)$, we have
\begin{enumerate}
\item  The subschemes $f^{-1}(g_1Y)$, $f^{-1}(g_2)(Y)$ of $X$  are of pure codimension $d$, supported in $X\setminus X^*$. 
\item Letting $[f^{-1}(g_iY)]\in \CH^d(X, X^*)$ denote the class of the zero-cycle 
\[
|f^{-1}(g_iY)|\in Z^d(X\setminus X^*),\ i=1,2,
\] 
we have
\[
[f^{-1}(g_1Y)]=[f^{-1}(g_2Y)]\in \CH^d(X, X^*).
\]
\end{enumerate}
\end{proposition}
 
\begin{proof} Our proof closely follows the proof of \cite[Lemma 7]{BS98} and the discussion in \cite[\S 6.1, 6.2]{BKS}.

From Lemma~\ref{lem:TranslateCycle} we have the dense open subset $O(f, Y, X^*,\0)$ of $G$ such that, for all $h_1, h_2\in O(f, Y, X^*,\0)(k)$, the closed subschemes  $f^{-1}(h_1Y)$ and $f^{-1}(h_2Y)$ of $X$  are of pure codimension $d$, supported in $X\setminus X^*$. Indeed,  by Lemma~\ref{lem:TranslateCycle}(1), $f^{-1}(h_iY)$ has pure codimension $d$ for $i=1,2$. Since each component of $X^*$ has dimension $\le d-1$,  $f^{-1}(h_iY)\cap X^*=\0$ by Lemma~\ref{lem:TranslateCycle}(2b).

For $(b,g)\in B\times G$, recall from Lemma~\ref{lem:Parametrize} that the fiber $\Sigma_{b,g}$ is a reduced closed subscheme of $(H\times\P^1)_{k(b,g)}$. Let $\widetilde{\Sigma}_{b,g}:=(f\times\id)^{-1}(\Sigma_{b,g})$.

 Recall the dominant morphism $v:B\times G\to G\times G$ of Lemma~\ref{lem:Parametrize} and
let $T\subset v^{-1}(O(f, Y, X^*,\0)\times O(f, Y, X^*,\0))\subset B\times G$ be the subset consisting of points $(b,g)$ such that
\begin{enumerate}
\item[(a)] $\widetilde{\Sigma}_{b,g}$ is of pure codimension $d$ on $(X\times\P^1)_{k(b,g)}$.
\item[(b)]  $\widetilde{\Sigma}_{b,g}\subset (X_\CM\amalg X(d-1)^0)\times\P^1_{k(b,g)}$
\item[(c)] $\widetilde{\Sigma}_{b,g}\cap X^* \times\P^1_{k(b,g)}$ is a finite set 
\item[(d)] $(f\times\id)^{-1}((\Sigma_{b,g})_\nlci)\cap X^* \times\P^1_{k(b,g)}=\0$
\item[(e)]  $\widetilde{\Sigma}_{b,g}\cap X_{k(b,g)} \times \{0,\infty\}$  is a closed subscheme of $(X\setminus X^*)_{k(b,g)}\times \{0,\infty\}$, finite over $k(b,g)$.
\end{enumerate}

As noted in \cite[Lemma 6.1]{BKS} and in the  proof of \cite[Lemma 7]{BS98},   $T$ is a constructible subset of $B\times G$. Indeed, let 
\[
\tilde{\Sigma}=(\id_{B\times G}\times f\times\id_{\P^1})^{-1}(\Sigma)\subset B\times G\times X\times\P^1. 
\]
Then conditions (a), (b), (c), (e) can be phrased as saying that the fiber of  
$\tilde{\Sigma}\cap B\times G\times F$ has dimension $\le r$ for suitable closed subsets $F$ of $X\times\P^1$ and integers $r$, a constructible condition. Condition (d) is similar, replacing $\Sigma$ with $\Sigma_\nlci$,  replacing $\tilde{\Sigma}$ with $(\id_{B\times G}\times f\times\id_{\P^1})^{-1}(\Sigma_\nlci)$ and taking $F=X^* \times\P^1$, $r=-1$. This shows that (d) defines a constructible condition, since   $(\Sigma_{b,g})_\nlci= (\Sigma_\nlci)_{b,g}$ because $\Sigma\to B\times G$ is flat.

By Theorem~\ref{thm:Relations} and Lemma~\ref{lem:GenPosRelation},   the 0-cycle
\[
p_{1*}(|\widetilde{\Sigma}_{b,g}|\cdot(X\times\{0\}-X\times\{\infty\})
\]
is a well-defined element of $R^d(X,X^*)$ for all $(b,g)\in T(k)$. Moreover,  if $v(b,g)=(g_1, g_2)$, we have
\begin{equation}\label{eqn:ParamRelation}
|f^{-1}(g_1Y)|-|f^{-1}(g_2Y)|=p_{1*}(|\widetilde{\Sigma}_{b,g}|\cdot (X\times\{0\}-X\times\{\infty\})
\end{equation}
hence $|f^{-1}(g_1Y)|-|f^{-1}(g_2Y)|$ is in $R^d(X, X^*)$.

We claim that $T$ contains an open dense subset $U$ of $B\times G$. To see this,  let $S=X\setminus  (X_\CM\amalg X(d-1)^0)$; following Remark~\ref{rem:Prelims}, we see that $\dim_kS\le d-2$. Now take an arbitrary $(b,g)\in B(k)\times G(k)$, giving $\Sigma_{b,g}\subset H\times \P^1$. Applying Lemma~\ref{lem:TranslateRelation}, we have the open subset $U(f, \Sigma_{b,g}, X^*, S)$ of $G$ such that for all $g'\in U(f, \Sigma_{b,g}, X^*, S)$, $(f\times\id)^{-1}(g'\Sigma_{b,g})$ satisfies the conditions (a)-(e) above. But since $g'\Sigma_{b,g}=\Sigma_{b, g'g}$, it follows that $T\cap b\times G$ contains the open subscheme $b\times U(f, \Sigma_{b,g}, X^*,S)\cdot g$, and since $T$ is constructible and $B(k)$ is Zariski dense in $B$, it follows that $T$ must contain a dense open subset $U$ of $B\times G$.

Since $v:B\times G\to G\times G$ is an open morphism, we have the dense open subset $v(U)\subset G\times G$. Fix a $g_0\in G(k)$ such that $g_0\times G\cap v(U)$ is open in $g_0\times G$ and let 
\[
V(f, Y, X^*)=p_2(g_0\times G\cap v(U)). 
\]
Note that the map $v:U(k)\to v(U)(k)$ is surjective, since $G$ is strongly $k$-rational. For $g_1, g_2\in V(f,Y, X^*)(k)$, choose $(b_i, h_i)\in U(k)$ with $v(b_i, h_i)=(g_0, g_i)$, $i=1,2$. Then by \eqref{eqn:ParamRelation}, the cycles
\[
|f^{-1}(g_0Y)|, |f^{-1}(g_1Y)|, |f^{-1}(g_2Y)|
\]
are all well-defined 0-cycles in $Z^d(X\setminus X^*)$ and
\[
[f^{-1}(g_iY)]=[f^{-1}(g_0Y)]\in \CH^d(X, X^*),\ i=1,2,
\]
hence $[f^{-1}(g_1Y)]=[f^{-1}(g_2Y)]$, which completes the proof.
\end{proof}

\begin{theorem}\label{thm:HomogFunct} Let $(X, X^*)$ be an admissible pair and  let $d=\dim_kX$.   Let $H$ be a quasi-projective homogeneous space for a strongly $k$-rational linear algebraic group $G$, and let $f:X\to H$ be a morphism. There is a group homomorphism $f^*:\CH^d(H)\to \CH^d(X, X^*)$ such that,  given $z\in Z^d(H)$ with class $[z]\in \CH^d(H)$, there a dense open subset $U_{f,z}$ of $G$ such that for all $g\in U_{f,z}(k)$,   we have
\begin{enumerate}\label{thm:1}
\item[(a)] $f^{-1}(\supp(g\cdot z))\subset X\setminus X^*$,
\item[(b)] $f^{-1}(\supp(g\cdot z))$ has pure codimension $d$ on $X$ and the cycle-theoretic pull-back $f^*(g\cdot z)\in Z^d(X\setminus X^*)$ is defined. 
\item[(c)] $[f^*(g\cdot z)]=f^*([z])\in \CH^d(X,X^*)$.
\end{enumerate}
\end{theorem}

\begin{proof} Write $z=\sum_{i=1}^nm_iY_i$ with the $Y_i$ codimension $d$ integral closed subschemes of $H$. We have the dense open subsets $V(f, Y_i, X^*)\subset G$ given by Proposition~\ref{prop:TranslationPullback}. Let $U_{f,z}=\cap_{i=1}^nV(f, Y_i, X^*)$, and choose $g\in U_{f,z}(k)$. Then (a), (b) for $g\cdot z$ follow from Proposition~\ref{prop:TranslationPullback}, giving us the well-defined cycle $f^*(g\cdot z)\in Z^d(X\setminus X^*)$. Moreover, for each $i$, the class $[f^*(g\cdot Y_i)]\in \CH^d(X, X^*)$ is independent of the choice of $g\in U_{f,z}(k)$, hence the same follows for the class $[f^*(g\cdot z)]\in \CH^d(X, X^*)$. Thus, sending $z\in Z^d(H)$ to $[f^*(g\cdot z)]\in \CH^d(X, X^*)$ for some $g\in U(f,z)(k)$ gives  a well-defined group homomorphism
\[
\tilde{f}^*:Z^d(H)\to \CH^d(X, X^*).
\]

Now suppose we have a second codimension $d$ cycle $z'=\sum_{i=1}^{n'}m'_iY'_i$, and take $g'\in U_{f,z'}$, giving the well-defined  cycle $f^*(g'\cdot z')\in Z^d(X\setminus X^*)$, with class $[f^*(g'\cdot z')]\in \CH^d(X, X^*)$ independent of the choice of $g'$. Supposing that $[z]=[z']\in \CH^d(H)$, we claim that 
\begin{equation}\label{eqn:RationalEquivPullback}
[f^*(g\cdot z)] = [f^*(g'\cdot z')]\in \CH^d(X, X^*),
\end{equation}
which will show that $\tilde{f}^*$ descends to a well-defined group homomorphism
\[
f^*:\CH^d(H)\to \CH^d(X, X^*),
\]
proving (c). 

Since $G$ is $k$-rational, the action of $G(k)$ on $H$ induces the trivial action on $\CH^d(H)$, so $[z]=[z']\in \CH^d(H)$ implies that $[h\cdot z]= [h'\cdot z']$ for all $h, h'\in G(k)$. To verify \eqref{eqn:RationalEquivPullback}, choose a codimension $d$ cycle $W\subset H\times\P^1$ with
\[
g\cdot z-g'z'=p_{1*}(W\cdot (H\times\{0\}-H\times\{\infty\})).
\]
It follows easily from Lemmas~\ref{lem:TranslateRelation}, \ref{lem:GenPosRelation} that there is a dense open subset $U(f, W, X^*)\subset G$ such that, for all $g''\in U(f, W, X^*)(k)$, the pullback cycle $((g''\cdot f)\times\id_{\P^1})^*(W)$ is a well-defined codimension $d$ cycle on $X\times\P^1$, the intersection products
\[
p_{1*}((f\times\id_{\P^1})^*((g''\times\id)\cdot W)\cdot (X\times\{0\})),\ 
p_{1*}((f\times\id_{\P^1})^*((g''\times\id)\cdot W)\cdot (X\times\{\infty\}))
\]
are well defined and supported in $X\setminus X^*$, and the cycle
\[
p_{1*}((f\times\id_{\P^1})^*((g''\times \id)W)\cdot (X\times\{0\}-X\times\{\infty\}))
\]
is in $R^d(X, X^*)$. We may take $g''\in U(f, W, X^*)(k)$ such that $g''g\in U_{f, z}(k)$ and $g''g'\in U_{f, z'}(k)$, so
\[
p_{1*}((f\times\id_{\P^1})^*((g''\times \id)W)\cdot (X\times\{0\}-X\times\{\infty\})) = f^*(g''gz)-f^*(g''g'z')
\] 
and thus
\[
[f^*(gz)]=[f^*(g''gz)]=[f^*(g''g'z')]=[f^*(g'z')]\in \CH^d(X, X^*),
\]
completing the proof.
\end{proof}

\begin{remark} \label{rem:SmoothCase}Theorem~\ref{thm:HomogFunct} is a partial generalization of a classical result in intersection theory.

Let $f:Y\to Z$ be a morphism of a smooth $k$-schemes. We have the well-defined group homomorphism
\[
f^*:\CH^d(Z)\to \CH^d(Y).
\]
Moreover, if $z\in Z^d(Z)$ is a codimension $d$ cycle such that $f^{-1}(\supp(z))$ is a pure codimension $d$ closed subset of $Y$, then the cycle-theoretic pullback $f^*(z)\in Z^d(Y)$ is defined and $f^*([z])=[f^*(z)]\in \CH^d(Y)$ (see \cite[\S 20.4]{Fulton}).

If we now take $Z=H$, a homogeneous space for a $k$-rational linear algebraic group $G$, then $G(k)$ acts trivially on $\CH^d(H)$, that is $[z]=[g\cdot z]\in \CH^d(H)$ for $z\in Z^d(H)$, $g\in G(k)$
Moreover, by Kleiman's transversality theorem \cite[Theorem 2]{Kleiman}, there is a dense open subset $U_{f, z}\subset G$ such that for $g\in U_{f, z}(k)$, $f^{-1}(\supp(gz))$ has pure codimension $d$ on $Y$, so $f^*(g\cdot z)$ is a well-defined codimension $d$ cycle on $Y$. Thus,  by the above discussion, we have
\[
f^*([z])=f^*([g\cdot z])=[f^*(g\cdot z)].
\]
\end{remark}

\begin{proposition}\label{prop:HomogFunctoriality} Let $(X, X^*)$ be an admissible pair and  let $d=\dim_kX$. Suppose we have a homomorphism of strongly $k$-rational linear algebraic groups $\rho:G_1\to G_2$, quasi-projective homogeneous spaces $H_i$ for $G_i$, $i=1,2$, a morphism $p:H_1\to H_2$ that is  $G_1$-equivariant via $\rho$, and finally, a morphism $f:X\to H_1$. Then the diagram
\[
\xymatrix{
\CH^d(H_2)\ar[rr]^{p^*}\ar[dr]_{(pf)^*}&&\CH^d(H_1)\ar[dl]^{f^*}\\
&\CH^d(X, X^*)
}
\]
commutes.
\end{proposition}

\begin{proof} Take a codimension $d$ cycle $z\in Z^d(H_2)$. Applying Remark~\ref{rem:SmoothCase} to the morphisms $p$, we have the dense open subset $U_{p,z}\subset G_2$ such that for all $g_2\in  U_{p,z}(k)$ $p^{-1}(\supp(g\cdot z))$ has pure codimension $d$ on $H_1$, giving the well-defined codimension $d$ cycle $p^*(g_2\cdot z)$ on $H_1$, and we have
\[
p^*[z]=[p^*(g_2\cdot z)]\in \CH^d(H_1).
\]

Applying Theorem~\ref{thm:HomogFunct} to the morphism $f$, we see that for each $g_1\in U_{f, p^*(g_2\cdot z)}(k)$, the cycle $f^*(g_1\cdot p^*(g_2\cdot z))$ is a well-defined 0-cycle on $X$, supported in $X\setminus X^*$, and 
\[
[f^*(g_1\cdot p^*(g_2\cdot z))]=f^*([p^*(g_2\cdot z))]).
\]
We may assume that $\rho(g_1)g_2$ is in $U_{pf, z}(k)\cap U_{p,z}(k)$, so $(pf)^*(\rho(g_1)g_2\cdot z)$ is a well-defined 0-cycle on $X$, supported in $X\setminus X^*$,  with
\[
(pf)^*([z])=[(pf)^*(\rho(g_1)g_2\cdot z)]\in \CH^d(X, X^*),
\]
and $p^*(\rho(g_1)g_2\cdot z)$ is a well-defined codimension $d$ cycle on $H_1$, with
\[
p^*[z]=[p^*(\rho(g_1)g_2\cdot z)]\in \CH^d(H_1).
\]

We have the identity of 0-cycles on $X\setminus X^*$
\[
(pf)^*(\rho(g_1)g_2\cdot z)=f^*(p^*(\rho(g_1)g_2\cdot z))=f^*(g_1\cdot p^*(g_2\cdot z)),
\]
which gives the identity of classes in $\CH^d(X, X^*)$
\begin{align*}
(pf)^*([z])&=[(pf)^*(\rho(g_1)g_2\cdot z)]\\
&=[f^*(p^*(\rho(g_1)g_2\cdot z))]\\
&=[f^*(g_1\cdot p^*(g_2\cdot z))]\\
&=f^*([p^*(g_2\cdot z)])\\
&=f^*(p^*([z])).
\end{align*}
\end{proof}

\begin{proposition} \label{prop:PullbackFunct3} Let $(X, X^*)$, $(Y, Y^*)$ be admissible pairs with   $\dim_kX=\dim_kY$, and let $q:Y\to X$ be a morphism; we suppose  the conditions (i)-(iii) of Lemma~\ref{lem:EtalePullback} are satisfied for $q$, giving the pullback map 
\[
q^*:\CH^d(X, X^*)\to \CH^d(Y, Y^*),\ d=\dim_kX.
\]
Let $f:X\to H$ be a morphism to a homogeneous space $H$ for a strongly $k$-rational linear algebraic group $G$,  as in Theorem~\ref{thm:HomogFunct}. Then all the morphisms in the diagram
\[
\xymatrix{
&\CH^d(H)\ar[dl]_{f^*}\ar[dr]^{(fq)^*}\\
\CH^d(X, X^*)
\ar[rr]^{q^*}&&\CH^d(Y, Y^*)\ 
}
\]
are  defined, and the diagram commutes.
\end{proposition}

\begin{proof} Note that $(fq)^*:\CH^d(H)\to \CH^d(Y, Y^*)$, $f^*:\CH^d(H)\to \CH^d(X, X^*)$ are defined following Theorem~\ref{thm:HomogFunct}, and $q^*:\CH^d(X, X^*)\to \CH^d(Y, Y^*)$ is defined following Lemma~\ref{lem:EtalePullback}. For $z\in Z^d(H)$, take $g\in U_{z,f}(k)\cap U_{z,fq}(k)$, so the cycle pullbacks $f^*(g\cdot z)\in Z^d(X\setminus X^*)$ and $(fq)^*(g\cdot z)\in Z^d(Y\setminus Y^*)$ are defined and
\[
f^*([z])=[f^*(g\cdot z)]\in \CH^d(X, X^*),\ (fq)^*([z])=[(fq)^*(g\cdot z)]\in \CH^d(Y, Y^*).
\]
Moreover $f^{-1}(\supp(g\cdot z))$ has pure codimension $d$ on $X$ and $(fq)^{-1}(\supp(g\cdot z))$ has pure codimension $d$ on $Y$. Thus $q^{-1}(\supp(f^*(g\cdot z))\subset (fq)^{-1}(\supp(g\cdot z))$ also has pure codimension $d$ on $Y$, so $q^*(f^*(g\cdot z))$ is defined and 
$q^*(f^*(g\cdot z))=(fq)^*(g\cdot z)\in Z^d(Y\setminus Y^*)$. 
Since $q^*([f^*([z])])=[q^*(f^*(g\cdot z))]\in \CH^d(Y, Y^*)$, we thus have
\[
q^*([f^*([z])])=(fq)^*([z])
\]
in $\CH^d(Y, Y^*)$, as claimed.
\end{proof}

\begin{proposition}\label{prop:BCCompat} Let $k\hookrightarrow F$ be an extension of fields, let $(X, X^*)$ be an admissible pair over $k$, and let $d=\dim_kX$. We suppose the conditions of Lemma~\ref{lem:base-change} are satisfied, giving the pullback map $p_X^*:\CH^d(X, X^*)\to \CH^d(X_F, X_F^*)$.\\[5pt]
1. Let $f:X\to H$ be a morphism to a homogeneous space $H$ for a strongly $k$-rational linear algebraic group $G$, giving the base-change $f_F:X_F\to H_F:=H\times_kF$. Then $G_F:G\times_kF$ is a strongly $F$-rational linear algebraic group over $F$ and  $H_F$ is a homogeneous space for $G_F$. Moreover, letting $p_H:H_F\to H$, $p_X:X_F\to X$ be the projections, the diagram
\[
\xymatrix{
\CH^d(H)\ar[r]^{f^*}\ar[d]^{p_H^*}&\CH^d(X, X^*)\ar[d]^{p_X^*}\\
\CH^d(H_F)\ar[r]^{f_F^*}&\CH^d(X_F, X_F^*)
}
\]
commutes.   \\[2pt]
2. Let  $(Y, Y^*)$ be an admissible pair over $k$ with $\dim_kY=\dim_kX$ and let $f:Y\to X$ be a morphism satisfying the conditions of Lemma~\ref{lem:EtalePullback}, giving the pullback map $f^*:\CH^d(X, X^*)\to \CH^d(Y, Y^*)$. Suppose that $(X_F, X_F^*)$ and  $(Y_F, Y_F^*)$ satisfy the conditions of Lemma~\ref{lem:base-change}, giving the pullback maps $p_X^*:\CH^d(X, X^*)\to \CH^d(X_F, X_F^*)$, $p_Y^*:\CH^d(Y, Y^*)\to \CH^d(Y_F, Y_F^*)$. Then  $f_F:Y_F\to X_F$ satisfies the conditions of Lemma~\ref{lem:EtalePullback}, and the diagram
\[
\xymatrix{
\CH^d(X, X^*)\ar[r]^{f^*}\ar[d]^{p_X^*}&\CH^d(Y, Y^*)\ar[d]^{p_Y^*}\\
\CH^d(X_F, X_F^*)\ar[r]^{f_F^*}&\CH^d(Y_F, Y_F^*)
}
\]
commutes, where the vertical maps are the base-change maps of Lemma~\ref{lem:base-change}.
\end{proposition}

\begin{proof} For (1), if we have a cycle $z\in Z^d(H)$ such that $f^{-1}(\supp z)$ has codimension $d$ on $X$ and is contained in $X_\reg$, then the base-change $p_H^*(z)$ is in $Z^d(H_F)$ and $f_F^{-1}(\supp(p_H^*(z)))$ also has codimension $d$ on $X_F$. Thus the pull-backs $f^*(z)$, $f_F^*(p_H^*(z))$ are defined. Moreover, the restriction of $f$ to  $f:X_\reg\to H$ is an lci-morphism, since a locally closed immersion $X_\reg\hookrightarrow \P^N_k$ is a regular embedding, by \cite[Proposition 19.1.1]{EGAIV4}.  Thus $f_F:(X_\reg)_F\to H_F$ is also an lci morphism and 
 \[
p_X^*(f^*(z))=f_F^*(p_H^*(z))
\]
as cycles on $X_F$, using  \cite[Theorem 6.2, Theorem 6.6]{Fulton} and the fact that lci pullback is induced by the cycle-theoretic pullback when the latter is defined \cite[\S 7.1, \S 20.4]{Fulton}.

On the other hand, given an arbitrary  $y=\sum_in_iY_i\in Z^d(H)$ with the $Y_i$ integral, there is a dense open subset $U$ of $U_{f,y}:=\cap_iV(f,Y_i,X^*)$ such that for $g\in U(k)$, if we consider $g$ as in $G(F)=G_F(F)$ via the map $\Spec F\to \Spec k$ induced by $k\subset F$, $g$ is in $U_{f_F, p_H^*(y)}:=\cap_iV(f_F,(Y_i\times_kF)_\red,X^*)$. Thus $f^*([y])=[f^*(g\cdot y)]$, $f_F^*([p_H^*(y)])=[f_F^*(g\cdot p_H^*(y))]$,  and hence
\begin{multline*}
f_F^*(p_H^*([y]))=f_F^*([p_H^*(y)])=[f_F^*(g\cdot p_H^*(y))]=[f_F^*( p_H^*(g\cdot y))]\\
=[p_X^*(f^*(g\cdot y))]=p_X^*([f^*(g\cdot y)])=p_X^*(f^*([y])).
\end{multline*}

For (2), it is an easy consequence of well-known properties of flat base-change that $f_F:Y_F\to X_F$ satisfies the conditions of Lemma~\ref{lem:EtalePullback} since $f:Y\to X$ does.  Then (2) follows by the functoriality of cycle pullback for compositions of flat morphisms. \end{proof}

For $(X, X^*)$ an admissible pair with $X$ of dimension $d$, and a morphism $f:X\times\P^1\to H$, we now show that the pullback of a suitable translate of a reduced codimension $d$ closed subscheme $W$ on $H$ gives us a rational equivalence
\[
p_{1*}(|f^{-1}(gW)|\cdot(X\times\{0\}-X\times\{\infty\})\in R^d(X, X^*).
\]

\begin{proposition}\label{prop:RatEquivPullback}
Let $(X, X^*)$ be an admissible pair and  let $d=\dim_kX$. Let $H$ be a quasi-projective homogeneous space for a strongly $k$-rational linear algebraic group $G$,  let $f:X\times\P^1\to H$ be a morphism and let $W\subset H$ be a reduced closed subscheme of pure codimension $d$. Then there is a dense open subset $U^*_{f,W, X^*}$ of $G$ such that, for all $g\in U^*_{f,W, X^*}(k)$, we have
\begin{enumerate}
\item $f^{-1}(g\cdot W)$ is a  closed subscheme of $X\times\P^1$ of pure codimension $d$,   supported in $(X_\CM\amalg X(d-1)^0)\times\P^1$.
\item $f^{-1}(g\cdot W)\cap (X\times\{0,\infty\})$ is a closed subscheme  of pure codimension $d$, supported in $(X\setminus X^*)\times\{0,\infty\}$.
\item $f^{-1}(g\cdot W)\cap (X^*\times\P^1)$ is a finite set, and $f^{-1}(g\cdot W)$ is defined by a regular sequence (of length $d$) in a neighborhood of $f^{-1}(W)\cap X^*\times\P^1$.
\item Let $|gW|\in Z^d(H)$ denote the cycle associated to the closed subscheme $gW$, and let $f_0, f_\infty:X\to H$ be the restriction of $f$ to $X\times\{0\}$, $X\times\{\infty\}$, respectively. The cycle-theoretic pull-backs $f^*(|gW|)$, $f_0^*(|gW|)$ and $f_\infty^*(|gW|)$ are all defined, and 
$f_0^*(|gW|)$ and $f_\infty^*(|gW|)$ are supported in $X\setminus X^*$.
\item For $i=0,\infty$,  the cycle class $[f^*_i(|gW|)]\in \CH^d(X, X^*)$ satisfies
\[
[f^*_i(|g\cdot W|)]=f_i^*([W])\in\CH^d(X, X^*),
\]
where $f_i^*:\CH^d(H)\to \CH^d(X, X^*)$ is the pullback map of Theorem~\ref{thm:HomogFunct}
\item $p_{1*}(|f^{-1}(gW)|\cdot(X\times\{0\}-X\times\{\infty\})=f_0^*(|gW|)-f_\infty^*(|gW|)$, and this 0-cycle 
is in $R^d(X, X^*)$.
\end{enumerate}
\end{proposition}

\begin{proof} Let $S:= X\setminus (X_\CM\amalg X(d-1)^0)$ and let 
\[
U^*_{f,W, X^*}=U(f,W, X^*, S)\cap   U_{f_0,|W|}\cap U_{f_\infty, |W|}. 
\]
By Lemma~\ref{lem:TranslateCycle}, $U^*_{f,W, X^*}$ is a dense open subset of $G$ and $U^*_{f,W, X^*}(k)$ is dense in $U^*_{f,W, X^*}$. 

 (4) and (6) follow from (1)-(3), using Remark~\ref{rem:CodimNotDim}. Similarly, since $U^*_{f, W, X^*}\subset  U_{f_0,|W|}\cap U_{f_\infty, |W|}$, (5) follows from the definition of the pullback maps $f_0^*$, $f_\infty^*$ given in  Theorem~\ref{thm:HomogFunct}. 

For $g\in U(f,W, X^*, S)(k)$, the properties (1)  and (3) are satisfied by Lemma~\ref{lem:TranslateRelation} and similarly, property (2) is satisfied  for $g\in (U_{f_0,|W|}\cap U_{f_\infty, |W|})(k)$ by Theorem~\ref{thm:HomogFunct}. 
\end{proof}

\section{A presentation of $K_0$}   One main goal of this paper is to construct a Chern class map
\begin{equation}\label{eqn:ChernClassMap}
c_d:K_0(X)\to \CH^d(X, X^*)
\end{equation}
for $(X, X^*)$ an admissible pair with  $d=\dim_kX$. The idea is to use the classical Chern class map
\[
c_d:K_0(H)\to \CH^d(H)
\]
for $H$ a homogeneous space for a strongly $k$-rational linear algebraic group, and then take the pullback $f^*$, using the pullback map constructed in Theorem~\ref{thm:HomogFunct}. Roughly speaking, we will produce a representable presentation of $K_0(X)$ for connected, quasi-projective $X$, and use that to define the Chern class map $c_{\dim_kX}$. In more detail,  suppose we have a  rank $m$ locally free sheaf  $E$ and a globally generated invertible sheaf $L$ on $X$ such that  $L\otimes E$ is globally generated. Choosing global generators $e_1,\ldots, e_M$ for $L\otimes E$ and global generators $s_0,\ldots, s_N$ for $L$ gives  morphisms 
\[
\phi:X\to \Gr(m, M),\ \psi:X\to \P^N
\]
The map to the product
\[
f:=(\psi, \phi):X\to \P^N\times\Gr(m, M)
\]
then satisfies   $E\cong  f^*(\sO_{\P^N}(-1)\boxtimes E_{m, M})$, where $E_{m, M}$ is the tautological rank $m$ quotient sheaf on $\Gr(m, M)$.   We then ``define''
\begin{equation}\label{eqn:cd"def"}
c_d(E)= f^*(c_d(\sO_{\P^N}(-1)\boxtimes E_{m, M})\in \CH^d(X, X^*).
\end{equation}
Our task in this section is to give a presentation of $K_0(X)$, built out of maps of $X$ to suitable homogeneous spaces, building on the association of the vector bundle $f^*(\sO_{\P^N}(-1)\boxtimes E_{m, M})$ to a map $f:X\to\P^N\times\Gr(m, M)$. In the next section, we use this presentation and the pullback maps constructed in \S\ref{sec: Homogeneous} to construct $c_d$ as a map of pointed sets \eqref{eqn:ChernClassMap}, extending \eqref{eqn:cd"def"}.  

We now proceed with constructing our presentation of $K_0(X)$. In this section, we take $k$ to be an arbitrary field.

We have the Grassmannian $\Gr(m,M)$ with its universal quotient sheaf $\pi_{m,M}:\sO_{\Gr(m,M)}^n\to E_{m,M}$ with $E_{m,M}$ locally free of rank $m$. Recall that $\Gr(m,M)$ represents the presheaf of sets on $\Sch_k$
\[
Z\mapsto \{(E, (f_1,\ldots, f_M))\}/\sim,
\]
where $E$ is a locally free sheaf on $Z$ of rank $m$,  $f_1,\ldots, f_M\in H^0(Z, E)$ are global generators of $E$, and
\[
(E, (f_1,\ldots, f_M))\sim (E',(f'_1,\ldots, f'_M))
\]
if there is an isomorphism $\phi:E\to E'$ with $\phi(f_i)=f_i'$, $i=1,\ldots, M$. Identifying a choice of 
generating sections $f_1,\ldots, f_M\in H^0(Z, E)$ of $E$ with a surjection $f_\bullet:\sO_Z^M\to E$, the representability is given by sending a morphism $\phi:Z\to \Gr(m,M)$ to $\phi^*(\pi_{m,M}):\sO_Z^M\to \phi^*(E_{m,M})$. 

We have the action of $\GL_M$ on $\Gr(m,M)$ induced by change of basis for $\sO_{\Gr(m, M)}^M$, giving the induced action of $\GL_M(k)$ on $\Hom_k(Z, \Gr(m,M))$, natural in $Z$.

For integers $0<M_1\le M_2$, we have the closed immersion  $i_{m, M_2, M_1}:\Gr(m,M_1)\to \Gr(m,M_2)$
representing the surjection
\[
\sO_{\Gr(m,M_1)}^{M_2}=\sO_{\Gr(m,M_1)}^{M_1}\oplus \sO_{\Gr(m,M_1)}^{M_2-M_1}\xrightarrow{\pi_{m,M_1}\circ p_1}E_{m, M_1},
\]
where the first equality is the map $(s_1,\ldots, s_{M_2})\mapsto ((s_1,\ldots, s_{M_1}), (s_{M_1+1},\ldots, s_{M_2}))$. We have the corresponding homomorphism of group-schemes $\rho_{M_2, M_1}:\GL_{M_1}\to \GL_{M_2}$, inserting an $(M_2-M_1)\times (M_2-M_1)$ identity matrix in the low right corner, and the map $i_{m, M_2, M_1}$ is $\rho_{M_2, M_1}$ equivariant.

\begin{definition}\label{def:EquivRel1} We let $\sim_{m,M}$ denote the equivalence relation on the presheaf 
\[
\Hom_k(-, \Gr(m,M)):\Sch_k^\op\to \Sets
\]
generated by  the two relations:
\\[5pt]
1. $\phi\sim_{m,M}  g\cdot \phi$ for $\phi\in \Hom_k(Z, \Gr(m,M))$, $g\in \GL_M(k)$\\[2pt]
2. Given $M'<M$,  $\phi':Z\to \Gr(m, M')$ represented by $(E, f_1,\ldots, f_{M'})$,  and $M-M'$ generating sections $f_{M'+1},\ldots, f_M$ of $E$, let $\phi:Z\to \Gr(m,M)$ be the morphism given by $(E, f_1,\ldots, f_M)$. Then $\phi\sim_{m,M} i_{m, M,M'}\circ \phi'$.
\end{definition}

Let $\Vect_m^\gl(Z)$ denote the set of isomorphism classes of globally generated rank $m$ locally free sheaves on $Z$, defining the presheaf
\[
\Vect_m^\gl:\Sch_k^\op\to \Sets
\]
via pullback: for $f:Z'\to Z$, sending  $E$ to $f^*(E)$ on representing locally free sheaves.

\begin{lemma}\label{lem:GloballyGen}  Sending $f:Z\to \Gr(m,M)$ to the class of $f^*E_{m,M}$ in $\Vect_m^\gl(Z)$ induces a bijection
\[
\Phi^\gl(Z):\colim_M\Hom(Z, \Gr(m,M))/\sim_{m,M}\to \Vect_m^\gl(Z)
\]
natural in $Z$. Here the maps in the colimit are induced by the closed immersions $i_{M_2,M_1}$.
\end{lemma}

\begin{proof} We first check that $i_{m,M_2, M_1 *}:\Hom(-, \Gr(m, M_1))\to \Hom(-, \Gr(m, M_2))$ descends to a map $\Hom(-, \Gr(m, M_1))/\sim_{m, M_1}\to \Hom(-, \Gr(m,M_2))/\sim_{m, M_2}$. The relation (1) of Definition~\ref{def:EquivRel1} is respected since $i_{m,M_2, M_1}$ is $\rho_{M_2, M_1}$-equivariant. Suppose we are given $M'_1<M_1$, a rank $m$ locally free sheaf $E$ with generating sections  $f_1,\ldots, f_{M_1'}$, and $f_{M_1'+1},\ldots, f_{M_1}$ as in Definition~\ref{def:EquivRel1}(2),  giving morphisms $\phi':Z\to \Gr(m, M_1')$ representing $(E, 
f_1,\ldots, f_{M_1'})$ and $\phi:Z\to \Gr(m, M_1)$ representing $(E, 
f_1,\ldots, f_{M_1})$. We then have $\phi\sim i_{m, M_1,M_1'}\phi'$, and we need to show that $i_{m, M_2,M_1}\phi\sim_{m, M_2} i_{m, M_2, M_1}(i_{m, M_1, M_1'}\phi')$. 

We have $i_{m, M_2, M_1}\circ i_{m, M_1, M_1'}=i_{m, M_2, M_1'}$, 
$i_{m, M_2,M_1}\phi$ is the morphism representing $(E, f_1,\ldots, f_{M_1}, 0_{M_1+1},\ldots, 0_{M_2})$ and $i_{m, M_2, M_1'}\phi'$ is the morphism representing 
$(E, f_1,\ldots, f_{M_1'}, 0_{M_1'+1},\ldots, 0_{M_2})$. Using the relation $\sim_{m, M_2}$ for the generating sections $f_1,\ldots, f_{M_1'}$ and $(f_{M_1'+1},\ldots, f_{M_1},  0_{M_1+1},\ldots, 0_{M_2})$ gives
\[
i_{m, M_2,M_1}\phi\sim_{m, M_2} i_{m, M_2, M_1'}\phi'= i_{m, M_2, M_1}(i_{m, M_1, M_1'}\phi')
\]
as desired.

Since the relation $\sim_{m,M}$ and the maps $i_{m, M_2,M_1}$ only change the choice of 
generating sections, and not the isomorphism class of the locally free sheaf, the map $(f:Z\to \Gr(m,M))\mapsto f^*E_{m, M}$ descends to a well-defined natural transformation 
\[
\Phi^\gl:\colim_M\Hom(-, \Gr(m,M))/\sim_{m,M}\to \Vect_m^\gl. 
\]
As every  rank $m$ vector bundle $E$ generated by $M$ sections is isomorphic to $f^*E_{m,M}$ for suitable $f:Z\to \Gr(m,M)$, we see that $\Phi^\gl$ is surjective.

Now, given a globally generated rank $m$ $E$ on $Z$, and two sets of global generators $(f_1,\ldots , f_M)$, $(f_1',\ldots , f'_{M'})$ with corresponding morphisms $\phi$, $\phi'$, we have the global generators $(f_1,\ldots , f_{M+M'})$, with $f_{M+i}=f'_i$ for $i=1,\ldots, M'$, with corresponding morphism $\phi''$, and  
\[
i_{m, M+M',M}\circ \phi\sim_{m, M+M'}
\phi''\sim_{m, M+M'}i_{m, M+M',M'}\circ\phi'.
\]
Thus,  $\phi$ and $\phi'$ map to the same element in  
$\colim_M\Hom(Z, \Gr(m,M))/\sim_{m,M}$, and hence $\Phi^\gl$ is injective.
\end{proof}

We now promote this presentation of $\Vect_m^\gl(-)$ to a presentation of $\Vect_m(-)$ as a presheaf on $\Sch_k^\qp$. For the remainder of this section, we work in the category $\Sch_k^\qp$.

Consider the product $\P^N\times\Gr(m, M)$. We have the universal rank $m$ quotient sheaf $\sO_{\Gr(m,M)}^M\to E_{m,M}$ on $\Gr(m,M)$ and its twist $E_{m,M}(-1):=\sO_{\P^N}(-1)\boxtimes E_{m,M}$.

Let $E$ be a locally free coherent sheaf of rank $m$ on $X$, and let $L$ be a globally generated invertible sheaf on $X$ such that $E\otimes L$ is globally generated. Let $V_L\subset H^0(X, L)$ be a finite dimensional $k$-subspace that generates $L$ under the canonical map 
\[
V_L\otimes_k\sO_X\to H^0(X, L)\otimes_k\sO_X\to L
\]
and let $V_{E\otimes L}\subset H^0(X, E\otimes L)$ be a  finite dimensional $k$-subspace that generates $E\otimes L$ under the similarly defined map $V_{E\otimes L}\otimes\sO_X\to E\otimes L$. Choosing isomorphisms $V_L\cong k^{N+1}$, $V_{E\otimes L}\cong k^M$, we have the morphisms
\[
g_{V_L}:X\to \P^N,\ g_{V_{E\otimes L}}:X\to \Gr(m, M)
\]
representing the surjections 
\[
\sO_X^{N+1}\cong V_L\otimes \sO_X\to L,\ \sO_X^M\cong V_{E\otimes L}\otimes \sO_X\to E\otimes L.
\]
Let $f_{V_L, V_{E\otimes L}}:=(g_{V_L}, g_{V_{E\otimes L}}):X\to \P^N\times \Gr(m, M)$. Then
\begin{equation}\label{eqn:Pullback1}
g_{V_L}^*(\sO_{\P^N}(1))\cong L,\ g_{V_{E\otimes L}}^*E_{m, M}\cong E\otimes L,\ f_{V_L, V_{E\otimes L}}^*(E_{m,M}(-1))\cong E.
\end{equation}

Let 
\[
\ver_{N,N'}:\P^N\times \P^{N'}\to \P^{N''},\quad N''=(N+1)(N'+1)-1,
\]
be the Veronese embedding given by a choice of $k$-basis for $H^0(\P^N\times\P^{N'}, \sO(1,1))=H^0(\P^N\times, \sO(1))\otimes_kH^0(\P^{N'}, \sO(1))$.

We have the twists $E_{m,M}(-1,0)$ and $E_{m,M}(0,1)\cong E_{m,M}(-1,0)(1,1)$ on $\P^N\times \P^{N'}\times \Gr(m,N)$. 

By the K\"unneth formula,  the product map
\[
H^0(\P^N, \sO(1))\otimes_k H^0(\Gr(m,M), E_{m,M})\to H^0(\P^N\times\Gr(m, M), E_{m, M}(1))
\]
is an isomorphism, in particular, the product
\[
H^0(\P^N, \sO(1))\otimes_k H^0(\P^{N'}, \sO(1))\to H^0(\P^N\times\P^{N'}, \sO(1,1))
\]
is an isomorphism.

The sheaf $E_{m,M}(1)$ on $\P^{N'}\times \Gr(m,M)$ is generated by global sections, and thus the 
canonical surjection $H^0(\P^{N'}\times \Gr(m,M), E_{m,M}(1))\otimes\sO_{\P^{N'}\times \Gr(m,M)}\to E_{m,M}(1)$ gives the morphism
\[
\phi_{N', m, M}:\P^{N'}\times \Gr(m,M)\to  \Gr(m,M'')
\]
where 
\[
M'':=\dim_kH^0(\P^{N'}\times \Gr(m,M), E_{m,M}(1))=(N'+1)\cdot M,
\]
and we choose a basis for $H^0(\P^{N'}\times \Gr(m,M), E_{m,M}(1))$.

After choosing a basis for $H^0(\P^{N}\times\P^{N'},\sO(1,1))$, we have the corresponding Veronese embedding
\[
\ver_{N, N'}:\P^N\times\P^{N'}\to \P^{N''};\ N''=(N'+1)(N+1)-1,
\]
giving us the map
\begin{equation}\label{Veronese2}
\lambda_{N, N', m,M}:\P^N\times\P^{N'}\times\Gr(m, M)\to \P^{N''}\times \Gr(m, M''),
\end{equation}
where
\[
\lambda_{N, N', m,M}:=(\ver_{N, N'}\circ p_{12}, \phi_{N', m, M}\circ p_{23}).
\]

\begin{lemma} \label{lem:TensorIdentity} We have\\[5pt]
1. $\lambda_{N, N', m,M}^*(E_{m, M''}(-1))\cong E_{m, M}(0,1)(-1,-1)\cong E_{m, M}(-1,0)=p_{13}^*(E_{m, M}(-1))$.\\[2pt]
2. Give $\P^N$, $\P^{N'}$, $\Gr(m, M)$ and $\Gr(m, M'')$ the usual actions of $\GL_{N+1}$,  $\GL_{N'+1}$, $\GL_M$ and $\GL_{M''}$.  Let $\rho:=\rho_{N,N'}:\GL_{N+1}\times \GL_{N'+1}\to \GL_{N''+1}$ be the tensor product representation, corresponding to the isomorphism induced by $\psi_{N, N'}$
\[
H^0(\P^N, \sO(1))\otimes_kH^0(\P^{N'}, \sO(1))\cong H^0(\P^{N''}, \sO(1))
\]
and let $\tau:=\tau_{N', m, M}: \GL_{N'+1}\times \GL_M\to \GL_{M''}$ be the representation induced by the canonical isomorphism
\begin{align*}
H^0(\P^{N'}, \sO(1))\otimes_k H^0(\Gr(m, M), E_{m, M})&\cong H^0(\P^{N'}\times \Gr(m, M), E_{m, M}(1))\\&\cong H^0(\Gr(m, M''), E_{m, M''})
\end{align*}
given by $\phi_{N', m, M}$. Then $\lambda_{N, N', m,M}$ is equivariant with respect to the homomorphism
\[
\GL_{N+1}\times \GL_{N'+1}\times \GL_M\to \GL_{N''+1}\times \GL_{M''}
\]
sending $(g_1, g_2, g_3)$ to $(\rho(g_1, g_2), \tau(g_2, g_3))$.\\[2pt]
3. Identify $\P^N$ with $\Gr(1, N+1)$, and write $i_{N', N}:\P^N\to \P^{N'}$ for $i_{1, N'+1, N+1}:\Gr(1, N+1)\to \Gr(1, N'+1)$.

For $N_1\le N_2$, $N_1'\le N_2'$,  $M_1\le M_2$, we have $N_1''\le N_2''$,  $M_1''\le M_2''$ and the diagram
\[
\xymatrixcolsep{70pt}
\xymatrix{
\P^{N_1}\times\P^{N_1'}\times \Gr(m, M_1)\ar[d]^{i_{N_2, N_1}\times i_{N_2', N_1'}\times i_{m, M_2, M_1}}\ar[r]^{\lambda_{N_1, N_1', m, M_1}}&
\P^{N_1''}\times\Gr(m, M_1'')\ar[d]^{i_{N_2'', N_1''}\times i_{m, M_2'', M_1''}}\\
\P^{N_2}\times\P^{N_2'}\times\Gr(m, M_2)\ar[r]^{\lambda_{N_2, N_2', m, M_2}}
&\P^{N_2''}\times\Gr(m, M_2'')
}
\]
commutes, up to the action of some $g\in\GL_{N_2''+1}(k)\times \GL_{M_2''}(k)$ on $\P^{N_2''}\times\Gr(m, M_2'')$.
\end{lemma}

\begin{proof} For  (1), we have
\begin{align*}
\lambda_{N, N', m,M}^*(E_{m, M''}(-1))&=
(p_{12}^*\ver_{N, N'}^*\sO(-1)\otimes  p_{23}^*\phi_{N', m, M}^*E_{m, M''})\\
&=p_{23}^*E_{m, M}\otimes \sO(0,1)\otimes \sO(-1,-1)\\
&=E_{m, M}(-1,0).
\end{align*}
(2) and (3) are easy computations.
\end{proof}

\begin{Not} Identifying $\P^N$ with $\Gr(1, N+1)$, we write $\sim_N$ for the equivalence relation $\sim_{1, N+1}$ on $\Hom(-, \Gr(1, N+1))$. We let $\sim_{N, m, M}$ be the equivalence relation $\sim_{N}\times \sim_{m, M}$ on $\Hom(-, \P^N\times\Gr(m,M))$, and write $\overline{\Hom}(-, \P^N\times\Gr(m, M))$ for ${\Hom}(-, \P^N\times\Gr(m, M))/\sim_{N, m, M}$.
\end{Not}

As above,  let $\N_+$ be the set of positive integers, ordered by the usual $\le$. We have the functor
\[
\P^{?_1}\times\Gr(m, ?_2):\N_+\times\N_+\to \Sch_k;\quad (N, M)\mapsto \P^N\times\Gr(m, M)
\]
where $(N,M)\le (N', M')$ gets sent to $i_{N', N}\times i_{m, M', M}$. The induced functor
\[
\Hom(-, \P^{?_1}\times\Gr(m, ?_2)):\N_+\times\N_+\to \Fun(\Sch^\op_k, \Sets)
\]
passes to the functor 
\[
\overline{\Hom}(-, \P^{?_1}\times\Gr(m, ?_2)):\N_+\times\N_+\to \Fun(\Sch^\op_k, \Sets).
\]

Similarly, we have the functor
\[
\Hom(-, \P^{?_1}\times\P^{?_2}\times\Gr(m, ?_3)):\N_+\times\N_+\times\N_+\to \Fun(\Sch_k^\op, \Sets)
\]
Defining the relation $\sim'_{N, N', M}$ on $\Hom(-, \P^N\times\P^{N'}\times\Gr(m, M))$ by
\[
f\sim'_{N , N', M} (g_1, g_2, g_3)\circ f,\ (g_1, g_2, g_3)\in \GL_{N+1}(k)\times\GL_{N'+1}(k)\times\GL_M(k)
\]
gives the functor
\[
\widetilde{\Hom}(-, \P^{?_1}\times\P^{?_2}\times\Gr(m, ?_3)):=\Hom(-, \P^{?_1}\times\P^{?_2}\times\Gr(m, ?_3))/\sim'_{?_1,?_2,?_3}
\]

\begin{lemma}  The maps $p_{13}:\P^N\times\P^{N'}\times\Gr(m, M)\to \P^N\times\Gr(m, M)$ and $\lambda_{N, N', m, M}: \P^N\times\P^{N'}\times\Gr(m, M)\to \P^{N''}\times \Gr(m, M'')$ induce maps of presheaves on $\Sch_k$
\begin{multline*}
\colim_{\N_+\times\N_+\times\N_+}\widetilde{\Hom}(-, \P^{?_1}\times\P^{?_2}\times\Gr(m, ?_3))\\
\xymatrix{\ar@<5pt>[r]^{p_{13}(-)}\ar@<-5pt>[r]_{\lambda(-,m)}&}
\colim_{\N_+\times\N_+}\overline{\Hom}(-, \P^{?_1}\times\Gr(m, ?_2)).
\end{multline*}
Here the map $p_{13}(-)$ uses the map $p_{13}:\N_+\times\N_+\times\N_+\to \N_+\times\N_+$ and the map $\lambda(-,m)$ uses the map $(N, N', M)\mapsto (N'', M'')$ on the index posets to define the maps on the colimits. 
\end{lemma}

\begin{proof} The fact that $p_{13}(-)$ and $\lambda(-,m)$ are well-defined follows directly from Lemma~\ref{lem:TensorIdentity}.
\end{proof}

\begin{proposition}\label{prop:VectmPresentation} Let $X$ be a quasi-projective $k$-scheme. Sending $f:X\to \P^N\times\Gr(m, M)$ to $f^*E_{m, M}(-1)\in \Vect_m(X)$ defines a bijection $\Psi(X, m)$:
\begin{align*}
\coeq[\colim_{\N_+\times\N_+\times\N_+}&\widetilde{\Hom}(X, \P^{?_1}\times\P^{?_2}\times\Gr(m, ?_3))\\&\xymatrix{\ar@<5pt>[r]^{p_{13}(X)}\ar@<-5pt>[r]_{\lambda(X,m)}&}
\colim_{\N_+\times\N_+}\overline{\Hom}(X, \P^{?_1}\times\Gr(m, ?_2))]\xymatrix{\ar[r]^{\Psi(X, m)}_\sim&} \Vect_m(X),
\end{align*}
natural in $X\in \Sch^\qp_k$.
\end{proposition}

\begin{proof} Let $\Pic^+(X)$ be the set of isomorphism classes of globally generated invertible sheaves on $X$, that is, $\Pic^+(X)=\Vect_1^\gl(X)$. It follows from Lemma~\ref{lem:GloballyGen} that sending $f:X\to \P^N$ to $f^*\sO_{\P^N}(1)$ induces a bijection
\begin{equation}\label{eqn:PicDescr}
\colim_{\N_+}\overline{\Hom}(X, \P^N)\xrightarrow{\sim} \Pic^+(X).
\end{equation}
Fixing one $L\in \Pic^+(X)$, let $\overline{\Hom}(X, \P^N)_L$ be the subset of $\overline{\Hom}(X, \P^N)$ consisting of the equivalence classes of morphisms $f:X\to \P^N$ such that $f^*\sO_{\P^N}(1)\cong L$. Similarly, let $\Vect_m^L(X)$ be the set of isomorphism classes of rank $m$ locally free sheaves $E$ such that $E\otimes L$ is globally generated and let $\overline{\Hom}(X, \P^N\times\Gr(m, M))_L$ be the subset of $\overline{\Hom}(X, \P^N\times\Gr(m, M))$ consisting of equivalence classes of maps $f:X\to \P^N\times \Gr(m, M)$ such that $p_1\circ f$ is in 
$\overline{\Hom}(X, \P^N)_L$. Then it is easy to see that for each $f\in \overline{\Hom}(X, \P^N\times\Gr(m, M))_L$, $(p_2\circ f)^*(E_{m, M})$ is isomorphic to a sheaf of the form $E\otimes L$, with $E\in \Vect_m^L(X)$, and   $f^*(E_{m, M}(-1))\cong E$. Thus, we have the well-defined map
\[
\phi_{N, m, M}^L:\overline{\Hom}(X, \P^N\times\Gr(m, M))_L\to \Vect_m^L(X).
\]

Arguing as in the proof of Lemma~\ref{lem:GloballyGen}, we see that the maps $\phi_{N, m, M}^L$ induce a bijection
\[
\phi^L:\colim_{\N_+\times\N_+}\overline{\Hom}(X, \P^N\times\Gr(m, M))_L\to \Vect_m^L(X).
\]
If we now consider the entire colimit, we have the map
\[
\phi:\colim_{\N_+\times\N_+}\overline{\Hom}(X, \P^N\times\Gr(m, M))\to \Pic^+(X)\times\Vect_m(X).
\]
sending $f:X\to \P^N\times\Gr(m, M)$ to the pair of classes $((p_1\circ f)^*(\sO_{\P^N}(1)), f^*(E_{m, M}(-1)))$. From our description of $\phi_L$, we see that $\phi$ is injective, with image of $\phi$ contained the subset 
\[
[\Pic^+(X)\times \Vect_m(X)]^+=\{(L, E)\mid E\otimes L\text{ is globally generated}\}\subset 
\Pic^+(X)\times \Vect_m(X).
\]
From \eqref{eqn:PicDescr}, it follows that 
\[
\phi:\colim_{\N_+\times\N_+}\overline{\Hom}(X, \P^N\times\Gr(m, M))\to [\Pic^+(X)\times \Vect_m(X)]^+
\]
is a bijection.

Let $\le^*$ be partial order on $\Pic^+(X)$ defined by $L\le^* L\otimes L'$ for $L, L'\in \Pic^+(X)$. Note that $(\Pic^+(X), \le^*)$ is right filtering since
\[
L\le^* L\otimes L',\ L'\le^* L'\otimes L=L\otimes L'
\]
for $L, L'\in \Pic^+(X)$. 
Given $L'\in \Pic^+(X)$, we have the map
\[
-\otimes L':[\Pic^+(X)\times \Vect_m(X)]^+\to [\Pic^+(X)\times \Vect_m(X)]^+
\]
sending $(L, E)$ to $(L\otimes L', E)$. This gives us the $(\Pic^+(X), \le^*)$ parametrized directed system $L'\mapsto -\otimes L'$. We claim that the projection $p_2:[\Pic^+(X)\times \Vect_m(X)]^+\to \Vect_m(X)$ induces a bijection
\[
\colim_{\Pic^+(X)} [\Pic^+(X)\times \Vect_m(X)]^+\xymatrix{\ar[r]^{\colim\, p_2}_\sim&}\Vect_m(X)
\]
First of all, as $X$ is quasi-projective, each $E\in \Vect_m(X)$ admits an $L\in \Pic^+(X)$ with $E\otimes L$ globally generated, so $p_2$ is surjective, hence $\colim\, p_2$ is also surjective. For injectivity, if two elements $(L, E)$, $(L', E')$ both go to the same element $E''\in \Vect_m(X)$, then $E=E'=E''$. Moreover, we have 
\[
(-\otimes L')(L, E)=(L\otimes L', E)=(-\otimes L)(L', E)
\]
so $(L,E)=(L',E)$ in the colimit. 

We factor the map 
\[
\Hom(X, \P^N\times\Gr(m, M))\to \Vect_m(X),\quad f\mapsto f^*(E_{m,M}(-1)),
\]
through  $p_2:[\Pic^+(X)\times \Vect_m(X)]^+\to \Vect_m(X)$ by sending $f$ to the pair
\[
((p_1\circ f)^*\sO_{\P^N}(1), f^*E_{m,M}(-1))\in [\Pic^+(X)\times \Vect_m(X)]^+.
\]

To conclude,  let $f:X\to \P^N\times\P^{N'}\times\Gr(m, M)$ be a morphism, and let $E=(p_{13}\circ f)^*(E_{m, M}(-1))$, $L=(p_1\circ f)^*(\sO_{\P^N}(1))$ and $L'=(p_2\circ f)^*(\sO_{\P^{N'}}(1))$. By construction, we have $E\otimes L=(p_3\circ f)^*(E_{m, M})$. We  have the map \eqref{Veronese2}
\[
\lambda_{N,N', ,M}:\P^N\times\P^{N'}\times\Gr(m,M)\to \P^{N''}\times\Gr(m, M'')
\]
with
\[
(p_1\circ \lambda_{N,N', ,M}\circ f)^*(\sO_{\P^{N''}}(1))= (\ver_{N,N'}\circ f)^*(\sO_{\P^{N''}}(1))=L\otimes L',
\]
and by Lemma~\ref{lem:TensorIdentity}(1)
\[
(\lambda_{N,N', ,M}\circ f)^*(E_{m, M''}(-1))\cong (p_{13}\circ f)^*(E_{m, M}(-1)) = E.
\]
Thus we have 
\[
((p_1\circ f)^*(\sO_{\P^N}(1)), (p_{13}\circ f)^*(E_{m, M}(-1))=(L, E)\in \Pic^+(X)\times \Vect_m(X)]^+
\]
and
\begin{multline*}
((p_1\circ\lambda_{N,N', ,M}\circ f)^*(\sO_{\P^{N''}}(1)),  (\lambda_{N,N', ,M}\circ f)^*(E_{m, M''}(-1))\\=(L\otimes L', E)\in [\Pic^+(X)\times \Vect_m(X)]^+.
\end{multline*}
This gives us
\begin{multline*}
(L\otimes L', E)= ((p_1\circ\lambda(X,m)(f))^*(\sO_{\P^{N''}}(1)),  \lambda(X,m)(f)^*(E_{m, M''}(-1))\\ = (-\otimes L')((p_1\circ f)^*(\sO_{\P^N}(1)), p_{13}(X)(f)^*(E_{m, M}(-1))\\=(-\otimes L')(L,E)\in [\Pic^+(X)\times \Vect_m(X)]^+.
\end{multline*}
Moreover, each such pair of elements $(L\otimes L', E)$, $(L, E)$ in  $[\Pic^+(X)\times \Vect_m(X)]^+$, $L, L'\in \Pic^+(X)$, arises in this way from some morphism 
\[
f:X\to \P^N\times\P^{N'}\times \Gr(m, M).
\]
Thus, we have the identification
\[
\coeq(p_{13}(X), \lambda(X, m))\cong \colim_{\Pic^+(X)} [\Pic^+(X)\times \Vect_m(X)]^+\cong \Vect_m(X),
\]
completing the proof.
 \end{proof}

Finally, we promote $\Psi(X,m)$ to give a presentation of $K_0(X)$. For simplicity, we will assume that $X$ is connected.

 Let $\ses_m(X)$ denote the set of isomorphism classes of short exact sequences of locally free sheaves on $X$,
\[
E_\bullet:=0\to E'\to E\to E''\to 0,
\]
with $E$ of rank $m$. We have the two morphisms
\[
\alpha, \beta:\ses_m(X)\to \Vect_m(X)
\]
with
\[
\alpha(E_\bullet)=E'\oplus E'',\ \beta(E_\bullet)=E.
\]

For $(F_1, F_1')\in \Vect_{m_1}(X)\times\Vect_{m_1'}(X)$, $(F_2, F_2')\in \Vect_{m_2}(X)\times\Vect_{m_2'}(X)$, we write
\[
(F_1, F_1')\sim_{K_0}(F_2, F_2')
\]
if there are $E_\bullet\in \ses_m(X)$, $E'_\bullet\in \ses_{m'}(X)$ with
\[
F_2\oplus \alpha(E'_\bullet)=F_1\oplus \alpha(E_\bullet),\ F_2'\oplus \beta(E'_\bullet)=F_1'\oplus \beta(E_\bullet)
\]
It is easy to see that $\sim_{K_0}$ defines an equivalence relation on $\amalg_{m,m'\ge0}\Vect_m(X)\times\Vect_{m'}(X)$.

Let $\Sets_\bullet$ denote the category of pointed sets and let $\Sch_k^{\qp,\con}\subset \Sch_k^\qp$ denote the full subcategory of connected $k$-schemes. For $F\in \Vect_m(X)$, we have the $K_0$-class $[F]\in K_0(X)$.

\begin{lemma} Take $X\in \Sch_k^{\qp,\con}$.   Then sending $(F, F')\in \Vect_{m}(X)\times\Vect_{m'}(X)$ to   $[F]-[F']\in K_0(X)$ defines a bijection in  $\Sets_\bullet$. 
\[
\amalg_{m,m'\ge0}\Vect_m(X)\times\Vect_{m'}(X)/\sim_{K_0}\xrightarrow{\sim}K_0(X),
\]
natural in $X\in \Sch_k^{\qp,\con}$.
Here $K_0(X)$ is pointed by 0 and $\amalg_{m,m'\ge0}\Vect_m(X)\times\Vect_{m'}(X)/\sim_{K_0}$ is pointed by the class of $(0,0)\in \Vect_0(X)\times\Vect_0(X)$.
\end{lemma}

\begin{proof}  Note that $\amalg_{m\ge0}\Vect_m(X)$ is a commutative monoid under direct sum.  If we restrict $\sim_{K_0}$ to the case of short exact sequences of the form $0\to E\xrightarrow{\id}E\to 0\to 0$, defining the equivalence relation $\sim_{\oplus}$, we see that sending $(F, F')\in \Vect_m(X)\times\Vect_{m'}(X)$ to the formal difference $F-F'$ in the group completion $[\amalg_{m\ge0}\Vect_m(X)]^\gp$ gives a bijection of pointed sets
\begin{equation}\label{eqn:GpCompletionMap}
\amalg_{m,m'\ge0}\Vect_m(X)\times\Vect_{m'}(X)/\sim_{\oplus}\to [\amalg_{m\ge0}\Vect_m(X)]^\gp.
\end{equation}
Also, as abelian groups, $[\amalg_{m\ge0}\Vect_m(X)]^\gp$ is canonically isomorphic to $K_0^\oplus(X)$, defined as the free abelian group on isomorphism classes of locally free coherent sheaves on $X$, modulo the subgroup generated by  $\{[E\oplus E']-[E]-[E']\mid (E,E')\in  \amalg_{m,m'\ge0}\Vect_m(X)\times\Vect_{m'}(X)\}$.  

From this, we see that the map \eqref{eqn:GpCompletionMap}  passes to a bijection (in $\Sets_\bullet$) of
\[
\amalg_{m,m'\ge0}\Vect_m(X)\times\Vect_{m'}(X)/\sim_{K_0}
\]
with the quotient (in the category $\Ab$)  of $K_0^\oplus(X)$ by the relation $[E']+[E'']\sim [E]$ for each  short exact sequence
\[
0\to E'\to E\to E''\to 0
\]
on $X$, which gives back the definition of $K_0(X)$.
\end{proof}

We finish the discussion by showing how to represent the relation $\sim_{K_0}$ in homogeneous varieties.

For this, we first describe representing the direct sum operation on twisted locally free coherent sheaves. 
We have the usual direct sum map
\[
\oplus:=\oplus_{(m_1, M_1), (m_2, M_2)}:\Gr(m_1, M_1)\times\Gr(m_2, M_2)\to
\Gr(m_1+m_2, M_1+M_2)
\]
classifying the surjection
\begin{multline*}
k^{M_1+M_2}\otimes\sO_{\Gr(m_1, M_1)\times\Gr(m_2, M_2)}\\
\cong 
[p_1^*(H^0(\Gr(m_1, M_1), E_{m_1, M_1})\otimes\sO_{\Gr(m_1, M_1)}\\ \oplus
p_2^*(H^0(\Gr(m_2, M_2), E_{m_2, M_2})\otimes\sO_{\Gr(m_2, M_2)}]\\\to
p_1^*E_{m_1,M_1}\oplus p_2^*E_{m_2, M_2}.
\end{multline*}
For the twisted version, we need a slightly more involved construction.

Consider the direct sum
\[
(p_{2}^*E_{m_1, M_1})(0,1)\oplus (p_{4}^*E_{m_2, M_2})(1,0)
\]
on $(\P^{N_1}\times \Gr(m_1, M_1))\times (\P^{N_2}\times \Gr(m_2, M_2))$, where the twist $(-)(a,b)$ means $(-)\otimes p_1^*\sO_{\P^{N_1}}(a)\otimes p_3^*\sO_{\P^{N_2}}(b)$. This direct sum is globally generated, by
\begin{multline*}
H^0(\Gr(m_1, M_1), E_{m_1, M_1})\otimes H^0(\P^{N_2}, \sO_{\P^{N_2}}(1))\\\oplus 
H^0(\Gr(m_2, M_2), E_{m_2, M_2})\otimes H^0(\P^{N_1}, \sO_{\P^{N_1}}(1))\cong k^M,
\end{multline*}
with
\[
M=M_1\cdot(N_2+1)+M_2\cdot(N_1+1).
\]
A choice of the above isomorphism gives the representing morphism
\[
\sigma_{(N_1, m_1, M_1),(N_2, m_2, M_2)}:
(\P^{N_1}\times \Gr(m_1, M_1))\times (\P^{N_2}\times \Gr(m_2, M_2))\to \Gr(m_1+m_2, M)
\]

We also have the Veronese embedding 
\[
\ver_{N_1, N_2}:\P^{N_1}\times \P^{N_2}\to \P^N,\ N=(N_1+1)(N_2+1)-1, 
\]
giving the map 
\[
\widetilde{\oplus}:(\P^{N_1}\times \Gr(m_1, M_1))\times (\P^{N_2}\times \Gr(m_2, M_2))\to
\P^N\times \Gr(m_1+m_2, M),
\]
defined as 
\[
\widetilde{\oplus}:=\widetilde{\oplus}_{(N_1, m_1, M_1),(N_2, m_2, M_2)}:=(\ver_{N_1, N_2}\circ p_{13}, \sigma_{(N_1, m_1, M_1),(N_2, m_2, M_2)}).
\]
Letting $m=m_1+m_2$, this satisfies
\begin{multline}\label{eqn:DirectSumPullback}
\widetilde{\oplus}^*(E_{m, M}(-1))\cong \sigma_{(N_1, m_1, M_1),(N_2, m_2, M_2)}^*(E_{m, M})(-1,-1)\\\cong
p_2^*E_{m_1, M_1}(-1,0)\oplus p_4^*E_{m_2, M_2}(0,-1),
\end{multline}
giving us the following computation.
\begin{lemma}\label{lem:TwistedDirectSum} Let $E_1, E_2$ be locally free  sheaves on $X$ of rank $m_1$, $m_2$, $L_1$, $L_2$ globally generated invertible sheaves on $X$ such that $E_1\otimes L_1$ and $E_2\otimes L_2$ are globally generated. For $i=1,2$, let $g_i:X\to \Gr(m_i, M_i)$ be the morphism gotten by choosing $M_i$ generating global sections of $E_i\otimes L_i$,  let $\iota_i:X\to \P^{N_i}$ be the morphism gotten by choosing $N_i+1$ generating global sections of $L_i$,  let $f_i:=(\iota_i, g_i):X\to 
\P^{N_i}\times\Gr(m_i, M_i)$ and let $m=m_1+m_2$. Then
\[
f_i^*(E_{m_i, M_i}(-1))\cong E_i,\ i=1,2,
\]
and
\[
(\widetilde{\oplus}\circ (f_1, f_2))^*(E_{m, M}(-1))\cong E_1\oplus E_2.
\]
\end{lemma}

\begin{proof} We have already seen that $f_i^*(E_{m_i, M_i}(-1))\cong E_i$, and the second identity follows from this and \eqref{eqn:DirectSumPullback}.
\end{proof}

 Let $\sO_{\Gr(m, M)}(1)$ denote the invertible sheaf $\det E_{m, M}$. The sheaf $\sO_{\Gr(m, M)}(1)$  is very ample and a suitable choice of basis for $H^0(\Gr(m, M), \sO_{\Gr(m, M)}(1))$ defines the classical Pl\"ucker embedding
\[
\plu_{m, M}:\Gr(m, M)\to \P^{N(m, M)},\ N(m, M)=\binom{M}{m}-1.
\]
More generally, for $r\ge1$, choosing a basis for 
$H^0(\Gr(m, M), \sO_{\Gr(m, M)}(r))\cong k^{N(m, M,r)+1}$ defines  an  embedding
\[
\plu_{m, M,r}:\Gr(m, M)\to \P^{N(m, M,r)}.
\]

Let $\Fl(m_\bullet, M)$, $m_\bullet:=(m_t, m_q)$, be the partial flag variety representing sequences of surjective maps of locally free sheaves
\[
\sO_Y^M\to E^t\to E^q
\]
with $E^t$ of rank $m_t$ and $E^q$ of rank $m_q$. We have the tautological sequence of surjections
\[
\sO_{\Fl(m_\bullet, M)}^M\to E^t_{m_\bullet, M}\to E^q_{m_\bullet, M},
\]
and  the tautological exact sequence
\begin{equation}\label{eqn:TautExSeq}
E_{m_\bullet, M}:\ 0\to E^s_{m_\bullet, M}\to E^t_{m_\bullet, M}\to E^q_{m_\bullet, M}\to 0
\end{equation}
on $\Fl(m_\bullet, M)$, with $E^s_{m_\bullet, M}$ of rank   $m_s:=m_t-m_q$.   Note that, although $E^t_{m_\bullet, M}$ and $E^q_{m_\bullet, M}$ are globally generated (by $\sO_{\Fl(m_\bullet, M)}^M$), this is not in general the case for $E^s_{m_\bullet, M}$. In any case, the surjections $\sO_{\Fl(m_\bullet, M)}^M\to E^t_{m_\bullet, M}$ and  $\sO_{\Fl(m_\bullet, M)}^M\to E^q_{m_\bullet, M}$ define morphisms
\[
\pi_{m_t, M}:\Fl(m_\bullet, M)\to \Gr(m_t, M),\ \pi_{m_q, M}:\Fl(m_\bullet, M)\to \Gr(m_q, M).
\]

Let $\sO_{\Fl}(1)$ denote the invertible sheaf $\det E^t_{m_\bullet, M}\otimes\det E^q_{m_\bullet, M}$. $\sO_{\Fl}(1)$ is very ample, and for $r\ge1$, we have the embedding 
\[
i_{m_\bullet, M, r}:\Fl(m_\bullet, M)\hookrightarrow\P^{N(m_\bullet,M,r)},
\]
with $i_{m_\bullet, M, r}^*(\sO_{\P^{N(m_\bullet,M,r)}}(1))\cong \sO_{\Fl}(r)$, 
defined by the composition
\begin{multline*}
\Fl(m_\bullet, M)\xrightarrow{(\pi_{m_t, M}, \pi_{m_q, M})}\Gr(m_t, M)\times\Gr(m_q, M)\\\xrightarrow{(\plu_{m_t, M,r}, \plu_{m_q,M,r})} \P^{N(m_t, M,r)}\times 
\P^{N(m_q, M,r)}\\
\xrightarrow{\ver_{N(m_t, M,r), N(m_q, M,r)}} \P^{N(m_\bullet, M,r)},
\end{multline*}
where $N(m_\bullet, M,r)=(N(m_t, M,r)+1)(N(m_q, M,r)+1)-1$.

Choose an 
$r:=r(m_t, m_q, M)>0$ so that, for $?=s,t,q$, we have 
\begin{equation}\label{eqn:TwistConditions}
\vbox{
\begin{enumerate}
\item[(a)] $H^1(\Fl(m_\bullet, M),  E^?_{m_\bullet, M}\otimes \sO_{\Fl}(r))=0$,
\item[(b)] $E^?_{m_\bullet, M}\otimes \sO_{\Fl}(r)$ is generated by global sections.
\end{enumerate}
}
\end{equation}

Consider the twist
 \[
 E_{m_\bullet, M}(r):\ 0\to E^s_{m_\bullet, M}\otimes\sO_{\Fl}(r)\to E^t_{m_\bullet, M}\otimes\sO_{\Fl}(r)\to E^q_{m_\bullet, M}\otimes\sO_{\Fl}(r)\to 0
 \]
 of the tautological sequence \eqref{eqn:TautExSeq}. Each term is globally generated, so after choosing isomorphisms
 \begin{equation}\label{eqn:ChosenSurjections}
 H^0(\Fl(m_\bullet, M), E^?_{m_\bullet, M}\otimes\sO_{\Fl}(r))\cong k^{M_?},\ ?=s,t,q,\ M_t=M_s+M_q,
 \end{equation}
 we have the morphisms
 \[
 g_?:\Fl(m_\bullet, M)\to \Gr(m_?, M_?),\ ?=s,t,q
 \]
 with
 \[
 g_?^*E_{m_?, M_?}\cong E^?_{m_\bullet, M}\otimes\sO_{\Fl}(r),\ ?=s,t,q.
 \]

 Let $N^*=N^*(N, m_\bullet, M,r):=(N+1)(N(m_\bullet, M,r)+1)-1$ and for $?=s,t,q$, let
 \[
 \Lambda_?:= \Lambda_?(N, m_\bullet, M,r):\P^N\times \Fl(m_\bullet, M)\to \P^{N^*}\times \Gr(m_?, M_?)
 \]
  be the composition
 \begin{multline*}
\P^N\times \Fl(m_\bullet, M) \xrightarrow{\id_{\P^N}\times (i_{m_\bullet, M,r}, g_?)}\P^N\times 
\P^{N(m_\bullet, M,r)}\times \Gr(m_?, M_?)\\\xrightarrow{\ver_{N, N(m_\bullet,M,r)}\times\id_\Gr}
\P^{N^*}\times \Gr(m_?, M_?).
\end{multline*}
We define
\[
\Lambda_{s+q}:\P^N\times \Fl(m_\bullet, M) \to \P^{N^*}\times \Gr(m_t, M_t)
\]
as the composition
\begin{multline*}
\P^N\times \Fl(m_\bullet, M)\xrightarrow{(\Lambda_s, \Lambda_q)}
(\P^{N^*}\times \Gr(m_s, M_s))\times_{\P^{N^*}} (\P^{N^*}\times \Gr(m_q, M_q))\\=
\P^{N^*}\times \Gr(m_s, M_s)\times \Gr(m_q, M_q)\xrightarrow{\id_{\P^{N^*}}\times\oplus}
\P^{N^*}\times \Gr(m_t, M_t).
\end{multline*}
 
 \begin{lemma}\label{lem:ExactSeqPresentation} Let $X$ be a quasi-projective $k$-scheme, let
 \[
E_\bullet:= 0\to E_s\to E_t\to E_q\to0
 \]
 be an exact sequence of locally free sheaves on $X$, with $E_s$ of rank $m_s$, $E_t$ of rank $m_t$ and $E_q$ of rank $m_q$. Let $L$ be a globally generated invertible sheaf on $X$ such that $E_t\otimes L$ is globally generated, and choose $M$ generating sections $s_1,\ldots, s_M$ of $E_t\otimes L$, giving the 
 corresponding morphism
 \[
 \phi:X\to \Fl(m_t,m_q, M)
 \]
 with
 \[
 \phi^*(E_{m_\bullet, M})\cong E_\bullet\otimes L.
 \]
 Choosing $N+1$ generating global sections $t_0,\ldots, t_N$ for $L$ gives the morphism
 \[
 \iota:X\to \P^N 
 \]
 and the morphism
 \[
 \psi:=(\iota, \phi):X\to \P^N\times\Fl(m_\bullet,M)
 \]
 with
 \[
 \psi^*(E_{m_\bullet, M}(-1))\cong E_\bullet.
 \]
 
 Let  $V$ be the $k$-vector space $k^M$ with basis $s_1,\ldots, s_M$, and choose $r\gg0$ so that \eqref{eqn:TwistConditions} for $(m_\bullet, M)$  holds.
 \begin{enumerate}
 \item Let $L_\phi=L\otimes \phi^*\sO_{\Fl}(r)$. Then  $E_?\otimes L_\phi$ is globally generated for $?=s,t,q$, and 
 \begin{equation}\label{eqn:PullbackIso}
  \phi^*(E_{m_\bullet, M}\otimes\sO_{\Fl}(r))\cong E_\bullet\otimes L_\phi.
\end{equation}
 \item Mapping $H^0(\Fl(m_\bullet, M), E^?_{m_\bullet, M}\otimes\sO_{\Fl}(r))$ to $H^0(X, E_?\otimes L_\phi)$ by the isomorphism \eqref{eqn:PullbackIso},  the image
 \[
 \phi^*(H^0(\Fl(m_\bullet, M), E^?_{m_\bullet, M}\otimes\sO_{\Fl}(r)))\subset H^0(X, E_?\otimes L_\phi)
 \]
 generates $E_?\otimes L_\phi$.
 \item Our choice of isomorphism $H^0(\Fl(m_\bullet, M), E^?_{m_\bullet, M}\otimes\sO_{\Fl}(r))\cong k^{M_?}$ gives rise to morphisms for $?=s,t,q$,
 \[
 \gamma_?:X\to \Gr(m_?, M_?)
 \]
 with $\gamma_?^*E_{m_?, M_?}\cong E_?\otimes L_\phi$ and with
 \[
 \gamma_?=g_?\circ  \phi.
 \]
 \item Let $W=\oplus_{i=0}^Nk\cdot t_i$, and let $W_\Fl:=H^0(\Fl(m_\bullet, M), \sO_{\Fl}(r))$. Map $W\otimes W_\Fl$  to $H^0(X, L_\phi)$ via the definition of $L_\phi$ as $L\otimes \phi^*\sO_{\Fl}(r)$ together with the evident map $W\to H^0(X, L)$. Then the image of 
 $W\otimes W_\Fl$ in $H^0(X, L_\phi)$ generates $L_\phi$, giving the morphism
 \[
 \ver_{W, W_\Fl}:X\to \P^{N^*},\ N^*=(N+1)\cdot\dim_kW_\Fl,
 \]
and, after a suitable choice of basis for $W$, $W_\Fl$  and $W\otimes W_\Fl$, we have
\[
(\ver_{W,W_\Fl}, \gamma_?)=\Lambda_?\circ \psi,\ ?=s,t,q.
\]
\item We have isomorphisms
\[
(\ver_{W,W_\Fl}, \gamma_?)^*(E_{m_?, M_?}(-1))\cong E_?,\ ?=s,t,q.
\]
 \end{enumerate}
 \end{lemma} 
 
 \begin{proof} For (1),  we have  
 \[
 \phi^*(E_{m_\bullet,M})\cong E_\bullet\otimes L
 \]
 by construction, and thus $\phi^*(E_{m_\bullet,M}\otimes\sO_\Fl(r))\cong E_\bullet\otimes L_\phi$. Since $E^?_{m_\bullet,M}\otimes\sO_\Fl(r)$ is globally generated for $?=s,t,q$, so are the $E_?\otimes L_\phi$.
 
 For (2), this follows from (1), since a space of global generators for $E^?_{m_\bullet,M}\otimes\sO_\Fl(r)$ pulls back via $\phi$ to a space of global generators for $\phi^*(E^?_{m_\bullet,M}\otimes\sO_\Fl(r))$.
 
 For (3), the map $g_?$ classifies the surjection
 \[
\pi: H^0(\Fl, E^?_{m_\bullet, M}\otimes\sO_\Fl(r))\otimes_k\sO_\Fl\to E^?_{m_\bullet, M}\otimes\sO_\Fl(r)
 \]
 after a choice of isomorphism $H^0(\Fl, E^?_{m_\bullet, M}\otimes\sO_\Fl(r))\cong k^{M_?}$ and thus $g_?\circ \phi$ classifies the surjection
 \begin{equation}\label{eqn:PiSurjection}
k^{M_?}\otimes_k\sO_X\cong  H^0(\Fl, E^?_{m_\bullet, M}\otimes\sO_\Fl(r))\otimes_k\sO_X\to \phi^*(E^?_{m_\bullet, M}\otimes\sO_\Fl(r))\cong E_?\otimes L_\phi
 \end{equation}
 induced by $\pi$ and the isomorphism from (1). As \eqref{eqn:PiSurjection} is the surjection used to define $\gamma_?$, we have
 \[
 \gamma_?=g_?\circ\phi,
 \]
 and
 \[
 \gamma_?^*(E^?_{m_\bullet, M})\cong  E_?\otimes L_\phi,\ ?=s,t,q.
 \]
 
 For (4), we have
 \[
 \Lambda_?\circ \psi=(\ver_{N, N(m_\bullet, M,r)}\times\id_\Gr)\circ(\id_{\P^N}\times (i_{m_\bullet, M,r}, g_?))\circ(\iota,\phi).
 \]
 so
 \[
 p_2\circ(\Lambda_?\circ \psi)=g_?\circ \phi=\gamma_?.
 \]
 Moreover, the composition $i_{m_\bullet, M,r}\circ\phi:X\to \P^{N(m_\bullet, M)}$ classifies the surjection
\[
k^{N(m_\bullet, M,r)+1}\otimes\sO_X\cong H^0(\Fl, \sO_\Fl(r))\otimes\sO_X\to \phi^*\sO_\Fl(r),
\]
induced by the canonical surjection $H^0(\Fl, \sO_\Fl(r))\otimes\sO_\Fl\to \sO_\Fl(r)$, after pulling back by $\phi$. Thus, the composition
\[
X\xrightarrow{(\iota, \phi)}\P^N\times\Fl(m_\bullet, M)\xrightarrow{\id_{\P^N}\times i_{m_\bullet, M,r}}
\P^N\times \P^{N(m_\bullet, M,r)}\xrightarrow{\ver_{N, N(m_\bullet, M,r)}}\P^{N^*}
\]
is the map classifying the surjection
\begin{multline*}
k^{N^*+1}\otimes\sO_X\cong k^{N+1}\otimes k^{N(m_\bullet, M,r)+1}\otimes\sO_X\\
\cong W\otimes W_\Fl\otimes\sO_X\to L\otimes\phi^*\sO_\Fl(r)=L_\phi
\end{multline*}
which, up to suitable choice of bases, is exactly the map $\ver_{W, W_\Fl}$. Thus
\[
p_1\circ (\Lambda_?\circ \psi)=\ver_{W, W_\Fl}
\]
so
\[
\Lambda_?\circ \psi=(\ver_{W, W_\Fl}, \gamma_?)
\]
as claimed. 

For point (5), 
\begin{align*}
(\ver_{W,W_\Fl}, \gamma_?)^*(E_{m_?, M_?}(-1))&= \gamma_?^*(E_{m_?, M_?})\otimes\ver_{W,W_\Fl}^*(\sO_{\P^{N^*}}(-1))\\
&=E_?\otimes L_\phi\otimes L_\phi^{-1}=E_?.
\end{align*}
 \end{proof}
 
\begin{definition}\label{def:TotalAndSum} Given integers $N, N_1, N_2$, $m_\bullet:=(m_t, m_q)$, $M$, $m_1, M_1$, choose $r$ satisfying \eqref{eqn:TwistConditions} for $(m_\bullet, M)$, and let  $M_t^*:=M_t\cdot(N(m_\bullet, M,r)+1)$,  $N^*_i=(N(m_\bullet, M,r)+1)(N_i+1)-1$,  $M_i^*:=M_1\cdot (N_i+1)$, $i=1,2$; see \eqref{eqn:ChosenSurjections} for the definition of $M_t$.\\[5pt]
1. Define
\[
\oplus(s+q,1):\P^N\times \Fl(m_\bullet, M)\times(\P^{N_1}\times\Gr(m_1, M_1)) \to \P^{N^*_1}\times\Gr(m_t+m_1, M_t^*+M_1^*)
\]
be the composition
\begin{multline*}
(\P^N\times \Fl(m_\bullet, M))\times(\P^{N_1}\times\Gr(m_1, M_1))\xrightarrow{\Lambda_{s+q}\times\id}
\P^{N^*}\times  \Gr(m_t, M_t)\times \P^{N_1}\times \Gr(m_1, M_1)\\
\xrightarrow{\tilde{\oplus}}
\P^{N_1^*}\times \Gr(m_t+m_1, M_t^*+M_1^*)
\end{multline*}
2. Define
\[
\oplus(t,2):\P^N\times \Fl(m_\bullet, M)\times(\P^{N_2}\times\Gr(m_2, M_2)) \to \P^{N^*_2}\times\Gr(m_t+m_2, M_t^*+M_2^*)
\]
as the composition
\begin{multline*}
\P^N\times \Fl(m_\bullet, M)\times(\P^{N_2}\times\Gr(m_2, M_2))\xrightarrow{\Lambda_t\times\id}
\P^{N^*}\times  \Gr(m_t, M_t)\times \P^{N_2}\times \Gr(m_2, M_2)\\
\xrightarrow{\tilde{\oplus}}
\P^{N_2^*}\times \Gr(m_t+m_2, M_t^*+M_2^*).
\end{multline*}
\end{definition}

\begin{remark} As the reader will see in the proof of Theorem~\ref{thm:K0Presentation}, the choice of $r$ in the definition of the maps
$\oplus(s+q,1)$ and  $\oplus(t,2)$ does not matter in our presentation of $K_0(X)$, and so we have not included this choice in the notation for these maps, even though this choice does enter into the definition of the numbers $M_?, M_?^*, M_i^*, N_i^*$, $?=s,t,q$, $i=1,2$.
\end{remark} 

Let $\widetilde{\Hom}(X,\P^N\times \Fl(m_\bullet, M))$ be the quotient of ${\Hom}(X,\P^N\times \Fl(m_\bullet, M))$ by the action of $\GL_{N+1}(k)\times\GL_M(k)$, and define $\widetilde{\Hom}(X,\P^N\times\Gr(m, M))$ similarly.

One can easily check that the maps $\oplus(s+q,1)$, $\oplus(t,2)$ define a map
\begin{multline}\label{multline:Relation}
\widetilde{\Hom}(X,\P^N\times \Fl(m_\bullet, M))\times \widetilde{\Hom}(X,\P^{N_1}\times\Gr(m_1, M_1))\times \widetilde{\Hom}(X,\P^{N_2}\times\Gr(m_2, M_2))\\\xrightarrow{(\oplus(s+q,1)_*\circ p_{12}, \oplus(t,2)_*\circ p_{13})_{N,m_\bullet, M}}\\ [\amalg_m \colim_{\N_+\times\N_+}\overline{\Hom}(X,\P^{?_1}\times\Gr(m, ?_2)]\times [\amalg_{m'} \colim_{\N_+\times\N_+}\overline{\Hom}(X,\P^{?_1}\times\Gr(m', ?_2)]
\end{multline}
Letting 
\begin{align*}
\coeq_I(X,m):=
\coeq[
\colim_{\N_+\times\N_+\times\N_+}&\widetilde{\Hom}(X, \P^{?_1}\times\P^{?_2}\times\Gr(m, ?_3))\\&\xymatrix{\ar@<5pt>[r]^{p_{13}(X)}\ar@<-5pt>[r]_{\lambda(X,m)}&}
\colim_{\N_+\times\N_+}\overline{\Hom}(X, \P^{?_1}\times\Gr(m, ?_2))]
\end{align*}
we have the induced maps
\begin{multline}\label{multline; coeqII}\ \\
\amalg_{N,m_\bullet, M}\widetilde{\Hom}(X, \P^N\times \Fl(m_\bullet, M))\times [\amalg_{N_1, m_1,M_1}\widetilde{\Hom}(X,\P^{N_1}\times\Gr(m_1, M_1))]\\\times[\amalg_{N_2, m_2,M_2}\widetilde{\Hom}(X,\P^{N_2}\times\Gr(m_2, M_2))]\\
\xymatrix{\ar@<5pt>[r]^{\prod_{N,m_\bullet, M}(\oplus(s+q,1)_*\circ p_{12}, \oplus(t,2)_*\circ p_{13})_{N,m_\bullet, M}}\ar@<-5pt>[r]_{\prod_{N,m_\bullet, M}(\pi_1\times\pi_2)\circ p_{23}}&}
\amalg_{m, m'}\coeq_I(X,m)\times \coeq_I(X, m')
\end{multline}
where $\pi_1=\amalg_m\pi_{1,m}$, with
\[
\pi_{1,m}:\amalg_{N_1, M_1}\widetilde{\Hom}(X,\P^{N_1}\times\Gr(m, M_1))\to
\coeq_I(X,m)
\]
is the evident map to the quotient, and $\pi_2$ is defined similarly.

\begin{theorem} \label{thm:K0Presentation} Let $X$ be a connected quasi-projective scheme over a field $k$. Let 
\[
\Pi_X:\amalg_{m,m'}\coeq_I(X,m)\times\coeq_I(X,m')\to\coeq_{II}(X)
\]
 denote the co-equalizer in $\Sets_*$ of the diagram \eqref{multline; coeqII}.  Recall the bijections $\Psi(X, m):\coeq_I(X, m)\xrightarrow{\sim} \Vect_m(X)$, $m\ge0$, from Proposition~\ref{prop:VectmPresentation}. Then the diagram
\[
\xymatrix{
\amalg_{m,m'}\coeq_I(X,m)\times\coeq_I(X,m')\ar[r]^-{\Pi_X}\ar[d]^{\amalg_{m,m'}\Psi(X,m)\times\Psi(X,m')}&\coeq_{II}(X)\\
\amalg_{m,m'}\Vect_m(X)\times\Vect_{m'}(X)\ar[r]& K_0(X)
}
\]
induces a bijection of pointed sets $\coeq_{II}(X)\to K_0(X)$, natural in $X$.
\end{theorem}

\begin{proof} We need to see that the equivalence relation $\sim^*_{K_0}$ on $\amalg_{m,m'}\Vect_m(X)\times\Vect_{m'}(X)$ induced by   the relation
\begin{multline*}
(\Psi(X,m_1+m_t)\circ \oplus(s+q,1)_*(\psi, f_1),\Psi(X,m_2+m_t)\circ \oplus(t,2)_*(\psi, f_2))\\\sim^* (\Psi(X,m_1)\circ f_1, \Psi(X,m_2)\circ f_2)
\end{multline*}
for $\psi:X\to \P^N\times \Fl(m_\bullet, M))$, $f_i:X\to \P^{N_i}\times\Gr(m_i, M_i)$, agrees with the equivalence relation $\sim_{K_0}$. 

It follows from Lemma~\ref{lem:ExactSeqPresentation} that each short exact sequence 
\begin{equation}\label{eqn:TargetSeq}
0\to E_s\to E_t\to E_q\to 0
\end{equation}
of locally free coherent sheaves on $X$ arises (up to isomorphism) as $\psi^*(E_{m_\bullet, M}(-1))$ for a suitable $\psi:=(\iota, \phi)$, with $\iota:X\to \P^N$ corresponding to a choice of $N+1$ global generators for an invertible sheaf $L$,  $\phi:X\to \Fl(m_\bullet, M)$ corresponding to the sequence of surjective maps $\sO_X^M\to E_t\otimes L\to E_q\otimes L$ constructed from a choice of $M$ global generators for $E_t\otimes L$. Conversely, each $\psi:X\to \P^N\times \Fl(m_\bullet, M)$ gives a corresponding exact sequence \eqref{eqn:TargetSeq} by taking $\psi^*(E_{m_\bullet, M}(-1))$, and these two operations are inverse to one another. 

Moreover,  given $\psi:X\to \P^N\times\Fl(m_\bullet, M)$ as above,   locally free sheaves $E_i$ of ranks $m_i$, and globally generated invertible sheaves $L_i$ such that $E_i\otimes L_i$ is globally generated, $i=1,2$, we have the corresponding morphisms
\[
f_i:X\to \P^{N_i}\times\Gr(m_i, M_i),\ i=1,2,
\]
by choosing a set of $M_i$ global generators for $E_i\otimes L_i$ and $N_i+1$ global generators for $L_i$. Applying Lemma~\ref{lem:ExactSeqPresentation} again, we  have
\[
\Psi(X, m_1+m_t)(\oplus(s+q,1)_*(\psi, f_1))=E_s\oplus E_q\oplus E_1,
\]
\[
\Psi(X, m_2+m_t)(\oplus(t,2)_*(\psi, f_2))= E_t\oplus E_2,
\]
and
\[
(\Psi(X,m_1)(f_1), \Psi(X, m_2)(f_2))= (E_1,E_2).
\]
Thus, the relation $\sim^*$ is given by
\[
(E_1, E_2)\sim^*(E_s\oplus E_q\oplus E_1, E_t\oplus E_2)
\]
for each $(E_1, E_2)\in \Vect_{m_1}(X)\times\Vect_{m_2}(X)$ and each short exact sequence $0\to E_s\to E_t\to E_q\to 0$ in $\ses_m(X)$. It is then easy to see that  $\sim^*_{K_0}=\sim_{K_0}$.
\end{proof}

\section{The top degree Chern class map} Throughout this section we fix an infinite field $k$, $(X, X^*)$ will be an admissible pair over $k$,  and we set $d:=\dim_kX$.  For simplicity, we will assume that $X$ is connected.

Using our  presentation of $K_0(X)$ from the previous section, we construct a map of pointed sets
\[
c_d:K_0(X)\to \CH^d(X, X^*)
\]
induced by the classical Chern class maps $c_d:K_0(Y)\to \CH^d(Y)$ for $Y$ smooth and quasi-projective over $k$. We will only use the $Y$ of the form a product of homogeneous spaces $\Gr(m, M)$ and $\Fl(m_\bullet, M')$, considered as homogeneous spaces for a product of general linear groups,  for which we have the well-defined pullback maps
\[
f^*:\CH^d(Y)\to \CH^d(X, X^*)
\]
for each morphism $f:X\to Y$, using Theorem~\ref{thm:HomogFunct}.

\begin{definition}
Given morphisms $f_i:X\to \P^{N_i}\times\Gr(m_i, M_i)$, $i=1,2$, we have the map to the product
\[
(f_1, f_2):X\to [\P^{N_1}\times\Gr(m_1, M_1)]\times [\P^{N_2}\times\Gr(m_2, M_2)].
\]
Let 
\[
p_1: [\P^{N_1}\times\Gr(m_1, M_1)]\times [\P^{N_2}\times\Gr(m_2, M_2)]\to
 \P^{N_1}\times\Gr(m_1, M_1),
\] 
and
\[
p_2: [\P^{N_1}\times\Gr(m_1, M_1)]\times [\P^{N_2}\times\Gr(m_2, M_2)]\to
 \P^{N_2}\times\Gr(m_2, M_2)
\] 
be the projections.

We have the element 
\[
[p_1^*(E_{m_1, M_1}(-1))]-[p_2^*(E_{m_2, M_2}(-1))]\in K_0(\P^{N_1}\times\Gr(m_1, M_1)\times \P^{N_2}\times\Gr(m_2, M_2))
\]
and its Chern class
\[
c_d([p_1^*(E_{m_1, M_1}(-1))]-[p_2^*(E_{m_2, M_2}(-1))])\in \CH^d(\P^{N_1}\times\Gr(m_1, M_1)\times \P^{N_2}\times\Gr(m_2, M_2)).
\]
Define $c_d(f_1-f_2)\in \CH^d(X, X^*)$  by
\[
c_d(f_1-f_2):=(f_1, f_2)^*(c_d([p_1^*(E_{m_1, M_1}(-1))]-[p_2^*(E_{m_2, M_2}(-1))])).
\]
\end{definition}

\begin{lemma}\label{lem:Compatibility} Let  $f_i:X\to \P^{N_i}\times\Gr(m_i, M_i)$ be morphisms, $i=1,2$.\\[5pt]
1. Given  elements $g_i\in \GL_{N_i+1}(k)\times\GL_{M_i}(k)$, $i=1,2$, we have
\[
c_d(f_1-f_2)=c_d(g_1\cdot f_1-g_2\cdot f_2)\text{ in }\CH^d(X, X^*).
\]
2. Given $M_i'\ge M_i$, $N_i'\ge N_i$, $i=1,2$, we have the closed immersions $i_{m_i, M_i', M_i}:\Gr(m_i, M_i)\to \Gr(m_i, M_i')$, $i_{N_i', N_i}:\P^{N_i}\to \P^{N_i'}$, giving the morphisms
\[
f_i':X\to \P^{N_i'}\times \Gr(m_i, M_i'),\ f_i':=(i_{N_i', N_i}\times i_{m, M_i', M_i})\circ f_i,\ i=1,2.
\]
Then
\[
c_d(f_1'-f_2')=c_d(f_1-f_2)\text{ in }\CH^d(X, X^*).
\]
3. For $i=1,2$, let $E_i$ be a locally free sheaf on $X$ of rank $m_i$, let $L_i$ be an invertible sheaf on $X$ and let $\varpi_i:\sO_X^{N_i+1}\to L_i$, $\pi_i:\sO_X^{M_i}\to E_i\otimes L_i$,  be surjections, such that $p_1\circ f_i$ classifies $\varpi_i$ and $p_2\circ f_i$ classifies $\pi_i$. Suppose we have another pair of surjections $\varpi_i':\sO_X^{N_i'+1}\to L_i$, $\pi_i':\sO_X^{M_i'}\to E_i\otimes L_i$, giving a second pair of maps
\[
f_i':X\to \P^{N_i'}\times\Gr(m_i, M_i'),\ i=1,2.
\]
Then
\[
c_d(f_1'-f_2')=c_d(f_1-f_2)\text{ in }\CH^d(X, X^*).
\]
4. Take integers $N_i, N_i', M_i>0$, $i=1, 2$ and an integer $m>0$. Consider the maps
\eqref{Veronese2}
\[
\lambda_i:=\lambda_{N_i, N_i', m,M_i}:\P^{N_i}\times\P^{N_i'}\times\Gr(m, M_i)\to \P^{N_i''}\times \Gr(m, M_i''),
\]
with $M_i''=(N_i'+1)\cdot M_i$, $N_i''=(N_i+1)\cdot(N_i'+1)-1$, $i=1,2$. We also have the projections
\[
p_{13,i}:\P^{N_i}\times\P^{N_i'}\times\Gr(m, M_i)\to \P^{N_i}\times\Gr(m, M_i),\ i=1,2.
\]
Let $g_i:X\to \P^{N_i}\times\P^{N_i'}\times\Gr(m, M_i)$ be morphisms, $i=1,2$, and let
\[
f_i=p_{13,i}\circ g_i:X\to  \P^{N_i}\times\Gr(m, M_i),\ f_i'=\lambda_i\circ g_i:X\to \P^{N_i''}\times \Gr(m, M_i''),\ i=1,2.
\]
Then
\[
c_d(f_1-f_2)=c_d(f_1'-f_2')\text{ in }\CH^d(X, X^*).
\]
5. Take non-negative integers $N, N_1, N_2$, $m_\bullet:=(m_t,m_q)$, $M$, $m_1$, $M_1$, $m_2$, $M_2$, with $m_t\ge m_q$, and choose an integer $r\gg0$ satisfying \eqref{eqn:TwistConditions} for $(m_\bullet, M)$.    Let $N(m_\bullet, M,r)=(N(m_t, M,r)+1)(N(m_q, M,r)+1)-1$. Recall from \eqref{eqn:ChosenSurjections} the integer $M_t$, and from Definition~\ref{def:TotalAndSum} the integers $M_t^*:=M_t\cdot(N(m_\bullet, M,r)+1)$,  $N^*_i=(N(m_\bullet, M,r)+1)(N_i+1)-1$,  $M_i^*:=M_i\cdot (N_i+1)$, $i=1,2$, and the morphisms
\[
\oplus(s+q,1):\P^N\times \Fl(m_\bullet, M)\times\P^{N_1}\times\Gr(m_1, M_1) \to \P^{N^*_1}\times\Gr(m_t+m_1, M_t^*+M_1^*)
\]
and
\[
\oplus(t,2):\P^N\times \Fl(m_\bullet, M)\times\P^{N_2}\times\Gr(m_2, M_2) \to \P^{N^*_2}\times\Gr(m_t+m_2, M_t^*+M_2^*).
\]
Take morphisms
\[
f_i:X\to \P^{N_i}\times\Gr(m_i, M_i), i=1,2,
\]
and
\[
\psi:X\to \P^N\times \Fl(m_\bullet, M),
\]
and let
\[
f_1':=\oplus(s+q,1)\circ (\psi, f_1):X\to \P^{N^*_1}\times\Gr(m_t+m_1, M_t^*+M_1^*),
\]
\[
 f_2'=\oplus(t,2)\circ (\psi, f_2):X\to \P^{N^*_2}\times\Gr(m_t+m_2, M_t^*+M_2^*).
\]
Then 
\[
c_d(f_1-f_2)=c_d(f_1'-f_2')\text{ in }\CH^d(X, X^*).
\]
\end{lemma}

\begin{proof}
 The fact that $\GL_{N_1+1}(k)\times\GL_{M_1}(k)\times \GL_{N_2+1}(k)\times\GL_{M_2}(k)$ acts trivially on $\CH^*(\P^{N_1}\times\Gr(m_1, M_1)\times \P^{N_2}\times\Gr(m_2, M_2))$ yields  (1). 

For (2), the closed immersions $i_{m_i, M_i', M_i}$, $i_{N_i', N_i}$ are equivariant with respect to the evident homomorphisms of general linear groups, so (2) follows from Proposition~\ref{prop:HomogFunctoriality}, and the fact that for 
\[
\iota:= i_{N_1',N_1}\times i_{m_1, M_1', M_1}\times  i_{N_2',N_2}\times i_{m_2, M_2', M_2}
\]
we have
\begin{multline*}
\iota^*(c_d([p_1^*(E_{m_1, M_1'}(-1))]-[p_2^*(E_{m_2, M_2'}(-1))]))\\=
c_d([p_1^*(E_{m_1, M_1}(-1))]-[p_2^*(E_{m_2, M_2}(-1))])
\end{multline*}

For (3), suppose $\pi_i$ is given by the global sections $h_{i,1},\ldots, h_{i, M_i}$ of $E_i\otimes L_i$ and $\pi_i'$ is given by the global sections $h_{i,1}',\ldots, h_{i, M_i'}'$ of $E_i\otimes L_i$, $i=1,2$.  Let $\pi_i''$ be given by the global sections $h_{i,1},\ldots, h_{i, M_i}, h_{i,1}',\ldots, h_{i, M_i'}'$.
Similarly, if $\varpi_i$ is given by global sections $s_{i,0},\ldots, s_{i, N_i}$ of $L_i$ and $\varpi_i'$ is given by the global sections $s_{i,0}',\ldots, s_{i, N_i'}'$ of $L_i$, $i=1,2$, let $\varpi_i''$ be given by global sections $s_{i,0},\ldots, s_{i, N_i}, s_{i,0}',\ldots, s_{i, N_i'}'$. Let $M_i''=M_i+M_i'$, $N_i''=N_i+N_i'+1$, and let
\[
f_i'':X\to \P^{N_i''}\times \Gr(m, M_i'')
\]
be the map given by the surjections $\varpi_i''$ and $\pi_i''$. Clearly, it suffices to compare $c_d(f_1-f_2)$ and $c_d(f_1''-f_2'')$. 

We have the inclusions $i_{m, M_i'', M_i}:\Gr(m, M_i)\to \Gr(m, M_i'')$ and $i_{N_i'', N_i}:\P^{N_i}\to \P^{N_i''}$. Let $\iota_i=i_{N_i'', N_i}\times i_{m, M_i'', M_i}$. 
Using (2), it suffices to show 
\[
c_d(\iota_1\circ f_1-\iota_2\circ f_2)=c_d(f_1''-f_2'')
\]
For this, let $t_0, t_1$ be standard homogeneous coordinates on $\P^1$, with $0=[1:0]$ and $\infty=[0:1]$. We have the locally free sheaves $E_i(1)$ on $X\times\P^1$, with generating global sections
\[
(t_0+t_1)h_{i,1},\ldots, (t_0+t_1)h_{i, M_i}, t_1h_{i,1}',\ldots, t_1h_{i, M_i'}'
\]
for $E_i(1)\otimes L_i$ and similarly, we have the invertible sheaves $L_i(1)$ on $X\times\P^1$ with
generating global sections
\[
(t_0+t_1)s_{i,0},\ldots, (t_0+t_1)s_{i, N_i}, t_1s_{i,0}',\ldots, t_1s_{i, N_i'}'
\]
 Letting 
\[
F_i:X\times\P^1\to \P^{N_i}\times\Gr(m, M_i'')
\]
be the corresponding morphism, we have 
\[
F_i|_{X\times\{0\}}=\iota_i\circ f_i, \ F_i|_{X\times\{\infty\}}=f_i'',\ i=1,2,
\]
after choosing trivializations of $\sO_{\P^1}(1)$ at $0$ and $\infty$. By Proposition~\ref{prop:RatEquivPullback}, we have
\[
c_d(F_1|_{X\times\{0\}}-F_2|_{X\times\{0\}})=c_d(F_1|_{X\times\{\infty\}}-F_2|_{X\times\{\infty\}}),
\]
proving (3).

To prove (4), Lemma~\ref{lem:TensorIdentity} gives us the isomorphism of locally free sheaves on 
$\P^{N_i}\times\P^{N_i'}\times\Gr(m, M_i)$,
\[
p_{13,i}^*(E_{m, M_i}(-1))\cong \lambda_i^*(E_{m, M_i''}(-1)),\ i=1,2,
\]
so
\[
c_d([p_{13,1}^*(E_{m, M_1}(-1))]-[p_{13,2}^*(E_{m, M_2}(-1))])=
c_d([\lambda_1^*(E_{m, M_1''}(-1))]-[\lambda_2^*(E_{m, M_2''}(-1))])
\]
in $\CH^d(\P^{N_i}\times\P^{N_i'}\times\Gr(m, M_i))$. On the other hand, the morphism $p_{13,i}$ is equivariant with respect to the projection
\[
p_{13}: \GL_{N_i+1}\times \GL_{N_i'+1}\times \GL_{M_i}\to \GL_{N_i+1}\times \GL_{M_i}
\]
and $\lambda_i$ is equivariant with respect to 
\[
\rho_i=(\rho_{\otimes,N_i+1, N_i'+1}\circ p_{12},  \rho_{\otimes,N'_i+1, M_i}\circ p_{23}:\GL_{N_i+1}\times \GL_{N_i'+1}\times \GL_{M_i}\to \GL_{N_i''}\times \GL_{M_i''}
\]
where $\rho_{\otimes, N, M}:\GL_N\times \GL_M\to \GL_{NM}$ is the tensor product representation (in suitable coordinates). Thus, by Proposition~\ref{prop:HomogFunctoriality}, we have
\begin{align*}
c_d(f_1'-f_2')&=(f_1', f_2')^*(c_d([p_1^*(E_{m_1, M''_1}(-1))]-[p_2^*(E_{m_2, M''_2}(-1))]))\\
&=(g_1, g_2)^*(c_d([\lambda_1^*(E_{m, M_1''}(-1))]-[\lambda_2^*(E_{m, M_1''}(-1))]))\\
&=(g_1, g_2)^*(c_d([p_{13,1}^*(E_{m, M_1}(-1))]-[p_{13,2}^*(E_{m, M_2}(-1))]))\\
&=(f_1, f_2)^*(c_d([E_{m, M_1}(-1)]-[E_{m, M_2}(-1)]))\\
&=c_d(f_1-f_2),
\end{align*}
proving (4).

For (5), we have the universal exact sequence
\begin{equation}\label{eqn:UnivExactSeq}
0\to E_{m_\bullet, M}^s\to E_{m_\bullet, M}^t\to E_{m_\bullet, M}^q\to0
\end{equation}
on $\Fl(m_\bullet, M)$. It follows from Lemma~\ref{lem:TwistedDirectSum} and the definition of the maps $\oplus(s+q,1)$ and $\oplus(t,2)$ (Definition~\ref{def:TotalAndSum}) that
\begin{multline*}
\oplus(s+q,1)^*(E_{m_t+m_1, M_t^*+M_1^*}(-1))\\\cong p_{12}^*(E_{m_\bullet, M}^s(-1,0)\oplus E_{m_\bullet, M}^q(-1,0))\oplus p_{34}^*(E_{m_1, M_1}(-1))
\end{multline*}
and
\[
\oplus(t,2)^*(E_{m_t+m_2, M_t^*+M_2^*}(-1))\cong p_{12}^*(E_{m_\bullet, M}^t(-1,0))\oplus p_{34}^*(E_{m_2, M_2}(-1))
\]
Let
\[
Y=[\P^N\times \Fl(m_\bullet, M)]\times[\P^{N_1}\times\Gr(m_1, M_1)] \times[\P^{N_2}\times\Gr(m_2, M_2)], 
\]
with projections
\[
p_0:Y\to \P^N\times \Fl(m_\bullet, M),\ p_1:Y\to \P^{N_1}\times\Gr(m_1, M_1),\ p_2:Y\to 
\P^{N_2}\times\Gr(m_2, M_2),
\]
and corresponding projections $p_{ij}$, $0\le i<j\le 2$.
Using the exact sequence \eqref{eqn:UnivExactSeq}, we have the relation in $K_0(Y)$
\begin{multline*}
[p_{01}^*(\oplus(s+q,1)^*(E_{m_t+m_1, M_t^*+M_1^*}(-1)))]-
[p_{02}^*(\oplus(t,2)^*(E_{m_t+m_2, M_t^*+M_2^*}(-1)))]\\=
[p_{1}^*(E_{m_1, M_1}(-1))]-[p_{2}^*(E_{m_1, M_1}(-1))], 
\end{multline*}
giving the relation in $\CH^d(Y)$
\begin{multline*}
c_d([p_{01}^*(\oplus(s+q,1)^*(E_{m_t+m_1, M_t^*+M_1^*}(-1)))]-
[p_{02}^*(\oplus(t,2)^*(E_{m_t+m_2, M_t^*+M_2^*}(-1)))])\\=
c_d([p_{1}^*(E_{m_1, M_1}(-1))]-[p_{2}^*(E_{m_2, M_2}(-1))])
\end{multline*}

We   have the map $(\psi, f_1, f_2):X\to Y$ with
\[
\oplus(s+q,1)\circ p_{01}\circ (\psi, f_1, f_2)=f_1',\ \oplus(t,2)\circ p_{02}\circ (\psi, f_1, f_2)=f_2',
\]
and
\[
p_{1}\circ (\psi, f_1, f_2)=f_1, \ p_{2}\circ (\psi,f_1,f_2)=f_2.
\]
Applying the functoriality statement of  Proposition~\ref{prop:HomogFunctoriality}  gives
\begin{align*}
c_d(f_1-f_2)&= (\psi, f_1, f_2)^*(c_d([p_{1}^*(E_{m_1, M_1}(-1))]-[p_{2}^*(E_{m_2, M_2}(-1))])\\
&=(\psi, f_1, f_2)^*c_d\Big([p_{01}^*(\oplus(s+q,1)^*(E_{m_t+m_1, M_t^*+M_1^*}(-1)))]\\&\hskip 100pt-
[p_{02}^*(\oplus(t,2)^*(E_{m_t+m_2, M_t^*+M_2^*}(-1)))]\Big)\\
&=c_d(f_1'-f_2')
\end{align*}
\end{proof}

\begin{theorem}\label{thm:TopChernClass} Let $k$ be  infinite field, let $(X, X^*)$ be an admissible pair over $k$, and let $d=\dim_kX$.  Then there is a map of pointed sets
\[
c_d:K_0(X)\to \CH^d(X, X^*),
\]
uniquely determined by the identity
\begin{equation}\label{eqn:CharId}
c_d([f_1^*(E_{m_1, M_1}(-1))]-[f_2^*(E_{m_2, M_2}(-1))])=c_d(f_1-f_2)
\end{equation}
for morphisms $f_i:X\to \P^{N_i}\times\Gr(m_i, M_i)$, $i=1,2$.
\end{theorem}

\begin{proof}  Breaking up $X$ into its connected components and discarding those of dimension $<d$, we may assume that $X$ is connected. We have the map of sets
\begin{multline*}
\amalg_{N_1, m_1, M_1, N_2, m_2, M_2}\Hom(X,\P^{N_1}\times\Gr(m_1, M_1))\times \Hom(X,\P^{N_2}\times\Gr(m_2, M_2))\\
\xrightarrow{c_d^{(0)}}\CH^d(X, X^*)
\end{multline*}
sending $(f_1, f_2)$ to $(f_1, f_2)^*(c_d([p_{12}^*(E_{m_1, M_1}(-1))]-[p_{34}^*(E_{m_2, M_2}(-1))]))$.
\\[5pt]
 {\bf Step 1} The map $c_d^{(0)}$
descends to a well-defined map of sets
\begin{multline*}
\amalg_{N_1, m_1, M_1, N_2, m_2, M_2}\underline{\Hom}(X,\P^{N_1}\times\Gr(m_1, M_1))\times \underline\Hom(X,\P^{N_2}\times\Gr(m_2, M_2))\\
\xrightarrow{c_d^{(1)}} \CH^d(X, X^*)
\end{multline*}
This follows from Lemma~\ref{lem:Compatibility}(1)-(3).\\[5pt]
{\bf Step 2} The map $c_d^{(1)}$ descends to a map of pointed sets
\[
c_d^{(2)}:\amalg_{m_1, m_2}\Vect_{m_1}(X)\times\Vect_{m_2}(X)\to \CH^d(X, X^*)
\]
via the surjection of Proposition~\ref{prop:VectmPresentation}, where  $*=\Vect_{0}(X)\times\Vect_{0}(X)$ is the base-point on the source, and $0$ is the base-point on the target.  This follows from 
Proposition~\ref{prop:VectmPresentation} and Lemma~\ref{lem:Compatibility}(4).\\[5pt]
{\bf Step 3} The map $c_d^{(2)}$ descends to a well-defined map of pointed sets
\[
c_d:K_0(X)\to \CH^d(X, X^*)
\]
via the surjection of Theorem~\ref{thm:K0Presentation}. This follows from Theorem~\ref{thm:K0Presentation} and Lemma~\ref{lem:Compatibility}(5).

The fact that $c_d$  satisfies the identity \eqref{eqn:CharId} follows from the definition of $c_d^{(0)}$. The identity \eqref{eqn:CharId} determines $c_d$ uniquely since $K_0(X)$ is generated by formal differences of classes $[E_1]-[E_2]$, for $E_1, E_2$ locally free coherent sheaves, and for each $E_1, E_2$, there are morphisms $f_i:X\to \P^{N_i}\times \Gr(m_i, M_i)$, $i=1,2$ with $E_i\cong f_i^*(E_{m_i, M_i}(-1))$, $i=1,2$.
\end{proof}

\begin{remark}\label{rem:ChernClassRegular} Let $Y\in\Sch^\qp_k$ be a regular scheme, $n>0$ an integer. We have the Chern class map $c_n:K_0(Y)\to \CH^n(Y)$ defined using  Fulton's Chern class operators $c_n^F(-)\cap -$  \cite[\S 3.2, Example 3.2.7]{Fulton}, 
\[
c_n(x):=c_n^F(x)\cap [Y], 
\]
where $[Y]\in \CH_*(Y)$ is the fundamental class. In case $Y$ is smooth over $k$, this is the same as the usual Chern class map defined by Grothendieck;  see Remark~\ref{rem:FultonChernClass} for further details.
\end{remark}

We note the following naturalities of the map $c_d$.

\begin{proposition}\label{prop:ChernClassFunct} Let $(X, X^*)$ be an admissible pair over an infinite field $k$, and let $d=\dim_kX$.\\[5pt]
1. Let $(Y, Y^*)$ be a second admissible pair with $\dim_kY=d$ and let $f:Y\to X$ be a morphism.  Suppose that the conditions (i)-(iii) of Lemma~\ref{lem:EtalePullback} are satisfied, giving the pullback map $f^*:\CH^d(X, X^*)\to \CH^d(Y, Y^*)$. Then the diagram
\[
\xymatrix{
K_0(X)\ar[r]^-{f^*}\ar[d]^{c_d}&K_0(Y)\ar[d]^{c_d}\\
\CH^d(X,X^*)\ar[r]^-{f^*}&\CH^d(Y, Y^*)
}
\]
commutes.\\[5pt]
2. Let $f:Y\to X$ be a morphism, where  $Y\in\Sch_k^\qp$ is regular and of pure dimension $d$ over $k$. We have the Chern class map $c_d:K_0(Y)\to \CH^d(Y)$ of Remark~\ref{rem:ChernClassRegular}.
Suppose that $f$ is flat over $X\setminus X^*$, giving the pullback map $f^*:\CH^d(X,X^*)\to \CH^d(Y)$, following Remark~\ref{rem:OpenIm}(2).   Then the diagram
\[
\xymatrix{
K_0(X)\ar[r]^-{f^*}\ar[d]^{c_d}&K_0(Y)\ar[d]^{c_d}\\
\CH^d(X,X^*)\ar[r]^-{f^*}&\CH^d(Y)
}
\]
commutes.  
\\[5pt]
3. Let $k\hookrightarrow F$ be an extension of fields. Suppose that the base-extension $X_F$ is reduced and $X_F\setminus (X^*)_F$ is regular, giving the pullback map (Lemma~\ref{lem:base-change})
\[
p_X^*:\CH^d(X, X^*)\to \CH^d(X_F, X_F^*).
\]
Then the diagram 
\[
\xymatrix{
K_0(X)\ar[r]^-{p_X^*}\ar[d]^{c_d}&K_0(X_F)\ar[d]^{c_d}\\
\CH^d(X,X^*)\ar[r]^-{p_X^*}&\CH^d(X_F, X_F^*)
}
\]
commutes.
\end{proposition}

\begin{proof} For (1),  we reduce to showing that, given morphisms $f_i:X\to \P^{N_i}\times\Gr(m_i, M_i)$, $i=1,2$, we have
\[
f^*(c_d(f_1-f_2))=c_d((f_1f)-(f_2f)).
\]
We have
\begin{align*}
f^*(c_d(f_1-f_2))&=f^*(f_1, f_2)^*c_d([p_{1}^*E_{m_1, M_1}(-1)]- [p_{2}^*E_{m_2, M_2}(-1)])\\
&=(f_1f, f_2f)^*c_d([p_{1}^*E_{m_1, M_1}(-1)]- [p_{2}^*E_{m_2, M_2}(-1)])\\
&=c_d((f_1f)-(f_2f)),
\end{align*}
where we use Proposition~\ref{prop:PullbackFunct3} for the second identity.  

For (2), we claim that $c_d:K_0(Y)\to \CH^d(Y)$, defined using Fulton's Chern class operators,  agrees with the Chern class map defined via Theorem~\ref{thm:TopChernClass} for the admissible pair $(Y,\0)$. Accepting the claim, (2) is a special case of (1),  using Remark~\ref{rem:OpenIm} to note that the pullback map $f^*:\CH^d(X, X^*)\to \CH^d(Y)$ is defined as a special case of the pullback map constructed in Lemma~\ref{lem:EtalePullback}, with $Y^*=\0$.

To prove the claim, we need only check that for $(f_1, f_2):Y\to \P^{N_1}\times\Gr(m_1,M_1)\times \P^{N_2}\times\Gr(m_2,M_2)=:H$ a morphism, that 
\begin{multline*}
c_d^F((f_1, f_2)^*([p_1^*(E_{m_1, M_1}(-1))]-[p_2^*(E_{m_2, M_2}(-1)]))\cap [Y]\\=
(f_1, f_2)^*(c_d([p_1^*(E_{m_1, M_1}(-1))]-[p_2^*(E_{m_2, M_2}(-1))])).
\end{multline*}
This is \eqref{eqn:BivChernClass} with $\alpha=[H]$, together with \eqref{eqn:ChernClassLciPullback}.

The proof of (3) is similar to the proof of (1). Given morphisms $f_i:X\to \P^{N_i}_k\times\Gr_k(m_i, M_i)$, $i=1,2$, we have the induced morphisms $f_{iF}:X_F\to \P_F^{N_i}\times\Gr_F(m_i, M_i)$, $i=1,2$. If we have an algebraic cycle 
\[
Z\in \CH^d(\P^{N_1}_k\times\Gr_k(m_1, M_1)\times \P^{N_2}_k\times\Gr_k(m_2, M_2))
\]
representing $c_d([p_{1}^*E_{m_1, M_1}(-1)]- [p_{2}^*E_{m_2, M_2}(-1)]))$, then the base-change 
\[
Z_F\in 
\CH^d(\P^{N_1}_F\times\Gr_F(m_1, M_1)\times \P^{N_2}_F\times\Gr_F(m_2, M_2)
\]
represents $c_d([p_{1}^*E^F_{m_1, M_1}(-1)]- [p_{2}^*E^F_{m_2, M_2}(-1)]))$, where $E^F_{m_1, M_1}$, $E^F_{m_2, M_2}$ are the respective base-changes of $E^F_{m_1, M_1}$, $E^F_{m_2, M_2}$. If  the pullback $(f_1, f_2)^*(Z)$ is well-defined, so that $c_d(f_1-f_2)$ is represented by  $(f_{1}, f_{2})^*(Z)$, then the pull-back $(f_{1F}, f_{2F})^*(Z_F)$ is also well-defined, and  $c_d(f_{1F}-f_{2F})$ is therefore represented by $(f_{1F}, f_{2F})^*(Z_F)$. Since
$(f_{1F}, f_{2F})^*(Z_F)=p_X^*(f_1, f_2)^*(Z)$ as cycles, we thus have
\[
c_d(f_{1F}-f_{2F})=p_X^*(c_d(f_1-f_2))\in \CH^*(X_F, X_F^*),
\]
which suffices to prove (3).
\end{proof}

\section{Riemann-Roch theorem}
In the introduction, we briefly discussed the consequences of the Grothendieck-Riemann-Roch theorem for the relation of the Chow ring and $K_0$ for a smooth quasi-projective $k$-scheme. In this section, we prove a version of GRR for $\CH^d(X, X^*)$; here $k$ will be an infinite field.

For   $X$ a quasi-projective $k$-scheme of dimension $d=\dim_kX$, we have the subgroup $F^dK_0(X)$ of $K_0(X)$ generated by the classes  $[i_{x*}k(x)]$, where $k(x)$ is the residue field of a closed point $i_x:x\to X$ of $X$ supported in $X_\reg(d)$.  

\begin{proposition}[See \hbox{\cite[Proposition 2.1]{LW}}] \label{prop:CycleClass} Let $(X, X^*)$ be an admissible pair and let $d=\dim_kX$. 
Then the group homomorphism
\[
\cl^d:Z^d(X\setminus X^*)\to K_0(X)
\]
sending a closed point $x\in X\setminus X^*$ to the $K_0$-class of the residue field $k(x)$, viewed as a coherent $\sO_X$-module of finite homological dimension, descends to a well-defined group homomorphism
\[
\cl^d:\CH^d(X, X^*)\to K_0(X)
\]
with image $F^dK_0(X)$.
\end{proposition}

\begin{proof} Take $Y\subset X$ with $X\setminus Y$ regular. It is shown in \cite[Proposition 2.1]{LW} that the map $Z_0(X\setminus Y)\to K_0(X)$ sending $x$ to the $K_0$-class of the residue field $k(x)$ descends to a well-defined homomorphism $\cl_0:\CH_0(X,Y)\to K_0(X)$.  There is a tacit assumption that $X\setminus Y$ is dense in $X$, but this is not needed for the proof, so is also applicable to $\CH^d(X, X^*)$, via Proposition~\ref{prop:Comparison}(2).

It remains to show that $\cl^d$ has image $F^dK_0(X)$. For this,  it is clear that $\cl^d(\CH^d(X, X^*))$ is the subgroup of $K_0(X)$ generated by classes $[k(x)]$ for $x$ a closed point in $X\setminus X^*$, which is $F^dK_0(X)$ in case $X^*=X_\sing\cup X(\le d-1)$. However, by Lemma~\ref{lem:CHSurj}, the map
\[
\CH^d(X, X^*)\to \CH^d(X, X_\sing\cup X(\le d-1))
\]
induced by the inclusion $X_\sing\cup X(\le d-1))\subset X^*$ is surjective. Since  the diagram
\[
\xymatrix{
\CH^d(X, X^*)\ar[r]\ar[dr]_{\cl^d}& \CH^d(X, X_\sing\cup X(\le d-1))\ar[d]^{\cl^d}\\
&K_0(X)
}
\]
commutes, we have $\cl^d(\CH^d(X, X^*))=F^dK_0(X)$.
\end{proof}
 
From now on we simply write
\[
\cl^d:\CH^d(X, X^*)\to F^dK_0(X).
\]

We have the Chern class map $c_d:K_0(X)\to \CH^d(X, X^*)$ defined in the previous section.

\begin{theorem}[Riemann-Roch for 0-cycles]\label{thm:RR} Let $k$ be an infinite field and let $(X, X^*)$ be an admissible pair over $k$.  Then 
\begin{enumerate}
\item The restriction of $c_d$ to $F^dK_0(X)$ is a homomorphism of abelian groups
\[
c_d:F^dK_0(X)\to \CH^d(X, X^*).
\]
\item We have
\begin{enumerate}
\item[(a)] $c_d\circ \cl^d=(-1)^{d-1}(d-1)!\cdot\id_{\CH^d(X, X^*)}$,
\item[(b)] $\cl^d\circ c_d=(-1)^{d-1}(d-1)!\cdot\id_{F^dK_0(X)}$.
\end{enumerate}
\end{enumerate}
In particular, the homomorphisms $\cl^d:\CH^d(X, X^*)\to F^dK_0(X)$ and $c_d:F^dK_0(X)\to \CH^d(X, X^*)$ both have kernels that are killed by $(d-1)!$.
\end{theorem}

\begin{proof} Note first that (1) is a consequence of (2a). Indeed, by Proposition~\ref{prop:CycleClass}, the map
\[
\cl^d:\CH^d(X, X^*)\to F^dK_0(X)
\]
is a surjective group homomorphism, and if we know that 
\[
c_d\circ \cl^d=(-1)^{d-1}(d-1)!\cdot\id_{\CH^d(X, X^*)}, 
\]
it follows that $c_d\circ \cl^d$ is a group homomorphism, which implies that $c_d$ is itself also a group homomorphism.

Similarly, the identity (2a) implies the identity (2b) by   applying composition $-\circ \cl^d$ with the surjective homomorphism $\cl^d$.

It thus remains to prove (2a).

Let  $j:X\to \bar{X}$ be a projective closure of $X$ and let $\bar{X}^*\supset \bar{X}_\sing\cup \bar{X}(d-1)$ be a closed subset with $\dim_k\bar{X}^*<d$ and with $j^{-1}(\bar{X}^*)=X^*$; since the closure of a reduced scheme is reduced, $\bar{X}(d)^0$ is reduced. Then $(\bar{X}, \bar{X}^*)$ is an admissible pair with $j(X\setminus X^*)\subset \bar{X}\setminus \bar{X}^*$. We have a well-defined pushforward map
\[
j_*:Z^d(X\setminus X^*)\to Z^d(\bar{X}\setminus \bar{X}^*),
\]
even though $j$ is not proper, and we have $j^*:K_0(\bar{X})\to K_0(X)$, with
\[
j^*(\cl^d(j_*(z)))=\cl^d(z)\in K_0(X)
\]
for $z\in Z^d(X\setminus X^*)$. We also have the pullback map $j^*:\CH^d(\bar{X}, \bar{X}^*)\to \CH^d(X, X^*)$ by Remark~\ref{rem:OpenIm}, with  $j^*[j_*(z)]=[z]$. We can apply Proposition~\ref{prop:ChernClassFunct}(1),  which gives the identity 
\[
j^*(c_d([j_*(z)]))=c_d(j^*[j_*(z)])=c_d([z]).
\]
Thus,  we may assume from the start that $X$ is projective. 

Suppose we have $z=\sum_{i=1}^s r_ix_i\in Z_0(X\setminus X^*)$ with the $x_i$ distinct and each $r_i\neq0$. We proceed by induction on $s$, the case $s=0$ following from the fact that $\cl^d(0)=0$ and $c_d(0)=0$.

So take $s>0$ and assume the result for all $s'<s$ and all admissible pairs $(V, X^*)$ with $V\subset X$ open, $V\supset X^*$.  Let $U=X\setminus\{x_1\}$, with inclusion $j:U\to X$. By our induction hypothesis, we have
\[
c_d(\cl^d([\sum_{i=2}^sr_i\cdot x_i]))=(-1)^{d-1}(d-1)!\cdot [\sum_{i=2}^sr_i\cdot x_i]\in \CH^d(X, X^*).
\]

We have $j^*[x_1]=0$ so 
\[
j^*[\sum_{i=1}^sr_i\cdot [x_i]))=j^*[\sum_{i=2}^sr_i\cdot [x_i]))
\]
and since $j^*\circ c_d\circ \cl^d=c_d\circ \cl^d\circ j^*$, we have
\begin{align*}
j^*(c_d(\cl^d&(\sum_{i=1}^sr_i\cdot [x_i]))-(-1)^{d-1}(d-1)!\cdot \sum_{i=1}^sr_i\cdot [x_i])\\
&=(c_d\circ \cl^d)(j^*(\sum_{i=1}^sr_i\cdot [x_i]))-(-1)^{d-1}(d-1)!\cdot j^*(\sum_{i=1}^sr_i\cdot [x_i])\\
&=(c_d\circ \cl^d)(j^*(\sum_{i=2}^sr_i\cdot [x_i]))-(-1)^{d-1}(d-1)!\cdot j^*(\sum_{i=2}^sr_i\cdot [x_i])\\
&=j^*(c_d(\cl^d(\sum_{i=2}^sr_i\cdot [x_i]))-(-1)^{d-1}(d-1)!\cdot \sum_{i=2}^sr_i\cdot [x_i])\\
&=0 \in \CH^d(U, X^*).
\end{align*}
Via the localization sequence (Lemma~\ref{lem:Localization})
\[
\Z\cdot [x_1]=\CH_0(x_1)\xrightarrow{i_{x_1*}}\CH^d(X, X^*)\xrightarrow{j^*}\CH^d(U, X^*)\to 0
\]
we find that  there is an integer $m$ with
\begin{equation}\label{eqn:RRIdentity}
c_d(\cl^d(\sum_{i=1}^sr_i\cdot [x_i]))-(-1)^{d-1}(d-1)!\cdot \sum_{i=1}^sr_i\cdot [x_i]=m\cdot [x_1]\in \CH^d(X, X^*),
\end{equation}
and we need to show that $m=0$.

Suppose first that $k$ has characteristic zero. Then $X_\reg(d)=X_\sm(d)$ and we can use resolution of singularities to give us a projective birational  morphism
\[
p:\tilde{X}\to   X(d)
\]
with $\tilde{X}$  smooth over $k$ and $p$ an isomorphism over $X_\reg(d)\supset X\setminus X^*$. Applying Remark~\ref{rem:OpenIm}(2), we have the well-defined pullback map
\[
p^*:\CH^d(X, X^*)\to \CH^d(\tilde{X}).
\]
with $p^*([x_i])=[y_i]$ if $y_i\in \tilde{X}$ is the closed point $p^{-1}(x_i)$. By Proposition~\ref{prop:ChernClassFunct}(2) and Riemann-Roch on the smooth $k$-scheme $\tilde{X}$, we have
\begin{multline*}
p^*(c_d(\cl^d(\sum_{i=1}^sr_i\cdot [x_i]))-(-1)^{d-1}(d-1)!\cdot \sum_{i=1}^sr_i\cdot [x_i])\\=
c_d(\cl^d(\sum_{i=1}^sr_i\cdot [y_i]))-(-1)^{d-1}(d-1)!\cdot \sum_{i=1}^sr_i\cdot [y_i]=0
\end{multline*}
 Thus 
 \[
0=p^*(c_d(\cl^d(\sum_{i=1}^sr_i\cdot [x_i]))-(-1)^{d-1}(d-1)!\cdot \sum_{i=1}^sr_i\cdot [x_i])=p^*(m[x_1])=m[y_1].
 \]

Let $\pi:\tilde{X}\to \Spec k$ be the structure morphism. Then we have
\[
0=\pi_*(m[y_1])=m[k(y_1):k]\cdot [\Spec k]\in \CH_0(\Spec k)=\Z\cdot [\Spec k]
\]
so $m=0$, and the induction goes through.

In characteristic $p>0$, we use a modification of this argument.\\[5pt]
{\bf Step 1}. We reduce to the case of $k$ finitely generated over $\F_p$. For this, there is a  subfield $k_0$ of $k$, finitely generated over $\F_p$, over which $X, X^*$ and  $x_1,\ldots, x_s\in X\setminus X^*$ are all defined, that is, there is a quasi-projective $k_0$-scheme $X_0$, a closed subset $X^*_0$ of $X_0$ and  closed points $x_{i0}$ of $X_0\setminus X_0^*$ such that
\[
(X_0, X_0^*, x_{10},\ldots, x_{s0})\times_{k_0}k\cong (X, X^*, x_1,\ldots, x_s) 
\]
Let $p_X:X=X_0\times_{k_0}k\to X_0$ be the projection.
 
We claim that $\dim_{k_0}X_0=\dim_kX=d$ and $X_0\setminus X_0^*$ is an admissible pair over $k_0$. This first identity follows from the faithful flatness of $k_0\hookrightarrow k$, more generally, $X_0(i)\times_{k_0}k\cong X(i)$, and thus $X_0^*\supset X_0(\le d-1)$. Similarly, $X(d)^0=X_0(d)^0\times_{k_0}k$ and by faithful flatness again, $X_0(d)^0\times_{k_0}k$ being reduced implies $X_0(d)^0$ is reduced. Take $y_0\in X_0\setminus X^*_0$ and let $y\in X\setminus X^*$ be a point lying over $y_0$,  giving the flat local extension $\sO_{X_0,y_0}\to \sO_{X, y}$. By \cite[Theorem 23.7]{MatsumuraCRT}, the fact that $\sO_{X,y}$ is regular of dimension $d$ implies $\sO_{X_0,y_0}$ is regular of dimension $d$, hence $X_0\setminus X_0^*\subset X_{0\reg}$. Since $X\setminus X^*$ is dense in $X(d)$, it follows that $X_0\setminus X_0^*$ is dense in  $X_{0\reg}$. Thus $(X_0, X_0^*)$ is an admissible pair over $k_0$, as claimed. In particular, we have well-defined pullback maps
\[
p_X^*:\CH^d(X_0, X_0^*)\to \CH^d(X, X^*),\ p_X^*:F^dK_0(X_0)\to F^dK_0(X).
\]
Since $x_{i0}\times_{k_0}k=x_i$, we have $p_X^*([x_{i0}])=[x_i]\in \CH^d(X, X^*)$ and $p_X^*[k(x_{i0})]=[k(x_i)]$ in $F^dK_0(X)$. By Proposition~\ref{prop:ChernClassFunct}(3) (substituting $(k_0, k)$ for $(k, F)$) we have
\[
p_X^*(c_d(\cl^d(\sum_{i=1}^sr_i[x_{i0}]))=c_d(\cl^d(\sum_{i=1}^sr_i[x_i]).
\]
This reduces us to proving the analog of \eqref{eqn:RRIdentity} for $(X_0, X_0^*)$; changing notation, we may assume that $k$ is finitely generated over $\F_p$. \\[5pt]
 {\bf Step 2}. Spreading out.    
   
 As above, we have  the identity \eqref{eqn:RRIdentity} in $\CH^d(X, X^*)$. 
 Represent $\cl^d(\sum_{i=1}^sr_i\cdot [x_i])$ as the $K_0$-class of a perfect complex $E$ on $X$, so 
\[
\cl^d(\sum_{i=1}^sr_i\cdot [x_i])=[E_{\text{ev}}]-[E_{\text{odd}}]
\]
where $E_{\text{ev}}=\oplus_{i\ge0}E_{2i}$ and $E_{\text{odd}}=\oplus_{i\ge0}E_{2i+1}$.  Taking a morphism
\begin{equation}\label{eqn:RepresentingMorphism}
(f_1, f_2):X\to (\P^{N_1}\times\Gr(m_1, M_1))\times(\P^{N_2}\times\Gr(m_2, M_2))
\end{equation}
with $f_1^*E_{m_1, M_1}(-1)\cong E_{\text{ev}}$ and $f_2^*E_{m_2, M_2}(-1)\cong E_{\text{odd}}$, we have
\[
c_d(\cl^d(\sum_{i=1}^sr_i\cdot [x_i]))=c_d(f_1-f_2):=(f_1, f_2)^*(c_d([p_1^*E_{m_1, M_1}(-1)]-[p_2^*E_{m_2, M_2}(-1)])).
\]

Represent 
\[
c_d([p_{1}^*E_{m_1, M_1}(-1)]-[p_{2}^*E_{m_2, M_2}(-1)])\in \CH^d(\P^{N_1}\times\Gr(m_1, M_1)\times\P^{N_2}\times\Gr(m_2, M_2))
\]
by a cycle
\[
C_d\in Z^d(\P^{N_1}\times\Gr(m_1, M_1)\times\P^{N_2}\times\Gr(m_2, M_2)),
\]
so $c_d(\cl^d(\sum_{i=1}^sr_i\cdot [x_i]))$ is represented by the 0-cycle
\[
(f_1, f_2)^*(g\cdot C_d)\in Z^d(X\setminus X^*)
\]
for some $g\in (\GL_{N_1+1}\times\GL_{M_1}\times \GL_{N_2+1}\times\GL_{M_2})(k)$, with $(f_1, f_2)^{-1}(\supp(gC_d))$ a finite set of closed points of $X\setminus X^*$. Similarly, the identity
\[
c_d(\cl^d(\sum_{i=1}^sr_i\cdot [x_i]))-(-1)^{d-1}(d-1)!\sum_{i=1}^sr_i\cdot [x_i])=m\cdot [x_1]\in \CH^d(X, X^*)
\]
is given by an identity of 0-cycles on $X\setminus X^*$, 
\[
\sum_{j=1}^\ell i_{C_j*}(\Div h_j)=m\cdot x_1 -(f_1, f_2)^*(g\cdot C_d)+(-1)^{d-1}(d-1)!\sum_{i=1}^sr_i\cdot x_i,
\] 
where $i_{C_j}:C_j\to X$ are Cartier curves on $X$ relative to $X^*$ with $h_j\in \sO_{C_j:C_j\cap X^*}^\times$.

Let $U_j\subset C_j$ be the maximal open subset of $C_j$ over which $h_j$ is a regular function. Since $h_j\in \sO_{C_j:C_j\cap X^*}^\times$,  $U_j$ contains $C_j\cap X^*$. Thus, letting $T=\cup_{j=1}^\ell i_{C_j}(C_j\setminus U_j)$, $T$ is a finite set of closed points of $X\setminus X^*$. Choose a closed point $r_{ij}$, $i=1,\ldots, q_j$ in each irreducible component of $C_j$, with $i_{C_j}(r_{ij})\in X\setminus X^*$ and let $R=\{i_{C_j}(r_{ij})\mid j=1,\ldots, \ell, i=1\ldots q_j\}$. Let
\[
P=\{x_1,\ldots, x_s\}\cup (f_1, f_2)^{-1}(\supp(gC_d))\cup T\cup R.
\]
Then $P$ is a finite set of closed points of $X\setminus X^*\subset X(d)$.

Since $k$ is finitely generated over $\F_p$, we can apply the spreading out method discussed in \S\ref{sec:Spread}. We spread out $(X(d), X^*\cap X(d), P)$ to $(\sX, \sX^*, \sP)\to \Spec S$ for  a finitely generated $\F_p$-algebra $S$ that is smooth over $\F_p$, with $\sX\to \Spec S$ projective, $k$ the quotient field of $S$ and $(\sX, \sX^*, \sP)\otimes_Sk=(X(d), X^*\cap X(d), P)$.

We then apply Lemma~\ref{lem:RegularRes} to give  a flat projective $S$-scheme $\tilde{\sX}\to \Spec S$ with a projective morphism $q:\tilde{\sX}\to \sX$ such that $q$ is \'etale over a neighborhood of $\sP$ in $\sX$, and with $\tilde{\sX}$ smooth and quasi-projective of dimension $d+\dim_kS$ over $\F_p$. Let $q_k:\tilde{X}\to X(d)$ be the pullback of $q$ over $\Spec k\to \Spec S$. Then $\tilde{X}$ is   projective of pure dimension $d$ over $k$ and is regular,  and  $q_k$ is a projective morphism that is \'etale over a neighborhood of $P$ in $X(d)$. 

In particular, the cycle-theoretic pullbacks $q_k^*((f_1, f_2)^*(g\cdot C_d)), q_k^*(1\cdot x_i)\in Z^d(\tilde{X})$ are  defined. Using Lemma~\ref{lem:PartialFunct} we have
\[
 [q_k^*((f_1, f_2)^*(g\cdot C_d))]-[q_k^*((-1)^{d-1}(d-1)!\sum_{i=1}^sr_ix_i)]=[m\cdot q_k^*(1\cdot x_1)]\in \CH^d(\tilde{X}).
 \]
 Using the functoriality of cycle-theoretic pullback (when defined), we have
 \[
 q_k^*((f_1, f_2)^*(g\cdot C_d))=(f_1q_k, f_2q_k)^*(g\cdot C_d).
 \]
 
By the naturality of the Chern class map $c_d:K_0(-)\to \CH^d(-)$ as expressed in Remark~\ref{rem:LciPullback} and \eqref{eqn:ChernClassLciPullback2},  we have
 \begin{align*}
[(f_1q_k, f_2q_k)^*(g\cdot C_d))]&=c_d([(f_1q_k)^*E_{m_1, M_1}(-1)]-[(f_2q_k)^*E_{m_2, M_2}(-1)]\\
&=c_d([q_k^*E_\text{ev}]-[q_k^*E_{\text{odd}}])\\
&=c_d(\cl^d([q_k^*(\sum_{i=1}^sr_i\cdot x_i)]))\\
&=(-1)^{d-1}(d-1)![q_k^*(\sum_{i=1}^sr_i\cdot x_i)]\in \CH^d(\tilde{X}),
\end{align*}
with this last identity following from Lemma~\ref{lem:GRRReg}. 
Thus 
\begin{align*}
m\cdot [q_k^*(1\cdot x_1)]&= [q_k^*((f_1, f_2)^*(g\cdot C_d))]-[q_k^*((-1)^{d-1}(d-1)!\sum_{i=1}^sr_ix_i]\\
&=[(f_1q_k, f_2q_k)^*(g\cdot C_d)]-(-1)^{d-1}(d-1)![q_k^*(\sum_{i=1}^sr_i\cdot x_i)]\\
&=0\in \CH^d(\tilde{X}).
\end{align*}

Let $\pi:\tilde{X}\to \Spec k$ be the structure morphism. Since $q_k^*(1\cdot x_1)$ is a non-zero, effective zero-cycle on the projective $k$-scheme $\tilde{X}$, we have $\pi_*(q_k^*(1\cdot x_1))\neq0\in \CH_0(\Spec k)=\Z$, and we see as before that $m=0$.
\end{proof}

\begin{corollary}\label{cor:RR} Let $k$ be an infinite field,  let  $(X, X^*)$, $(X, X')$ be admissible pairs over $k$ and let   $d=\dim_kX$. Suppose that $X'\subset X^*$. Then the map
\[
j^{X^*, X'}_*:\CH^d(X, X^*)\to \CH^d(X, X')
\]
induced by the inclusion $X\setminus X^*\subset X\setminus X'$ is surjective with kernel killed by $(d-1)!$. 
\end{corollary}

\begin{proof} The map $j^{X^*, X'}_*$ is surjective by Lemma~\ref{lem:CHSurj}. We have the commutative diagram
\[
\xymatrix{
\CH^d(X, X^*)\ar[r]^{j^{X^*, X'}_*}\ar[dr]_{\cl^d}&\CH^d(X, X')\ar[d]^{\cl^d}\\
&F^dK_0(X)
}
\]
By Theorem~\ref{thm:RR}, the kernel of $\cl^d:\CH^d(X, X^*)\to F^dK_0(X)$ is killed by $(d-1)!$, so the same holds for the kernel of $j^{X*, X'}_*$.
\end{proof}

\begin{remark} We expect that the map $j^{X^*, X'}_*:\CH^d(X, X^*)\to \CH^d(X, X')$ is an isomorphism for all admissible pairs $(X, X^*)$, $(X, X')$ with  $X'\subset X^*$. 
\end{remark}

In case $X$ is equi-dimensional of dimension $d$ over $k$, Theorem~\ref{thm:RR} translates into the corresponding result for $\CH_0(X, X^*)$, using Proposition~\ref{prop:Comparison}  to give the identification $\CH^d(X, X^*)=\CH_0(X, X^*)$. 

If $X$ is not equi-dimensional over $k$, given a closed subset $Y^*$ of $X$  with $X\setminus Y^*$ regular  and dense in $X$. we still have the homomorphism
\[
\cl_0:\CH_0(X, Y^*)\to F_0K_0(X)
\]
with $F_0K_0(X)$ the subgroup of $K_0(X)$ generated by the classes $[k(x)]$ for $x$ a closed point in $X_\reg$ (see \cite[Proposition 2.1]{LW}). In general, there  is no analog of Theorem~\ref{thm:RR} for $\cl_0$, even taking $Y^*=X_\sing$ to be the minimal choice. 

Given $X, Y^*$ as above, let $X^*(\le i)=(Y^*\cap X(i))\cup X(\le i-1)$. Then we have the map \eqref{eqn:CH0CH*Comp}, so the (surjective) maps $\cl^i:\CH^i(X(\le i), X^*(\le i))\to F^iK_0(X(\le i))$ induce the  homomorphism
\[
\prod_{i=0}^{\dim_kX}\cl^i_0:\CH_0(X, Y^*)\to \oplus_{i=0}^{\dim_kX}F^iK_0(X(\le i)).
\]

\begin{proposition}\label{prop:RRCH0} The map  
$\prod_{i=0}^{\dim_kX}\cl^i_0$ is surjective. Moreover, suppose $X$ is reduced and 
\[
\codim_{X(i)}X(i)\cap X(j)\ge 2,\ 0\le i<j\le \dim_kX.
\]
Then the kernel of $\prod_{i=0}^{\dim_kX}\cl^i_0$ is killed by $(\dim_kX-1)!$.
\end{proposition}

\begin{proof} Since $Z_0(X\setminus Y^*)=\oplus_{i=0}^{\dim_kX}Z^i(X(\le i), X^*(\le i))$, the map 
\eqref{eqn:CH0CH*Comp} is surjective. As $\cl^i:\CH^i(X(\le i), X^*(\le i))\to F^iK_0(X(\le i))$ is surjective for each $i$, the surjectivity of $\prod_{i=0}^{\dim_kX}\cl^i_0$ follows.

In case $\codim_{X(i)}X(i)\cap X(j)\ge 2$ for all $0\le i<j\le \dim_kX$, the map \eqref{eqn:CH0CH*Comp} is an isomorphism (Proposition~\ref{prop:CH0CH*}) and then the second statement follows from Theorem~\ref{thm:RR}.
\end{proof}

\begin{remark}The inclusion $\iota_{\le i}:X(\le i)\to X$ induces a pullback map
\[
\iota_{\le i}^*:F_0K_0(X)\to F^iK_0(X(\le i))
\]
and 
\[
\oplus_{i=0}^{\dim_kX}\iota_{\le i}^*:F_0K_0(X)\to \oplus_{i=0}^{\dim_k X}F^iK_0(X(\le i))
\]
 is surjective. However, it is not clear that Proposition~\ref{prop:RRCH0} sheds any light on the exponent of the kernel of $\cl_0:\CH_0(X, Y^*)\to F_0K_0(X)$ in case $X$ is not equi-dimensional, even under the hypotheses of Proposition~\ref{prop:RRCH0}.
 \end{remark} 

\begin{ex}  Consider the case $X=\Spec k[x,y,z]/(z(x,y))$, $Y=\{(0,0,0)\}$ discussed in Example~\ref{ex:CH0}. It is not hard to show that $F_0K_0(X)=0$: if $p\neq (0,0,0)$ is a closed point in the $z$-axis, with defining ideal $(f_p(z))\subset k[z]$, then as $f_p(0)$ is in $k^\times$, we can extend $f_p(z)$ to a regular function $\tilde{f}_p(x,y,z)\in k[x,y,z]/(z(x,y))$ with restriction to the $x$-$y$-plane the constant function with value $f_p(0)$. Then $\tilde{f}_p(x,y,z)$ is a unit on $X\setminus\{p\}$, hence a non-zero divisor in $k[x,y,z]/(z(x,y))$, giving the exact sequence
\[
0\to \sO_X\xrightarrow{\times \tilde{f}_p}\sO_X\to k(p)\to 0,
\]
showing that $[k(p)]=0$ in $K_0(X)$. Similarly, for $q\neq(0,0,0)$ a closed point in the $x$-$y$-plane, we also have $[k(q)]=0$ in $K_0(X)$, so $F_0K_0(X)=0$. 

On the other hand, we have computed
\[
\CH_0(X, Y^*)=Z_0(z\text{-axis}\setminus\{(0,0,0)\})
\]
so $\cl_0:\CH_0(X, Y^*)\to F_0K_0(X)$ has kernel the free abelian group on the closed points of $\A^1_k\setminus\{0\}$, a free abelian group whose rank is the cardinality of $k$ (since we are assuming $k$ is an infinite field). 

In \cite{GK}, Gupta-Krishna use the modification $\CH_0^{\text{BK}}(X, Y^*)$ of $\CH_0(X, Y^*)$ defined by Binda-Krishna in \cite[\S 3.3]{BindaKrishna},  using finite morphisms $f:C\to X$ with $C$ of dimension one and $f$ an lci morphism over a neighborhood of $Y^*\cap f(C)$ in $X$ instead of taking $f$ to be a closed immersion. The same analysis as for $\CH_0(X, Y^*)$ in the above example gives 
\[
\CH_0^{\text{BK}}(X, Y^*)=Z_0(z\text{-axis}\setminus\{(0,0,0)\})=Z_0(\A^1_k\setminus\{0\}),
\]
for $k$ an arbitrary field. This says that the result  \cite[Theorem 8.1(1), (2)]{GK} for $X$ a geometrically reduced affine scheme of finite type over a field $k$ should add the hypothesis that $X$ is equi-dimensional over $k$. 
\end{ex}

\section{Spreading out}\label{sec:Spread} We collect some elementary notions and results on the ``spreading out'' method, to be used in the proof of Theorem~\ref{thm:RR} in case the base-field $k$ has positive characteristic.

Let $Y\to \Spec k$ be a  quasi-projective $k$-scheme, with $k$ a field of characteristic $p>0$. Suppose that $k$ is finitely generated over $\F_p$. Fixing a locally closed immersion $Y\subset\P^N_k$, let $\bar{Y}$ be the closure and $Z:=\bar{Y}\setminus Y$ the complement.  

Let $S_0\subset k$ be a finitely generated $\F_p$-sub-algebra with quotient field $k$, such that $\bar{Y}$ and $Z$ have generators for their respective homogeneous ideals $I_{\bar{Y}}, I_Z$ in $k[X_0,\ldots, X_N]$  contained in $S_0[X_0,\ldots, X_N]$; let $\sI_{\bar{Y}}$, $\sI_Z$ be the  homogeneous ideals in $S_0[X_0,\ldots, X_N]$ with these generators, and let $\bar{\sY}_0, \bar{\sZ}_0\subset \P^N_{S_0}$ be the closed subschemes defined by these ideals. Inverting a non-zero element of $S_0$ if necessary, we may assume that $\sI_{\bar{Y}}\subset \sI_Z$, so $\bar{\sZ}_0$ is a closed subscheme of $\bar{\sY}_0$. 

As $S_0$ is a domain,  finitely generated over $\F_p$, the regular locus in $\Spec S_0$ is open and dense, so after inverting a non-zero element of $S_0$, we may assume that $S_0$ is regular, and since $\F_p$ is perfect, $S_0$ is smooth over $\F_p$. Similarly, $\bar{\sY}_0\to \Spec S_0$, $\bar{\sZ}_0\to \Spec S_0$ are flat over the generic point $\Spec k\hookrightarrow \Spec S_0$, so inverting an additional non-zero element of $S_0$ to form the finitely generated $\F_p$-algebra $S$, we have projective $S$-schemes $\bar{\sY}, \bar{\sZ}$ with $\bar{\sZ}$ a closed subscheme of $\bar{\sY}$,  and the quasi-projective $S$-scheme $\sY:=\bar{\sY}\setminus\bar{\sZ}$,  with all three flat over $S$.  We call $\sY\to \Spec S$ a {\em spreading out} of $Y\to \Spec k$. Clearly, if $Y$ is projective over $k$, then $\sY$ is projective over $\Spec S$. In general, $\sY$ is quasi-projective over $\Spec S$, and since $S$ is finitely generated over $\F_p$, $\sY$ is also quasi-projective over $\F_p$.

Similarly, if we have closed subschemes $Y_1,\ldots, Y_r$ of $Y$, we construct as above a spreading out of $\sY$ together with closed subschemes $\sY_1,\ldots, \sY_r$ that give a spreading out $(\sY, \sY_1,\ldots, \sY_r)$ of $(Y, Y_1,\ldots, Y_r)$. 

We call a spreading out $\sY\to \Spec S$ a {\em smooth spreading out} of $Y$ if $\sY$ is smooth over $\F_p$.\\[5pt]
We note that, if a smooth spreading out of $Y$ exists, then as $\sY$ is a regular scheme, $Y$ is also regular, being a localization of $\sY$. Similarly, $\sY$ is equi-dimensional of dimension $d+\dim_{\F_p}S$ over $\F_p$, if and only if $Y$ is equi-dimensional of dimension $d$ over $k$.

\begin{remark}\label{rem:FultonChernClass}
We have the  Chow groups $\CH_m(Y)$ of algebraic cycles of dimension $m$ (over $k$) modulo rational equivalence. We will use Fulton's Chern class operators  \cite[\S 3.2, Example 3.2.7]{Fulton}
\[
c_n^F(-)\cap -:K_0(Y)\to \Hom_\Z(\CH_m(Y)\to \CH_{m-n}(Y))
\]
Assuming $Y$ is equi-dimensional of dimension $d$ over $k$, we can index the Chow groups by codimension, giving
\[
c_n^F(-)\cap -:K_0(Y)\to \Hom_\Z(\CH^m(Y)\to \CH^{m+n}(Y))
\]
and we have
\[
c_n^F(-)\cap [Y]:K_0(Y)\to \CH^n(Y)
\]
where $[Y]\in \CH^0(Y)$ is the fundamental class: if $Y$ has integral components $Y_1,\ldots, Y_r$, then $[Y]$ is the class of $\sum_{i=1}^r1\cdot Y_i$ in $\CH_{\dim_kY}(Y)$. 

If $Y$ is smooth over $k$, the Chern class map defined following Grothendieck's method \cite{Grothendieck}
\[
c_n(-):K_0(Y)\to \CH^n(Y)
\]
satisfies
\[
c_n(x)=c_n^F(x)\cap[Y]\in \CH^n(Y);
\]
see \cite[\S 3.2, Example 3.2.7]{Fulton}. We will therefore often  write $c_n(x)$ for $c_n^F(x)\cap [Y]$ for $Y$ equi-dimensional over $k$. Given a smooth spreading out $\sY\to \Spec S$ of $Y$, we will use the same notation replacing $Y$ with $\sY$, viewing $\sY$ as a smooth quasi-projective $\F_p$-scheme.
\end{remark}

\begin{lemma}\label{lem:GRRReg} Let $k$ be a finitely generated field extension of $\F_p$, and take $Y\in \Sch_k^{\qp}$. Suppose $Y$ is equi-dimensional over $k$ and regular. Let $Z\subset Y$ be an integral closed subscheme of codimension $m$, giving the corresponding elements $[Z]\in \CH^m(Y)$, $[\sO_Z]\in K_0(Y)$.   Then
\[
c_n^F([\sO_Z])\cap[Y]=\begin{cases} 0 &\text{ for }n<m,\\ (-1)^{m-1}(m-1)![Z]&\text{ for }n=m.\end{cases}
\]
\end{lemma}

\begin{proof} Since $k$ is finitely generated over $\F_p$, there is  a  spreading out $(\sY, \sZ)\to \Spec S$ of $(Y, Z)$. Since $\sY$ is of finite type over $\F_p$, the regular locus in $\sY$ is open, and by assumption contains $Y$. Thus, after inverting a non-zero element of  $S$ if necessary, we may assume the $\sY$ is regular. Since $\F_p$ is perfect, this implies that $\sY$ is smooth over $\F_p$, hence $\sY\to \Spec S$ is a smooth spreading out of $Y$. 

We may assume that  $\sZ\subset \sY$ is an integral codimension $m$ closed subscheme of $\sY$, again after localizing $S$ further if necessary. Let $j:Y\to \sY$ be the inclusion, a flat morphism. Then as $c_n^F(-)\cap-$ is functorial for flat morphisms \cite[Theorem 3.2]{Fulton}, we have
\[
j^*(c_n^F([\sO_\sZ])\cap[\sY])=c_n^F([\sO_Z])\cap[Y].
\]
Since $\sY$ is smooth over $\F_p$, we have
\[
c_n^F([\sO_\sZ])\cap[\sY]=c_n([\sO_\sZ])\in \CH^n(\sY),
\]
and by Grothendieck-Riemann-Roch for the smooth quasi-projective $\F_p$-scheme $\sY$ \cite[Example 15.3.1]{Fulton}, we have
\[
c_n([\sO_\sZ])=\begin{cases} 0 &\text{ for }n<m,\\ (-1)^{m-1}(m-1)![\sZ]&\text{ for }n=m,\end{cases}
\]
where $[\sZ]\in \CH^m(\sY)$ is the class of the integral closed subscheme $\sZ$. But since $j^*([\sZ])=[Z]$, the lemma is a consequence of these three identities.
\end{proof}

Now let $(X, X^*)$ be an admissible pair over $k$ with a finite set of closed points $P\subset X\setminus X^*$, $P=\{p_0,\ldots, p_r\}$. Suppose $k$ is finitely generated over $\F_p$. Then there exists a spreading out $(\sX, \sX^*, \sP)\to \Spec S$ of $(X, X^*, P)$, $\sP=\{\mathfrak{p}_0,\ldots, \mathfrak{p}_r\}$.  

\begin{lemma}\label{lem:RegularRes} There exists a smooth quasi-projective $\F_p$-scheme $\sY\to \Spec \F_p$ and a morphism $q:\sY\to \sX$ over $\F_p$ such that\\[5pt]
1. $q$ is proper and generically finite\\[2pt]
2. $q$ is \'etale over a neighborhood of $P\subset \sP$\\[2pt]
3. After localizing $S$ is necessary, $\sY\to \Spec S$ is flat\\[5pt]
Moreover, if $X$ is projective over $k$, we may take  $\sY$ to be projective over   $\Spec S$ and over $\sX$.
\end{lemma}

\begin{proof} Since $\sX\setminus \sX^*$ is of finite type over $\F_p$, the regular locus of  $\sX\setminus \sX^*$ is an open subscheme, containing the regular scheme $X\setminus X^*$. Thus,  after localizing $S$ if necessary, we may assume that $\sX\setminus \sX^*$ is regular. Since $\F_p$ is perfect, this implies that $\sX\setminus \sX^*$ is smooth over $\F_p$. Similarly, we may assume that $\sP\subset \sX\setminus \sX^*$. For each $i=1,\ldots, r$ choose a closed point $x_i\in \mathfrak{p}_i$, so $x_i$ is in the smooth locus of $\sX$ over $\F_p$.

By \cite[Theorem 1.1, Theorem 2.6]{BhattSnowden}, there is a smooth, finite-type $\F_p$-scheme $\sY$ and a proper morphism $q:\sY\to \sX$ that is generically finite and is \'etale over a neighborhood of $\{x_0,\ldots, x_r\}$. The points (1), (2), (3) follow from this.

  In case $X\to \Spec k$ is projective, $\sX\to \Spec S$ is projective by construction, say $\sX$ is a closed subscheme of $\P^N_S$. Let $\overline{\Spec S}\subset \P^M_{\F_p}$ be a projective closure of the 
 finite-type affine $\F_p$-scheme $\Spec S$, and let 
 \[
 \overline{\sX}\subset \P^N_{\P^M_{\F_p}}=\P^N_{\F_p}\times_{\F_p}\P^M_{\F_p}\subset \P^{NM-1}_{\F_p}
 \]
 be the closure of $\sX$. Applying the Bhatt-Snowden construction as above to $\overline{\sX}$ gives us $\overline{\sY}\to \overline{\sX}$, and  since $\overline{\sX}$ is proper over $\F_p$, 
  \cite[Theorem 2.6]{BhattSnowden} tells us that we can take $\overline{\sY}$ to be projective over $\F_p$. But then $\overline{\sY}\to \overline{\sX}$ is a projective morphism, so pulling back over $\Spec S\hookrightarrow \P^M_{\F_p}$ gives $\sY\to \sX\to \Spec S$ with $\sY\to \Spec S$ and $\sY\to \sX$ projective. 

\end{proof}

\begin{remark}\label{rem:LciPullback} Let $Y$ be quasi-projective  and  equi-dimensional over $k$, and let $f:Y\to H$ be a morphism to a smooth quasi-projective $k$-scheme $H$. Take $x\in K_0(H)$. Since $H$ is smooth over $k$, the morphism $f$ is an lci morphism, so we have a well-defined lci pullback map \cite[\S 6.6]{Fulton}
\[
f^!:\CH_n(H)\to \CH_{n-c}(Y)
\]
where $c=\dim_kH-\dim_kY$; if $Y$ is regular, we write this as
\[
f^*:\CH^*(H)\to \CH^*(Y).
\]

For $\alpha\in \CH_r(H)$,  $m>0$ an integer and $x\in K_0(H)$, we have
\begin{equation}\label{eqn:BivChernClass}
f^!(c_m^F(x)\cap \alpha)=c_m^F(f^*(x))\cap f^!(\alpha).
\end{equation}
Indeed, we have the bivariant class in $\CH^*(H=H)$ by  $(Z\xrightarrow{g}H, y\in \CH_r(Z))\mapsto c_m^F(g^*x)\cap y$ (see \cite[\S17.3, pg. 325]{Fulton}), and  $f^!$ defines a bivariant class in $\CH^*(Y\to H)$, via $(g:Z\to H, y\in \CH_r(Z))\mapsto f^!(y)\in \CH_{r-c}(Z\times_HY)$. Then \eqref{eqn:BivChernClass} follows from \cite[Proposition 17.3.2]{Fulton} applied to $c=f^!$. In particular, taking $\alpha=[H]$ and assuming $Y$ is regular, we have
\begin{equation}\label{eqn:ChernClassLciPullback}
f^*(c_m(x))=c_m(f^*(x))\in \CH^m(Y)
\end{equation}

If now $c_m(x)$ is represented by some $z\in Z^m(H)$ such that $f^{-1}(\supp(z))$ has pure codimension $m$ in $Y$, then the cycle-theoretic pullback $f^*(z)\in Z^m(Y)$ is defined and we have
\begin{equation}\label{eqn:ChernClassLciPullback2}
c_m(f^*(x))=[f^*(z)]\in \CH^m(Y).
\end{equation} 
\end{remark}

\end{document}